\documentclass[11pt]{amsart}

\usepackage{amsmath,amsthm,amssymb}

\usepackage{epic,eepic}
\usepackage[all]{xy}
\usepackage{tikz}
\usetikzlibrary{arrows.meta}
\usepackage{array}
\usepackage{fancybox}

\allowdisplaybreaks

\newtheorem{thm}{Theorem}[section]
\newtheorem{prop}[thm]{Proposition}

\newtheorem{lem}[thm]{Lemma}

\theoremstyle{definition}

\numberwithin{equation}{section}
\newtheorem{rem}[thm]{\bf Remark}
\newtheorem{ex}[thm]{\bf Example}

\def\bbR{\mathbb{R}}
\def\bbT{\mathbb{T}}

\def\bbZ{\mathbb{Z}}

\def\bfSigma{\mathbf{\Sigma}}
\def\bfa{\mathbf{a}}
\def\bfb{\mathbf{b}}
\def\bfc{\mathbf{c}}

\def\bfe{\mathbf{e}}
\def\bfg{\mathbf{g}}

\def\bfv{\mathbf{v}}

\def\bfzero{\mathbf{0}}

\def\rmsgn{\mathrm{sgn}}

\newcommand{\overunder}[2]{
\!\!\begin{array}{c}
\scriptstyle{#1}\\[-9pt]
-\!\!\!-\\[-.1in]
\scriptstyle{#2}
\end{array}
}
\def\cp{\textcircled{$+$}}
\def\cm{\textcircled{$-$}}
\def\bp{\boxed{\!+\!}}
\def\bm{\boxed{\!-\!}}

\begin{document}
\bibliographystyle{amsalpha}

\title[Local and global patterns of $G$-fans]{
Local and global patterns  of rank 3 $G$-fans\\ of totally-infinite type}
\address{\noindent Graduate School of Mathematics, Nagoya University, 
Chikusa-ku, Nagoya,
464-8604, Japan}
\email{nakanisi@math.nagoya-u.ac.jp}
\author{Tomoki Nakanishi}
\maketitle

\begin{abstract}
We focus on the $G$-fans associated with cluster patterns whose initial exchange matrices are of infinite type.
We study the asymptotic behavior of the $g$-vectors around the initial $G$-cone under the alternating mutations for two indices of infinite type.
In the rank 3 case,
we classify them into several patterns.
As an application, the incompleteness of the $G$-fans of infinite type is proved.
We observed that 
the local pattern of a rank 3 $G$-fan of totally-infinite type classified by the above types
correlates with its global pattern.
Following the classification of the local patterns  (together with the Markov constant),
we present several prototypical examples of the global patterns of the rank 3 $G$-fans of totally-infinite type,
many of which are new in the literature.
\end{abstract}

\section{Introduction}
\label{sec:intro1}

Let $\bbT_n$ be the $n$-regular tree whose edges are labeled by $1$, \dots, $n$,
and fix an initial vertex  $t_0\in \bbT_n$.
Let $\bfSigma(B)$ be a cluster pattern of rank $n$ with the initial exchange matrix $B=B_{t_0}$ \cite{Fomin02,Fomin07},
where $B$ is an arbitrary skew-symmetrizable (integer) matrix.
For each seed $\Sigma_t$ ($t\in \bbT_n$) of $\bfSigma(B)$, one can associate an $n\times n$ integer matrix
$G_t$ called a \emph{$G$-matrix} \cite{Fomin07,Nakanishi11a}.
Each column vector $\bfg_{i;t}$ ($i=1$, \dots, $n$) of $G_t$ is called a \emph{$g$-vector}.
Each $G$-matrix is unimodular \cite{Nakanishi11a} thanks to the column sign-coherence of the $C$-matrices proved by \cite{Gross14}.
Thus, we have a full-dimensional cone $\sigma(G_t)$ in $\bbR^n$ spanned by the $g$-vectors of $G_t$,
which we call a \emph{$G$-cone}.
Collecting all  $G$-cones and their faces, we have a fan $\Delta(B)$.
This fan was initially studied in \cite{Reading11, Reading12}
and called the \emph{$g$-vector fan}.
Here, we alternatively call it the \emph{$G$-fan} to emphasize the construction by $G$-matrices.
The $G$-fan is a geometric realization of the \emph{cluster complex} of $\bfSigma(B)$ \cite{Fomin03a,Gross14,Reading17};
thus, it contains essential information of $\bfSigma(B)$.
Also, it plays a central role in the scattering diagram method to study cluster patterns and cluster algebras
\cite{Gross14, Muller15, Cheung15,Reading15, Reading17,Yurikusa19,Nakanishi22a,Reading22}.

In this paper, we focus on the $G$-fans whose initial exchange matrices are of infinite type,
which have been studied less so far except for the rank 2 and the affine case, and some limited examples of rank 3.
The content of the paper is divided into two parts.

In the first part,
we study the asymptotic behavior of the $g$-vectors around the initial $G$-cone under the alternating mutations for two indices
$i \neq j$ such that the initial exchange matrix $B$ satisfies $|b_{ij}b_{ji}|\geq 4$. 
It is easy to see that the problem is reduced to the rank 3 case.
Then, using the results of \cite{Gekhtman19},
we classify them into six patterns (Types 1, 2, 3, 4-1, 4-2, and 4-3).
As an immediate application, the incompleteness of the $G$-fans of infinite type is proved.

In the second part, we give some experimental results.
We say that a skew-symmetrizable matrix $B$ is of  \emph{totally-infinite type}
if $|b_{ij}b_{ji}|\geq 4$ for any pair $i\neq j$.
We also say that a $G$-fan $\Delta(B)$ is of totally-infinite type
if $B$ is of totally-infinite type.
For such $\Delta(B)$ of rank 3, we assign the triplet of the types for
the pairs of indices $(i,j)$=$(2,3)$,  $(3,1)$, $(1,2)$ in the first part.
This data certainly controls the local behavior (the \emph{local pattern}) of $\Delta(B)$ around the rays $\bbR_{\geq 0} \bfe_i$ ($i=1$, 2, 3).
On the other hand,
looking at the pictures of these $G$-fans, we observed that 
the local pattern of a $G$-fan correlates with its global behavior (the \emph{global pattern}).
To support this observation,
following the classification of the local patterns  (together with additional data called the \emph{Markov constant} \cite{Beineke06,Akagi24}),
we present several prototypical examples of the global patterns of the rank 3 $G$-fans of totally-infinite type,
many of which are new in the literature.
See Table \ref{tab:corres1} for the summary.
Our naive hope is that these examples may exhaust all global patterns of the rank 3 $G$-fans of totally-infinite type with possible minor variants.

The content of the paper is as follows.
In Section 2, we recollect the results on the $c$- and $g$-vectors of rank 2 cluster patterns.
In Section 3, we study 
the asymptotic behavior of $g$-vectors
around the initial $G$-cone under the alternating mutations for two indices of infinite type,
and classify them into six patterns.
In Section 4, we prove the incompleteness of the $G$-fans of infinite type  (Theorem \ref{thm:incomplete1}).
In Section 5, we visualize the above patterns of $g$-vectors by the pictures of rank 3 $G$-fans.
In Section 6, we exhibit
several examples of global patterns of the rank 3 $G$-fans of totally-infinite type
based on the observation explained as above.

This work is supported in part by JSPS grant No. JP22H01114.

\section{Review of rank 2 results}

We recall the mutation formulas for the exchange matrices, $c$-vectors, and $g$-vectors \cite{Fomin07,Nakanishi11a}:
\begin{align}
\label{eq:bmut1}
b_{ij;t'} & = 
\begin{cases}
-b_{ij;t}
&
\text{$i=k$ or $j=k$},
\\
b_{ij;t} + b_{ik;t}[b_{kj;t}]_+ + [-b_{ik;t}]_+ b_{kj;t},
&
\text{otherwise},
\end{cases}
\\
\label{eq:cmut1}
\bfc_{i;t'}
& =
\begin{cases}
-\bfc_{k;t} 
& i=k,
\\
\bfc_{i;t}
+ [\varepsilon_{k;t} b_{ki;t}]_+ \bfc_{k;t}
& i \neq k,
\end{cases}
\\
\label{eq:gmut1}
\bfg_{i;t'}
& =
\begin{cases}
-\bfg_{k;t} 
+\sum_{j=1}^n [-\varepsilon_{k;t} b_{jk;t}]_+ \bfg_{j;t}
& i=k,
\\
\bfg_{i;t}
& i \neq k,
\end{cases}
\end{align}
where $t$ and $t'$ are $k$-adjacent vertices in $\bbT_n$,
$[a]_+=\max(a,0)$, and $\varepsilon_{k;t}$ is the sign for the $c$-vector $\bfc_{k;t}$  (the tropical sign) .

Here we recollect the results on the $c$- and $g$-vectors of rank 2 cluster patterns
obtained by \cite{Reading12}. Here, we present them in the convention of \cite{Gekhtman19}.
See also \cite{Nakanishi22a} for the background on the $c$- and $g$-vectors.

\subsection{Chebyshev polynomials}
The \emph{Chebyshev polynomials $U_n(t)$ $(n\in \bbZ_{\geq -2})$ of the second kind}
are defined by the following recursion relation:
\begin{gather}
U_{-2}(t)= -1,
\quad
U_{-1}(t)= 0,
\\
\label{eq:Urel1}
U_{n}(t) = 2t U_{n-1}(t) - U_{n-2}(t)
\quad
(n \geq 0).
\end{gather}
For example, we have
\begin{align}
U_0 (t) =1,
\quad
U_1(t) = 2t,
\quad
U_2(t) = 4t^2-1,
\quad
U_3(t) = 8t^3 - 4t.
\end{align}
The following properties are well-known.
\begin{gather}
\label{eq:Up1}
U_n(t) >0
\quad (n \geq 0,\, t\geq1),
\\
\label{eq:Up2}
U_n(t)^2 - U_{n-1}(t)U_{n+1}(t)=1
\quad (n \geq -1),
\\
\label{eq:Up3}
\frac{U_n(t)}{U_{n-1}(t)} 
> 
\frac{U_{n+1}(t)}{U_{n}(t)}
\quad (n\geq 1),
\\
\label{eq:Up4}
U_n(t) = \frac{(t+\sqrt{t^2-1})^{n+1} - (t-\sqrt{t^2-1})^{n+1}}{2\sqrt{t^2-1}}
\quad (n \geq -2),
\\
\label{eq:Up5}
\lim_{n\rightarrow \infty}
\frac{U_{n+1}(t)}{U_{n}(t)}= t+\sqrt{t^2-1}
\quad
(|t| \geq 1).
\end{gather}

\subsection{$C$- and $G$-matrices of rank 2 cluster patterns}

We label the vertices of the $2$-regular tree $\bbT_2$ by integers in the following way:
\begin{align}
\label{eq:T2}
\cdots
\
\overunder{2}
\
-2
\
\overunder{1}
\
-1
\
\overunder{2}
\
0
\
\overunder{1}
\
1
\
\overunder{2}
\
2
\
\overunder{1}
\
\cdots
\end{align}
We take the vertex 0 as the initial vertex $t_0$.
Starting from the initial vertex 0, we call the sequence of mutations $1$, $2$, $1$, $2$, \dots,
toward the right direction the \emph{forward mutations}.
The opposite sequence of mutations $2$, $1$, $2$, $1$, \dots,
toward the left direction 
is called the \emph{backward mutations}.
We exclusively consider the initial exchange matrix of \emph{infinite type}
\begin{align}
\label{eq:B1}
B_0=
\begin{pmatrix}
0 & -b
\\
a & 0
\end{pmatrix}
\quad
(a,\, b \in \bbZ_{>0}, \ ab\geq 4).
\end{align}
The case $ab=4$ is of affine type, and the case $ab\geq 5$ is of non-affine (infinite) type.
Let $C_t$ and $G_t$ be the $C$- and $G$-matrices at each vertex $t\in \bbT_2$.
Let $\bfc_{i;t}$ and $\bfg_{i;t}$ ($i=1,\, 2$) be the $i$th 
column vectors of $C_t$ and $G_t$ ($c$- and $g$-vectors).
They are explicitly expressed by the Chevychev polynomials as follows.
We set 
\begin{align}
\kappa = \sqrt{ab},
\quad
\nu = \sqrt{b/a}.
\end{align}
In particular,
\begin{align}
\nu \kappa = b,
\quad
\nu^{-1} \kappa = a.
\end{align}
In this paper, we only use the value of $U_n(t)$ at $t=\kappa/2\geq 1$.
So, we simply write $U_n:= U_n(\kappa/2)>0$.
Then, we have
\begin{gather}
U_n = \kappa U_{n-1} - U_{n-2}
\quad
(n \geq 0),
\\
U_0 =1,
\quad
U_1 = \kappa,
\quad
U_2 = \kappa^2-1,
\quad
U_3 = \kappa^3 - 2\kappa.
\end{gather}
The $C$-matrices are given by the following formulas \cite{Gekhtman19}:
\begin{gather}
C_0=\begin{pmatrix}
1 & 0
\\
0 & 1
\end{pmatrix}
,
\quad
C_1=\begin{pmatrix}
-1 & 0
\\
0 & 1
\end{pmatrix}
,
\quad
C_2=\begin{pmatrix}
-1 & 0
\\
0 & -1
\end{pmatrix}
,
\\
C_{2n+1}=\begin{pmatrix}
U_{2n-2} & -\nu U_{2n-1}
\\
\nu^{-1} U_{2n-3} & -U_{2n-2}
\end{pmatrix}
\quad
(n\geq 1)
,
\quad
C_{2n}=\begin{pmatrix}
-U_{2n-2} & \nu U_{2n-3}
\\
- \nu^{-1} U_{2n-3} & U_{2n-4}
\end{pmatrix}
\quad
(n\geq 1),
\\
C_{-2n-1}=\begin{pmatrix}
U_{2n} & -\nu U_{2n-1}
\\
 \nu^{-1} U_{2n+1} & -U_{2n}
\end{pmatrix}
\quad
(n\geq 0)
,
\quad
C_{-2n}=\begin{pmatrix}
-U_{2n-2} & \nu U_{2n-1}
\\
- \nu^{-1} U_{2n-1} & U_{2n}
\end{pmatrix}
\quad
(n\geq 0).
\end{gather}
We have $|C_t|= (-1)^t$ by \eqref{eq:Up2}.
We see that each $c$-vector $\bfc_{i;t}$ has the definite sign $\varepsilon_{i;t}$ (the \emph{sign-coherence} of $c$-vectors).
We call it the \emph{tropical sign} at $(i;t)$.
By \eqref{eq:Up1}, we read the tropical signs in the above formula as follows:
\begin{align}
\label{eq:ts1}
\begin{tabular}{wc{10pt}|wc{15pt}wc{10pt}wc{10pt}wc{10pt}wc{10pt}|wc{10pt}|wc{10pt}wc{10pt}wc{10pt}wc{10pt}wc{10pt}wc{10pt}wc{15pt}}
$t$ &  & $-4$ & $-3$& $-2$ & $-1$ & $0$ & $1$ & $2$ & $3$ & $4$& $5$ & $6$
\\
\hline
$\varepsilon_{1;t}$ & $\cdots$ & $-$ & \bp & $-$ & \bp & \cp & $-$ & \cm & $+$ & \cm & $+$ & \cm & $\cdots$
\\
$\varepsilon_{2;t}$ & $\cdots$ & \bp & $-$ & \bp & $-$ & \bp & \cp & $-$ & \cm & $+$& \cm & $+$ & $\cdots$
\end{tabular}
\end{align}
Observe that the pattern is irregular at $t=0$, $1$, $2$.
The circled and boxed signs are relevant to the forward and backward
mutations.

One can convert the above formulas to the ones for $G$-matrices by using the duality \cite{Nakanishi11a}
\begin{align}
\label{eq:dual1}
G_t=D^{-1}(C_t^T)^{-1}D,
\end{align}
where the skew-symmetrizer $D$ of $B_0$ is given by $D=\mathrm{diag}(a,b)$. 
Then, we have the formulas
\begin{gather}
G_0=\begin{pmatrix}
1 & 0
\\
0 & 1
\end{pmatrix}
,
\quad
G_1=\begin{pmatrix}
-1 & 0
\\
0 & 1
\end{pmatrix}
,
\quad
G_2=\begin{pmatrix}
-1 & 0
\\
0 & -1
\end{pmatrix}
,
\\
\label{eq:GU1}
G_{2n+1}=\begin{pmatrix}
U_{2n-2} & \nu U_{2n-3}
\\
- \nu^{-1} U_{2n-1} & -U_{2n-2}
\end{pmatrix}
\quad
(n\geq 1)
,
\quad
G_{2n}=\begin{pmatrix}
U_{2n-4} & \nu U_{2n-3}
\\
- \nu^{-1} U_{2n-3} & - U_{2n-2}
\end{pmatrix}
\quad
(n\geq 1),
\\
\label{eq:GU2}
G_{-2n-1}=\begin{pmatrix}
U_{2n} & \nu U_{2n+1}
\\
-  \nu^{-1} U_{2n-1} & -U_{2n}
\end{pmatrix}
\quad
(n\geq 0)
,
\quad
G_{-2n}=\begin{pmatrix}
U_{2n} & \nu U_{2n-1}
\\
- \nu^{-1} U_{2n-1} & - U_{2n-2}
\end{pmatrix}
\quad
(n\geq 0).
\end{gather}

\subsection{Sequences of $g$-vectors}

Under the forward mutations, we have the following sequence of $g$-vectors
($m\geq 1$):
\begin{align}
\label{eq:gseq1}
\bfg_m=
\begin{pmatrix}
\alpha_m
\\
\beta_m
\end{pmatrix}
:=
\begin{cases}
\bfg_{1;m}
=
\begin{pmatrix}
U_{m-3}
\\
-\nu^{-1}U_{m-2}
\end{pmatrix}
& \text{$m$: odd},
\\
\bfg_{2;m}
=
\begin{pmatrix}
\nu U_{m-3}
\\
-U_{m-2}
\end{pmatrix}
& \text{$m$: even}.
\end{cases}
\end{align}
In \eqref{eq:ts1}, the circled signs are stabilized to $-$
for $t\geq  2$.
Thus, by \eqref{eq:gmut1}, the above $g$-vectors obey the recursion with the initial condition
\begin{gather}
\bfg_1=
\begin{pmatrix}
-1
\\
0
\end{pmatrix}
,
\quad
\bfg_2=
\begin{pmatrix}
0
\\
-1
\end{pmatrix},
\\
\label{eq:gmut2}
\bfg_{m+2}
=
\begin{cases}
-\bfg_m + a \bfg_{m+1}
&
\text{$m$: odd},
\\
-\bfg_m + b \bfg_{m+1}
&
\text{$m$: even}.
\end{cases}
\end{gather}
Both components $\alpha_m$ and $\beta_m$ diverge in the limit $m\rightarrow \infty$.
However, by \eqref{eq:Up3} and \eqref{eq:Up5},
the slope $\beta_m/\alpha_m$ monotonically converges to the following value:
\begin{align}
\label{eq:lim1}
\lim_{m\rightarrow \infty}
 \frac{\beta_m}{\alpha_m}
=\lim_{m\rightarrow \infty}
 -\nu^{-1}\frac{U_{m-2}}{U_{m-3}}
 =
 -\frac{ab+\sqrt{ab(ab-4)}}{2b}
=
 -\frac{2a}{ab-\sqrt{ab(ab-4)}}.
\end{align}

Similarly,
under the backward mutations,
we have the following sequence of $g$-vectors
($m\geq 1$):
\begin{align}
\label{eq:gseq2}
\bfg'_m=
\begin{pmatrix}
\alpha'_m
\\
\beta'_m
\end{pmatrix}
=:
\begin{cases}
\bfg_{2;-m}
=
\begin{pmatrix}
\nu U_{m}
\\
-U_{m-1}
\end{pmatrix}
& \text{$m$: odd},
\\
\bfg_{1;-m}
=
\begin{pmatrix}
U_{m}
\\
-\nu^{-1} U_{m-1}
\end{pmatrix}
& \text{$m$: even}.
\end{cases}
\end{align}
In \eqref{eq:ts1}, the boxed signs are stabilized to $+$
for $t\leq 0$.
Thus, by \eqref{eq:gmut1}, the above $g$-vectors obey the recursion with the initial condition
\begin{gather}
\bfg'_1=
\begin{pmatrix}
b
\\
-1
\end{pmatrix}
,
\quad
\bfg'_2=
\begin{pmatrix}
ab-1
\\
-a
\end{pmatrix},
\\
\label{eq:gmut3}
\bfg'_{m+2}
=
\begin{cases}
-\bfg'_m + b \bfg'_{m+1}
&
\text{$m$: odd},
\\
-\bfg'_m + a \bfg'_{m+1}
&
\text{$m$: even}.
\end{cases}
\end{gather}
Again, the slope $\beta'_m/\alpha'_m$ monotonically converges to the following value:
\begin{align}
\label{eq:lim2}
\lim_{m\rightarrow \infty}
 \frac{\beta'_m}{\alpha'_m}
=\lim_{m\rightarrow \infty}
 -\nu^{-1}\frac{U_{m-1}}{U_{m}}
=
 -\frac{2a}{ab+\sqrt{ab(ab-4)}}
 =
 -\frac{ab-\sqrt{ab(ab-4)}}{2b}.
\end{align}

Let 
\begin{align}
\label{eq:v1}
\bfv=
\begin{pmatrix}
1
\\
-(ab+ \sqrt{ab(ab-4)})/2b
\end{pmatrix}
,
\quad
\bfv'=
\begin{pmatrix}
1
\\
-(ab- \sqrt{ab(ab-4)})/2b
\end{pmatrix}
.
\end{align}
Then, 
\begin{align}
\lim_{m \rightarrow \infty} \alpha_m^{-1} \bfg_m = \bfv,
\quad
\lim_{m \rightarrow \infty} \alpha'_m{}^{-1} \bfg'_m = \bfv'.
\end{align}
If $ab=4$, we have $\bfv=\bfv'$.

\begin{ex}
We consider the case $(a,b)=(3,2)$ as the running example throughout the paper.
We have $\kappa = 6$, $\nu=\sqrt{2/3}$, $\nu^{-1}=\sqrt{3/2}$,
and 
\begin{align}
U_0 =1,
\quad
U_1 =\sqrt{6},
\quad
U_2 = 5,
\quad
U_3 = 4\sqrt{6},
\quad
U_4  =19,
\quad
U_5 = 15\sqrt{6}.
\end{align}
Then, we have
\begin{gather*}
\bfg_3=
\begin{pmatrix}
1
\\
-3
\end{pmatrix}
,
\quad
\bfg_4=
\begin{pmatrix}
2
\\
-5
\end{pmatrix}
,
\quad
\bfg_5=
\begin{pmatrix}
5
\\
-12
\end{pmatrix}
,
\quad
\bfg_6=
\begin{pmatrix}
8
\\
-19
\end{pmatrix}
,
\quad
\bfg_7=
\begin{pmatrix}
19
\\
-45
\end{pmatrix}
,
\\
\bfg'_1=
\begin{pmatrix}
2
\\
-1
\end{pmatrix}
,
\quad
\bfg'_2=
\begin{pmatrix}
5
\\
-3
\end{pmatrix}
,
\quad
\bfg'_3=
\begin{pmatrix}
8
\\
-5
\end{pmatrix}
,
\quad
\bfg'_4=
\begin{pmatrix}
19
\\
-12
\end{pmatrix}
,
\quad
\bfg'_5=
\begin{pmatrix}
30
\\
-19
\end{pmatrix}
.
\end{gather*}
\end{ex}

\subsection{Incompleteness of $G$-fan}
For $\bfa$, $\bfb\in \bbZ^2$, let $\sigma(\bfa,\bfb)=\bbR_{\geq 0} \bfa + \bbR_{\geq 0} \bfb$ be the cone in $\bbR^2$ spanned by $\bfa$ and $\bfb$.
The \emph{$G$-fan} $\Delta(B_0)$ consists of the
cones $\sigma(\bfe_1,\bfe_2)$, $\sigma(\bfg_1,\bfe_2)$, $\sigma(\bfe_1,\bfg'_1)$,
 $\sigma(\bfg_m, \bfg_{m+1})$,  $\sigma(\bfg'_m, \bfg'_{m+1})$  ($m\geq 1$) and their faces.
It is depicted in Figure \ref{fig:Gfan1}.
For the affine type ($ab=4$),
the one-dimensional region
$\sigma(\bfv)\setminus\{\bfzero\}$
(``the crack'') is not covered by the $G$-fan $\Delta(B)$.
For the non-affine type ($ab\geq 5$),
the two-dimensional region 
$\sigma(\bfv,\bfv')\setminus\{\bfzero\}$
(the Badlands)
is not covered by the $G$-fan $\Delta(B)$.
So, $\Delta(B)$ is incomplete in both cases.
(Recall that a fan in $\bbR^n$ is \emph{complete} if its support covers the entire region of  $\bbR^n$.)

\begin{figure}
\centering
\begin{tikzpicture}[scale=1.7]
\draw(0,0)--(1,0);
\draw(0,0)--(0,1);
\draw(0,0)--(-1,0);
\draw(0,0)--(0,-1);
\draw(0,0)--(0.25,-1);
\draw(0,0)--(0.33,-1);
\draw(0,0)--(0.375,-1);
\draw(0,0)--(0.4,-1);
\draw(0,0)--(0.625,-1);
\draw(0,0)--(0.66,-1);
\draw(0,0)--(0.75,-1);
\draw(0,0)--(1,-1);
\draw [dotted, thick] (0,0)--(0.5,-1);
 \node at (1.2,0){\small $\bfe_1$};
 \node at (0,1.2){\small $\bfe_2$};
 \node at (-1.2,0){\small $\bfg_1$};
 \node at (0,-1.2){\small $\bfg_2$};
 \node at (0.3,-1.2){\small $\bfg_3$};
 \node at (1.07,-1.17){\small $\bfg'_1$};
 \node at (0.8,-1.17){\small $\bfg'_2$};
 \node at (0.52,-1.2){\small $\bfv$};
 \node at (0,-1.7){(a) affine type};
 \node at (0,-2){$(a,b)=(4,1)$};
 \end{tikzpicture}
\hskip50pt
\begin{tikzpicture}[scale=1.7]
\filldraw [fill=black!10, draw=white] (0,0) -- (0.276,-1) -- (0.723,-1) -- (0,0);
\draw(0,0)--(1,0);
\draw(0,0)--(0,1);
\draw(0,0)--(-1,0);
\draw(0,0)--(0,-1);
\draw(0,0)--(0.25,-1);
\draw(0,0)--(0.266,-1);
\draw(0,0)--(0.272,-1);
\draw(0,0)--(0.276,-1);
\draw(0,0)--(0.723,-1);
\draw(0,0)--(0.727,-1);
\draw(0,0)--(0.733,-1);
\draw(0,0)--(0.75,-1);
\draw(0,0)--(1,-1);
 \node at (0.28,-1.2){\small $\bfv$};
 \node at (0.75,-1.17){\small $\bfv'$};
 \node at (0,-1.7){(b) non-affine type};
  \node at (0,-2){$(a,b)=(5,1)$};
 \end{tikzpicture}
 \vskip-10pt
\caption{Examples of $G$-fans of infinite type.
In the case (b), the slopes of the $g$-vectors $\bfg_i$, $\bfg'_i$ ($i\geq 2$)
are so close to the ones of $\bfv$ and $\bfv'$, respectively,
so that it hard to depict them.
}
\label{fig:Gfan1}
\end{figure}
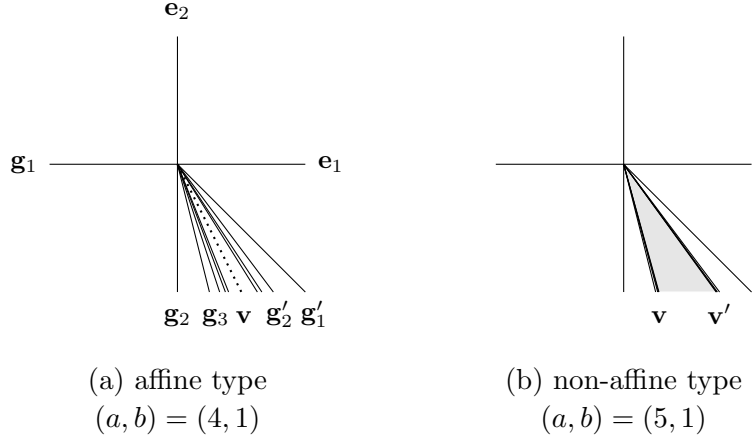

\section{Asymptotic behavior of $g$-vectors under alternating mutations}
\label{sec:lemv1}

Let us consider a cluster pattern $\bfSigma(\tilde B)$ of rank $n$ with
the initial exchange matrix $\tilde B$.
Suppose that  $\bfSigma(\tilde B)$ is of infinite type.
Then, by \cite[Thm.~1.8]{Fomin03a},
there is a matrix $\tilde B'$ such that
$\tilde B'$ is mutation-equivalent to $\tilde B$, and 
$|\tilde b'_{ij}\tilde b'_{ji}|\geq 4$ for some pair $i\neq j$.
The $G$-fans $\Delta(B)$ and $\Delta(B')$ are related by a piecewise-linear
transformation (e.g., \cite[II.Prop.~2.25]{Nakanishi22a}).
So, without loss of generality,
we may choose $\tilde B'$ as the initial exchange matrix.
By changing the labels if necessary,
we may assume  that the initial exchange matrix has the form
\begin{align}
\label{eq:tB0}
\tilde B_0=
\begin{pmatrix}
B_0 & *
\\
* & *
\end{pmatrix}
,
\end{align}
where $B_0$ is the rank 2 matrix of \eqref{eq:B1}.
By identifying the initial vertex with the vertex 0 in 
\eqref{eq:T2}, we consider mutations restricted on $\bbT_2$ for the indices 1 and 2.
This gives a subfan $\Delta_{12}(\tilde B)$ of $\Delta(\tilde B)$.
As in the rank 2 case,
we have the sequences of $g$-vectors
 $\tilde \bfg_m$ and $\tilde \bfg'_m$ ($m\geq 1$)
for the forward and backward mutations along $\bbT_2$.
By \eqref{eq:gmut1}, we have
\begin{align}
\tilde \bfg_m
= 
\begin{pmatrix}
\bfg_m
\\
*
\end{pmatrix}
,
\quad
\tilde \bfg'_m
= 
\begin{pmatrix}
\bfg'_m
\\
*
\end{pmatrix}
,
\end{align}
where $\bfg_m$ and $\bfg'_m$ the ones in \eqref{eq:gseq1} and \eqref{eq:gseq2}.
The fan $\Delta_{12}(\tilde B)$ consists of the
cones
\begin{gather}
\sigma(\bfe_1,\dots, \bfe_n),\quad \sigma(\bfg_1,\bfe_2, \dots, \bfe_n),\quad \sigma(\bfe_1,\bfg'_1, \bfe_3,\dots, \bfe_n),
\\
 \sigma(\bfg_m, \bfg_{m+1}, \bfe_3, \dots, \bfe_n), \quad \sigma(\bfg'_m, \bfg'_{m+1}, \bfe_3, \dots, \bfe_n)\quad  (m\geq 1)
 \end{gather}
and their faces.
Let
\begin{align}
\label{eq:tv1}
\tilde \bfv := \lim_{m\rightarrow \infty} \alpha_m^{-1}\bfg_m
=
\begin{pmatrix}
\bfv
\\
*
\end{pmatrix}
,
\quad
\tilde \bfv' := \lim_{m\rightarrow \infty} \alpha'_m{}^{-1}\bfg'_m
=
\begin{pmatrix}
\bfv'
\\
*
\end{pmatrix}
.
\end{align}
It is not obvious whether each component has a finite limit.
However, the following fact holds:

\begin{prop}
\label{prop:v1}
Each component of  $\tilde\bfv$ and $\tilde\bfv'$ has a finite limit.
Moreover, if $ab=4$, then $\tilde\bfv=\tilde\bfv'$.
\end{prop}

We only consider the mutations along the subtree $\bbT_2$ for the indices 1 and 2.
So, by \eqref{eq:gmut1}, for $j\geq 3$, the mutation of the $j$th component of $\tilde \bfg_m$ and $\tilde \bfg'_m$ only depends on
the $j$th row of the exchange matrix $\tilde B_t$.
Also, by \eqref{eq:bmut1}, the mutation of the $j$th row of $\tilde B_t$ only depends on the first, the second, and the $j$th row of $B_t$.
So, to prove Proposition \ref{prop:v1}, it is enough to concentrate on the rank 3 situation.
Thus, we may assume that the initial matrix $\tilde B_0$ has the form
\begin{align}
\label{eq:B0}
\tilde B_0
=
\begin{pmatrix}
0 & - b & - bc_0
\\
a & 0 & - ad_0
\\
c_0 & d_0 & 0
\end{pmatrix}
\quad
(a,\, b\, \in \bbZ_{>0},\, ab\geq 4;\, c_0,\, d_0\in \bbZ),
\end{align}
where the skew-symmetrizer is choosen as $D=\mathrm{diag}(a,b,ab)$.
In \eqref{eq:gmut1}, the third column of the exchange matrix is irrelevant.
So, the choice of $D$ is irrelevant.
In the same convention for the rank 2 case, we identify $\bbT_2$ with $\bbZ$.
Then, for $t\in \bbZ$, the exchange matrix at $t$ is given by
(omitting the third column)
\begin{align}
\tilde B_t
=
\begin{pmatrix}
0 & (-1)^{t+1} b 
\\
(-1)^t a & 0 
\\
c_t & d_t 
\end{pmatrix}
\end{align}
where $c_t$ and $d_t$ are integers determined by the matrix mutation \eqref{eq:bmut1} along $\bbT_2$.

In the rest of this section, we prove Proposition \ref{prop:v1}.
To utilize the mutation formula \eqref{eq:gmut1} of $g$-vectors,
 it is necessary to know the explicit expressions (and  the signs)
of $c_t$ and $d_t$.
Fortunately, this was already obtained in \cite{Gekhtman19}.
The result crucially depends on the signs of $c_0$ and $d_0$.
So, we separate it into six types (Types 1, 2, 3, 4-1, 4-2, and 4-3).

\subsection{Type 1. {$c_0,\, d_0 \geq 0$}.}
This is the easiest case.
We omit the trivial case $(c_0,d_0)=\bfzero$.
According to \cite[Prop.~3.1]{Gekhtman19},
the sign-pattern $(\rmsgn(c_t), \rmsgn(d_t))$ completely coincides with the pattern of the tropical signs
$(\varepsilon_{1;t}, \varepsilon_{2;t})$.
If $c_t$ or $d_t$ is zero,
we conveniently set its sign to the corresponding tropical sign. (This is harmless because it does not
contribute to the mutation \eqref{eq:gmut1} anyway.)
This coincidence implies the following important consequence.
In \eqref{eq:gmut1}, we have $[-\varepsilon_{k;t} b_{jk;t}]_+ =0$ if
$\rmsgn({b_{jk;t})} = \varepsilon_{k;t}$.
Thus, \emph{$c_t$ and $d_t$ never contribute to the mutations for $\bfg_m$ and $\bfg'_m$.}
So, we have
\begin{align}
\label{eq:type1g1}
\tilde \bfg_m = 
\begin{pmatrix}
\bfg_m
\\
0
\end{pmatrix}
,
\quad
\tilde \bfg'_m = 
\begin{pmatrix}
\bfg'_m
\\
0
\end{pmatrix},
\end{align}
and
\begin{align}
\tilde \bfv = 
\begin{pmatrix}
\bfv
\\
0
\end{pmatrix}
,
\quad
\tilde \bfv' = 
\begin{pmatrix}
\bfv'
\\
0
\end{pmatrix}.
\end{align}
Therefore, the claim in Proposition \ref{prop:v1} holds.

\begin{ex}
\label{ex:type1}
We continue to consider the case $(a,b)=(3,2)$.
All examples in this section were calculated by computer with Quiver Mutation \cite{Keller08c}
to confirm the obtained formulas.
Let $(c_0,d_0)=(2,2)$. Then, we have
\begin{gather*}
\tilde\bfg_1=
\begin{pmatrix}
-1
\\
0
\\
0
\end{pmatrix}
,
\quad
\tilde\bfg_2=
\begin{pmatrix}
0
\\
-1
\\
0
\end{pmatrix}
,
\quad
\tilde\bfg_3=
\begin{pmatrix}
1
\\
-3
\\
0
\end{pmatrix}
,
\quad
\tilde\bfg_4=
\begin{pmatrix}
2
\\
-5
\\
0
\end{pmatrix}
,
\quad
\tilde\bfg_5=
\begin{pmatrix}
5
\\
-12
\\
0
\end{pmatrix}
,
\\
\tilde\bfg'_1=
\begin{pmatrix}
2
\\
-1
\\
0
\end{pmatrix}
,
\quad
\tilde\bfg'_2=
\begin{pmatrix}
5
\\
-3
\\
0
\end{pmatrix}
,
\quad
\tilde\bfg'_3=
\begin{pmatrix}
8
\\
-5
\\
0
\end{pmatrix}
,
\quad
\tilde\bfg'_4=
\begin{pmatrix}
19
\\
-12
\\
0
\end{pmatrix}
,
\quad
\tilde\bfg'_5=
\begin{pmatrix}
30
\\
-19
\\
0
\end{pmatrix}
.
\end{gather*}
\end{ex}

\subsection{Type 2. {$c_0> 0,\, d_0 < 0$}.}

According to \cite[Prop.~3.1]{Gekhtman19},
the sign-pattern $(\rmsgn(c_t), \rmsgn(d_t))$ is given as follows:
\begin{align}
\label{eq:bsign1}
\begin{tabular}{wc{30pt}|wc{15pt}wc{10pt}wc{10pt}wc{10pt}wc{10pt}|wc{10pt}|wc{10pt}wc{10pt}wc{10pt}wc{10pt}wc{10pt}wc{10pt}wc{15pt}}
$t$ &  & $-4$ & $-3$& $-2$ & $-1$ & $0$ & $1$ & $2$ & $3$ & $4$& $5$ & $6$
\\
\hline
$\rmsgn(c_t)$ & $\cdots$ & $-$ & \bp & $-$ & \bp & \cp & $-$ & \cm & $+$ & \cm & $+$ & \cm & $\cdots$
\\
$\rmsgn(d_t)$ & $\cdots$ & \bp & $-$ & \bp & $+$ & \bm & \cm & $+$ & \cm & $+$& \cm & $+$ & $\cdots$
\end{tabular}
\end{align}
If $c_t$ or $d_t$ is zero, we conveniently set its sign as above.
(The same applies to the other cases.)

Comparing them with the tropical signs \eqref{eq:ts1}, we find
\begin{align}
\rmsgn(d_1)\neq \varepsilon_{2;1},
\quad
\rmsgn(d_0)\neq \varepsilon_{2;0}.
\end{align}
First, we consider the forward mutations.
We have
\begin{align}
\tilde \bfg_1
=\begin{pmatrix}
\bfg_{1}
\\
0
\end{pmatrix}
=
\begin{pmatrix}
-1
\\
0
\\
0
\end{pmatrix}
,
\quad
\tilde \bfg_{2}
=\begin{pmatrix}
\bfg_{2}
\\
0
\end{pmatrix}
+ [-d_1]_+ \bfe_3 =
\begin{pmatrix}
0
\\
-1
\\
-d_1
\end{pmatrix}
\end{align}
By \eqref{eq:bmut1},
we have
\begin{align}
d_1 = d_0 + c_0[- b]_+ + [-c_0]_+ (-b) = d_0.
\end{align}
In the rest, $\tilde\bfg_m$ mutate exactly in the same way as $\bfg_m$.
Since the mutation \eqref{eq:gmut2} is linear in $\bfg_m$, we have
\begin{align}
\label{eq:type2g1}
\tilde \bfg_m =
\begin{pmatrix}
\alpha_m
\\
\beta_m
\\
d_0 \beta_m
\end{pmatrix}
\quad
(m\geq 1)
.
\end{align}
Similarly, for the backward mutations, we have
\begin{align}
\tilde \bfg'_{1} = 
\begin{pmatrix}
\bfg'_1
\\
0
\end{pmatrix}
 + [-d_0]_+ \bfe_3 =
\begin{pmatrix}
b
\\
-1
\\
-d_0
\end{pmatrix}
,
\quad
\tilde \bfg'_{2} = -  \bfe_1 + a\tilde \bfg'_1
=
\begin{pmatrix}
ab - 1
\\
-a
\\
-a d_0
\end{pmatrix}
.
\end{align}
In the rest, $\tilde\bfg'_m$ mutate exactly in the same way as $\bfg'_m$.
Then, by the linearity, we have
\begin{align}
\label{eq:type2g2}
\tilde \bfg'_m =
\begin{pmatrix}
\alpha'_m
\\
\beta'_m
\\
d_0 \beta'_m
\end{pmatrix}
\quad
(m\geq 1)
.
\end{align}
Thus, we obtain
\begin{align}
\tilde \bfv = 
\begin{pmatrix}
\bfv
\\
d_0v_2
\end{pmatrix}
,
\quad
\tilde \bfv' = 
\begin{pmatrix}
\bfv'
\\
d_0v'_2
\end{pmatrix}
.
\end{align}
Therefore, the claim in Proposition \ref{prop:v1} holds.

\begin{ex}
\label{ex:type2}
Let $(c_0,d_0)=(2,-2)$. Then, we have
\begin{gather*}
\tilde\bfg_1=
\begin{pmatrix}
-1
\\
0
\\
0
\end{pmatrix}
,
\quad
\tilde\bfg_2=
\begin{pmatrix}
0
\\
-1
\\
2
\end{pmatrix}
,
\quad
\tilde\bfg_3=
\begin{pmatrix}
1
\\
-3
\\
6
\end{pmatrix}
,
\quad
\tilde\bfg_4=
\begin{pmatrix}
2
\\
-5
\\
10
\end{pmatrix}
,
\quad
\tilde\bfg_5=
\begin{pmatrix}
5
\\
-12
\\
24
\end{pmatrix}
,
\\
\tilde\bfg'_1=
\begin{pmatrix}
2
\\
-1
\\
2
\end{pmatrix}
,
\quad
\tilde\bfg'_2=
\begin{pmatrix}
5
\\
-3
\\
6
\end{pmatrix}
,
\quad
\tilde\bfg'_3=
\begin{pmatrix}
8
\\
-5
\\
10
\end{pmatrix}
,
\quad
\tilde\bfg'_4=
\begin{pmatrix}
19
\\
-12
\\
24
\end{pmatrix}
,
\quad
\tilde\bfg'_5=
\begin{pmatrix}
30
\\
-19
\\
38
\end{pmatrix}
.
\end{gather*}
\end{ex}

\subsection{Type 3. {$c_0,\, d_0 \leq 0$}.}

We omit the trivial case $(c_0,d_0)=\bfzero$.
According to \cite[Prop.~3.1]{Gekhtman19},
the sign-pattern $(\rmsgn(c_t), \rmsgn(d_t))$ is given as follows:
\begin{align}
\begin{tabular}{wc{30pt}|wc{15pt}wc{10pt}wc{10pt}wc{10pt}wc{10pt}|wc{10pt}|wc{10pt}wc{10pt}wc{10pt}wc{10pt}wc{10pt}wc{10pt}wc{15pt}}
$t$ &  & $-4$ & $-3$& $-2$ & $-1$ & $0$ & $1$ & $2$ & $3$ & $4$& $5$ & $6$
\\
\hline
$\rmsgn(c_t)$ & $\cdots$ & $-$ & \bp & $+$ & \bm & \cm & $+$ & \cm & $+$ & \cm & $+$ & \cm & $\cdots$
\\
$\rmsgn(d_t)$ & $\cdots$ & \bp & $-$ & \bp & $+$ & \bm & \cm & $+$ & \cm & $+$& \cm & $+$ & $\cdots$
\end{tabular}
\end{align}
Comparing them with the tropical signs \eqref{eq:ts1}, we find
\begin{align}
\rmsgn(c_0)\neq \varepsilon_{1;0},
\quad
\rmsgn(d_1)\neq \varepsilon_{2;1},
\quad
\rmsgn(d_0)\neq \varepsilon_{2;0},
\quad
\rmsgn(c_{-1})\neq \varepsilon_{1;-1}.
\end{align}
First, consider the forward mutations.
We have
\begin{gather}
\tilde \bfg_1 = -\bfe_1 + [-c_0]_+ \bfe_3
=
\begin{pmatrix}
-1
\\
0
\\
-c_0
\end{pmatrix},
\quad
\tilde \bfg_2 = - \bfe_2 + [-d_1]_+ \bfe_3
=
\begin{pmatrix}
0
\\
-1
\\
-d_1
\end{pmatrix},
\\
d_1 = d_0 + c_0[- b]_+ + [-c_0]_+ (-b) = d_0 + bc_0 \leq 0.
\end{gather}
In the rest, $\tilde\bfg_m$ mutate exactly in the same way as $\bfg_m$.
Then,
by the linearity, we have
\begin{align}
\label{eq:type3g1}
\tilde \bfg_m=
\begin{pmatrix}
\bfg_m
\\
c_0 \alpha_m + (d_0 + bc_0)\beta_m
\end{pmatrix}
\quad
(m\geq 1)
.
\end{align}
Similarly, for the backward mutations, we have
\begin{gather}
\tilde \bfg'_1 =
 \begin{pmatrix}
\bfg'_1
\\
0
\end{pmatrix}
 + [-d_0]_+ \bfe_3
=
\begin{pmatrix}
b
\\
-1
\\
-d_0
\end{pmatrix}
,
\\
\tilde \bfg'_2 = - \bfe_1 + a  \tilde \bfg'_1 + [-c_{-1}]_+ \bfe_3
=
\begin{pmatrix}
ab -1
\\
-a
\\
-c_{-1} -a d_0
\end{pmatrix}
,
\\
c_{-1} = c_0 + c_0[a ]_+ + [-c_0]_+ a = c_0.
\end{gather}
In the rest, $\tilde\bfg'_m$ mutate exactly in the same way as $\bfg'_m$.
Then,
by the linearity, we have
\begin{align}
\label{eq:type3g2}
\tilde \bfg'_m=
\begin{pmatrix}
\bfg'_m
\\
c_0 \alpha'_m + (d_0 + bc_0)\beta'_m
\end{pmatrix}
\quad
(m\geq 1)
.
\end{align}
Thus, we obtain
\begin{align}
\tilde \bfv = 
\begin{pmatrix}
\bfv
\\
c_0+
(d_0+bc_0)v_2
\end{pmatrix}
,
\quad
\tilde \bfv' = 
\begin{pmatrix}
\bfv'
\\
c_0+(d_0+bc_0)v'_2
\end{pmatrix}.
\end{align}
Therefore, the claim in Proposition \ref{prop:v1} holds.

\begin{ex}
\label{ex:type3}
Let $(c_0,d_0)=(-2,-2)$. We have $d_0 + b c_0  = -6$, and
\begin{gather*}
\tilde\bfg_1=
\begin{pmatrix}
-1
\\
0
\\
2
\end{pmatrix}
,
\quad
\tilde\bfg_2=
\begin{pmatrix}
0
\\
-1
\\
6
\end{pmatrix}
,
\quad
\tilde\bfg_3=
\begin{pmatrix}
1
\\
-3
\\
16
\end{pmatrix}
,
\quad
\tilde\bfg_4=
\begin{pmatrix}
2
\\
-5
\\
26
\end{pmatrix}
,
\quad
\tilde\bfg_5=
\begin{pmatrix}
5
\\
-12
\\
62
\end{pmatrix}
,
\\
\tilde\bfg'_1=
\begin{pmatrix}
2
\\
-1
\\
2
\end{pmatrix}
,
\quad
\tilde\bfg'_2=
\begin{pmatrix}
5
\\
-3
\\
8
\end{pmatrix}
,
\quad
\tilde\bfg'_3=
\begin{pmatrix}
8
\\
-5
\\
14
\end{pmatrix}
,
\quad
\tilde\bfg'_4=
\begin{pmatrix}
19
\\
-12
\\
34
\end{pmatrix}
,
\quad
\tilde\bfg'_5=
\begin{pmatrix}
30
\\
-19
\\
54
\end{pmatrix}
.
\end{gather*}
\end{ex}

\subsection{Type 4. {$c_0< 0,\, d_0 > 0$}.}
We separate it into three subcases.
\begin{align}
\label{eq:case41}
&\text{Type 4-1:}\quad
\frac{ab-\sqrt{ab(ab-4)}}{2a}
\leq  -\frac{d_0}{c_0}
\leq
\frac{ab+\sqrt{ab(ab-4)}}{2a}
,
\\
\label{eq:case42}
&\text{Type 4-2:}\quad
\frac{ab+\sqrt{ab(ab-4)}}{2a}
<
  -\frac{d_0}{c_0}
,
\\
\label{eq:case43}
&\text{Type 4-3:}\quad
  -\frac{d_0}{c_0}
<
\frac{ab-\sqrt{ab(ab-4)}}{2a}
.
\end{align}
The conditions are equivalent to
\begin{align}
\label{eq:type41cond1}
&\text{Type 4-1:}\quad
abc_0d_0 +ad_0^2 + bc_0^2 \leq 0,
\\
\label{eq:type42cond1}
&\text{Type 4-2:}\quad
abc_0d_0 +ad_0^2 + bc_0^2 > 0,
\quad
2d_0 + bc_0 >0,
\\
\label{eq:type43cond1}
&\text{Type 4-3:}\quad
abc_0d_0 +ad_0^2 + bc_0^2 > 0,
\quad
2d_0 + bc_0 <0.
\end{align}

\subsubsection{Type 4-1}
In this case,
according to \cite[Prop.~3.2]{Gekhtman19},
the sign-pattern $(\rmsgn(c_t), \rmsgn(d_t))$ is given as follows:
\begin{align}
\begin{tabular}{wc{30pt}|wc{15pt}wc{10pt}wc{10pt}wc{10pt}wc{10pt}|wc{10pt}|wc{10pt}wc{10pt}wc{10pt}wc{10pt}wc{10pt}wc{10pt}wc{15pt}}
$t$ &  & $-4$ & $-3$& $-2$ & $-1$ & $0$ & $1$ & $2$ & $3$ & $4$& $5$ & $6$
\\
\hline
$\rmsgn(c_t)$ & $\cdots$ & $-$ & \bp & $-$ & \bp & \cm & $+$ & \cm & $+$ & \cm & $+$ & \cm & $\cdots$
\\
$\rmsgn(d_t)$ & $\cdots$ & \bp & $-$ & \bp & $-$ & \bp & \cm & $+$ & \cm & $+$& \cm & $+$ & $\cdots$
\end{tabular}
\end{align}
Comparing them with the tropical signs \eqref{eq:ts1}, we find
\begin{align}
\rmsgn(c_0)\neq \varepsilon_{1;0},
\quad
\rmsgn(d_1)\neq \varepsilon_{2;1}.
\end{align}
For the forward mutations, this is the same as Case 3. Thus, we have
\begin{align}
\label{eq:type41g1}
\tilde \bfg_m=
\begin{pmatrix}
\bfg_m
\\
c_0 \alpha_m + (d_0 + bc_0)\beta_m
\end{pmatrix}
\quad
(m\geq 1)
.
\end{align}
For the backward mutations, this is the same as Case 1. Thus, we have
\begin{align}
\label{eq:type41g2}
\tilde \bfg'_m = 
\begin{pmatrix}
\bfg'_m
\\
0
\end{pmatrix}
\quad
(m\geq 1)
.
\end{align}
So, we obtain
\begin{align}
\tilde \bfv = 
\begin{pmatrix}
\bfv
\\
c_0+
(d_0+bc_0)v_2
\end{pmatrix}
,
\quad
\tilde \bfv' = 
\begin{pmatrix}
\bfv'
\\
0
\end{pmatrix}.
\end{align}
If $ab=4$, the condition \eqref{eq:case41} reduces to
\begin{align}
\frac{2}{a} = -\frac{d_0}{c_0} = \frac{b}{2}.
\end{align}
Since $v_2=-a/2$, we have $c_0+
(d_0+bc_0)v_2=0$. Thus, we have $\bfv=\bfv'$.
Therefore, the claim in Proposition \ref{prop:v1} holds.

\begin{ex}
\label{ex:type41}
Let $(c_0,d_0)=(-2,2)$.
We have
\begin{align}
\frac{ab-\sqrt{ab(ab-4)}}{2a}
=\frac{3-\sqrt{3}}{3}= 0.422\cdots,
\quad
\frac{ab+\sqrt{ab(ab-4)}}{2a}
=
\frac{3+\sqrt{3}}{3}=1.577\cdots.
\end{align}
Therefore, $(c_0,d_0)=(-2,2)$ satisfies the condition.
We have $d_0+bc_0=-2$,
and
\begin{gather*}
\tilde\bfg_1=
\begin{pmatrix}
-1
\\
0
\\
2
\end{pmatrix}
,
\quad
\tilde\bfg_2=
\begin{pmatrix}
0
\\
-1
\\
2
\end{pmatrix}
,
\quad
\tilde\bfg_3=
\begin{pmatrix}
1
\\
-3
\\
4
\end{pmatrix}
,
\quad
\tilde\bfg_4=
\begin{pmatrix}
2
\\
-5
\\
6
\end{pmatrix}
,
\quad
\tilde\bfg_5=
\begin{pmatrix}
5
\\
-12
\\
14
\end{pmatrix}
,
\\
\tilde\bfg'_1=
\begin{pmatrix}
2
\\
-1
\\
0
\end{pmatrix}
,
\quad
\tilde\bfg'_2=
\begin{pmatrix}
5
\\
-3
\\
0
\end{pmatrix}
,
\quad
\tilde\bfg'_3=
\begin{pmatrix}
8
\\
-5
\\
0
\end{pmatrix}
,
\quad
\tilde\bfg'_4=
\begin{pmatrix}
19
\\
-12
\\
0
\end{pmatrix}
,
\quad
\tilde\bfg'_5=
\begin{pmatrix}
30
\\
-19
\\
0
\end{pmatrix}
.
\end{gather*}
\end{ex}

\subsubsection{Type 4-2}
In this case, there is some $N\geq 0$ such that
\begin{align}
\label{eq:case42N}
\frac{\nu U_{N+1}}{U_N} \leq - \frac{d_0}{c_0} < \frac{\nu U_N}{U_{N-1}},
\end{align}
where we ignore the second inequality for $N=0$.
(In \cite[Prop.~3.3]{Gekhtman19}, $N=0$ is excluded, but it should have been included.)
We need to separate it into two subtypes.

\medskip\noindent
\emph{Type 4-2-1}. $N$ is odd.
According to \cite[Prop.~3.3]{Gekhtman19},
the sign-pattern $(\rmsgn(c_t), \rmsgn(d_t))$ is given as follows:
\begin{align}
\begin{tabular}{wc{30pt}|wc{15pt}wc{10pt}wc{10pt}|wc{10pt}|wc{10pt}wc{10pt}wc{10pt}wc{10pt}wc{10pt}wc{10pt}wc{10pt}wc{10pt}wc{15pt}}
$t$ &  & $-2$ & $-1$& $0$ & $1$ &  & $\scriptstyle N$ & $\scriptstyle N+1$ & $\scriptstyle N+2$ & $\scriptstyle N+3$&  & 
\\
\hline
$\rmsgn(c_t)$ & $\cdots$ & $-$ & \bp & \cm & $+$ & $\cdots$ & $+$ & \cp & $-$ & \cm & $+$ & \cm & $\cdots$
\\
$\rmsgn(d_t)$ & $\cdots$ & \bp & $-$ & \bp & \cm & $\cdots$ & \cm & $+$ & \cp & $-$& \cm & $+$ & $\cdots$
\end{tabular}
\end{align}
Comparing them with the tropical signs \eqref{eq:ts1}, we find
\begin{align}
\rmsgn(c_0)\neq \varepsilon_{1;0},
\quad
\rmsgn(d_1)\neq \varepsilon_{2;1},
\quad
\rmsgn(c_{N+1})\neq \varepsilon_{1;N+1},
\quad
\rmsgn(d_{N+2})\neq \varepsilon_{2;N+2}.
\end{align}
Therefore, $\tilde \bfg'_m$ ($m \geq 1$) and $\tilde\bfg_m$ ($1\leq m \leq N+1$)
are the same as Case 4-1.
So, let us determine $\tilde\bfg_m$ ($m\geq N+2$).
We first note the result by \cite[(3.22)]{Gekhtman19},
\begin{align}
c_{N+1}& = c_0 U_{N+1} + d_0 \nu^{-1} U_N \geq 0,
\\
d_{N+1}& = - c_0 \nu U_{N} - d_0 U_{N-1} >0,
\end{align}
where the inequalities match the condition \eqref{eq:case42N}.
We also have
\begin{align}
d_{N+2} &= d_{N+1} + c_{N+1}[-b]_+ + [-c_{N+1}]_+ (-b) = d_{N+1}.
\end{align}
Below, we use the recursion
\begin{align}
U_{N-1} + U_{N+1} = \kappa U_N,
\end{align}
which is obtained by \eqref{eq:Urel1}, repeatedly without mentioning.
Then, by \eqref{eq:gseq1} and \eqref{eq:type41g1},  we have
\begin{align*}
\tilde \bfg_{N+2} & =
-\tilde \bfg_{N} + a \tilde \bfg_{N+1} + c_{N+1} \bfe_3
\\ 
&=
\begin{pmatrix}
\bfg_{N+2}
\\
c_0  U_{N-1} - (d_0 + bc_0) \nu^{-1} U_N
\end{pmatrix}
+ c_{N+1} \bfe_3
=
\begin{pmatrix}
\bfg_{N+2}
\\
0
\end{pmatrix}
.
\end{align*}
Here, an unexpected and pleasant cancellation occurred.
Similarly, 
\begin{align*}
\tilde \bfg_{N+3} &=
-\tilde \bfg_{N+1} + b \tilde \bfg_{N+2} + d_{N+2} \bfe_3
\\
&=
\begin{pmatrix}
\bfg_{N+3}
\\
- c_0  \nu U_{N-2} + (d_0 + bc_0)  U_{N-1}
\\
\end{pmatrix}
+
 d_{N+1} \bfe_3
 =
\begin{pmatrix}
\bfg_{N+3}
\\
0
\end{pmatrix}
.
\end{align*}
Thus, we have
\begin{align}
\label{eq:type421g1}
\tilde \bfg_{m}
=
\begin{pmatrix}
\bfg_{m}
\\
0
\end{pmatrix}
\quad
(m \geq N+2).
\end{align}
So, we obtain
\begin{align}
\tilde \bfv = 
\begin{pmatrix}
\bfv
\\
0
\end{pmatrix}
,
\quad
\tilde \bfv' = 
\begin{pmatrix}
\bfv'
\\
0
\end{pmatrix}.
\end{align}
Therefore, the claim in Proposition \ref{prop:v1} holds.

\begin{ex}
\label{ex:type421}
Let $N=3$ and $(c_0,d_0)=(-100,159)$.
We have
\begin{align}
\frac{\nu U_4}{U_3}= \frac{19}{12}= 1.583\cdots,
\quad
\frac{\nu U_3}{U_2}= \frac{8}{5}= 1.6.
\end{align}
Therefore, $(c_0,d_0)=(-100,159)$ satisfies the condition.
We have $d_0+bc_0=-41$,
and
\begin{align*}
\tilde\bfg_1&=
\begin{pmatrix}
-1
\\
0
\\
100
\end{pmatrix}
,
\quad
\tilde\bfg_2=
\begin{pmatrix}
0
\\
-1
\\
41
\end{pmatrix}
,
\quad
\tilde\bfg_3=
\begin{pmatrix}
1
\\
-3
\\
23
\end{pmatrix}
,
\quad
\tilde\bfg_4=
\begin{pmatrix}
2
\\
-5
\\
5
\end{pmatrix}
,
\quad
\tilde\bfg_5=
\begin{pmatrix}
5
\\
-12
\\
0
\end{pmatrix}
,
\\
\tilde\bfg_6&=
\begin{pmatrix}
8
\\
-19
\\
0
\end{pmatrix}
,
\quad
\tilde\bfg_7=
\begin{pmatrix}
19
\\
-45
\\
0
\end{pmatrix}
,
\\
\tilde\bfg'_1&=
\begin{pmatrix}
2
\\
-1
\\
0
\end{pmatrix}
,
\quad
\tilde\bfg'_2=
\begin{pmatrix}
5
\\
-3
\\
0
\end{pmatrix}
,
\quad
\tilde\bfg'_3=
\begin{pmatrix}
8
\\
-5
\\
0
\end{pmatrix}
,
\quad
\tilde\bfg'_4=
\begin{pmatrix}
19
\\
-12
\\
0
\end{pmatrix}
,
\quad
\tilde\bfg'_5=
\begin{pmatrix}
30
\\
-19
\\
0
\end{pmatrix}
.
\end{align*}
\end{ex}

\medskip\noindent
\emph{Type 4-2-2}. $N$ is even.
According to \cite[Prop.~3.3]{Gekhtman19},
the sign-pattern $(\rmsgn(c_t), \rmsgn(d_t))$ is given as follows:
\begin{align}
\begin{tabular}{wc{30pt}|wc{15pt}wc{10pt}wc{10pt}|wc{10pt}|wc{10pt}wc{10pt}wc{10pt}wc{10pt}wc{10pt}wc{10pt}wc{10pt}wc{10pt}wc{15pt}}
$t$ &  & $-2$ & $-1$& $0$ & $1$ &  &  &$\scriptstyle N$ & $\scriptstyle N+1$ & $\scriptstyle N+2$ & $\scriptstyle N+3$&   
\\
\hline
$\rmsgn(c_t)$ & $\cdots$ & $-$ & \bp & \cm & $+$ & $\cdots$ & $+$ & \cm & $+$ & \cp & $-$ & \cm & $\cdots$
\\
$\rmsgn(d_t)$ & $\cdots$ & \bp & $-$ & \bp & \cm & $\cdots$ & \cm & $+$ & \cp & $-$& \cm & $+$ & $\cdots$
\end{tabular}
\end{align}
Comparing them with the tropical signs \eqref{eq:ts1}, we find
\begin{align}
\rmsgn(c_0)\neq \varepsilon_{1;0},
\quad
\rmsgn(d_1)\neq \varepsilon_{2;1},
\quad
\rmsgn(d_{N+1})\neq \varepsilon_{2;N+1},
\quad
\rmsgn(c_{N+2})\neq \varepsilon_{1;N+2}.
\end{align}
Therefore, $\tilde\bfg'_m$ ($m \geq 1$) and $\tilde\bfg_m$ ($1\leq m \leq N+1$)
are the same as Case 4-1.
So, let us determine $\tilde\bfg_m$ ($m\geq N+2$).
We first note the result by \cite[(3.23)]{Gekhtman19},
\begin{align}
c_{N+1}& = - c_0 U_{N} -  d_0 \nu^{-1} U_{N-1} > 0,
\\
d_{N+1}& =  c_0 \nu U_{N+1} + d_0 U_{N} \geq 0,
\end{align}
where the inequalities match the condition \eqref{eq:case42N}.
We also have
\begin{align}
c_{N+2} &= c_{N+1} + d_{N+1}[-a]_+ + [-d_{N+1}]_+ (-a) = c_{N+1}.
\end{align}
Then, by \eqref{eq:gseq1} and \eqref{eq:type41g1},  we have
\begin{align*}
\tilde \bfg_{N+2} & =
-\tilde \bfg_{N} + b \tilde \bfg_{N+1} + d_{N+1} \bfe_3
\\ 
&=
\begin{pmatrix}
\bfg_{N+2}
\\
c_0  \nu U_{N-1} - (d_0 + bc_0)  U_N
\end{pmatrix}
+ d_{N+1} \bfe_3
=
\begin{pmatrix}
\bfg_{N+2}
\\
0
\end{pmatrix}
.
\end{align*}
Similarly, 
\begin{align*}
\tilde \bfg_{N+3} &=
-\tilde \bfg_{N+1} + a \tilde \bfg_{N+2} + c_{N+2} \bfe_3
\\
&=
\begin{pmatrix}
\bfg_{N+3}
\\
- c_0   U_{N-2} + (d_0 + bc_0)  \nu^{-1} U_{N-1}
\\
\end{pmatrix}
+
 c_{N+1} \bfe_3
 =
\begin{pmatrix}
\bfg_{N+3}
\\
0
\end{pmatrix}
.
\end{align*}
Thus, we have
\begin{align}
\tilde \bfg_{m}
=
\begin{pmatrix}
\bfg_{m}
\\
0
\end{pmatrix}
\quad
(m \geq N+2).
\end{align}
So, we obtain
\begin{align}
\tilde \bfv = 
\begin{pmatrix}
\bfv
\\
0
\end{pmatrix}
,
\quad
\tilde \bfv' = 
\begin{pmatrix}
\bfv'
\\
0
\end{pmatrix}.
\end{align}
Therefore, the claim in Proposition \ref{prop:v1} holds.

\begin{ex}
\label{ex:type422}
Let $N=4$ and $(c_0,d_0)=(-50,79)$.
We have
\begin{align}
\frac{\nu U_5}{U_4}= \frac{30}{19}= 1.578\cdots,
\quad
\frac{\nu U_4}{U_3}= \frac{19}{12}= 1.583\cdots.
\end{align}
Therefore, $(c_0,d_0)=(-50, 79)$ satisfies the condition.
We have $d_0+bc_0=-21$,
and
\begin{align*}
\tilde\bfg_1&=
\begin{pmatrix}
-1
\\
0
\\
50
\end{pmatrix}
,
\quad
\tilde\bfg_2=
\begin{pmatrix}
0
\\
-1
\\
21
\end{pmatrix}
,
\quad
\tilde\bfg_3=
\begin{pmatrix}
1
\\
-3
\\
13
\end{pmatrix}
,
\quad
\tilde\bfg_4=
\begin{pmatrix}
2
\\
-5
\\
5
\end{pmatrix}
,
\quad
\tilde\bfg_5=
\begin{pmatrix}
5
\\
-12
\\
2
\end{pmatrix}
,
\\
\tilde\bfg_6&=
\begin{pmatrix}
8
\\
-19
\\
0
\end{pmatrix}
,
\quad
\tilde\bfg_7=
\begin{pmatrix}
19
\\
-45
\\
0
\end{pmatrix}
,
\\
\tilde\bfg'_1&=
\begin{pmatrix}
2
\\
-1
\\
0
\end{pmatrix}
,
\quad
\tilde\bfg'_2=
\begin{pmatrix}
5
\\
-3
\\
0
\end{pmatrix}
,
\quad
\tilde\bfg'_3=
\begin{pmatrix}
8
\\
-5
\\
0
\end{pmatrix}
,
\quad
\tilde\bfg'_4=
\begin{pmatrix}
19
\\
-12
\\
0
\end{pmatrix}
,
\quad
\tilde\bfg'_5=
\begin{pmatrix}
30
\\
-19
\\
0
\end{pmatrix}
.
\end{align*}
\end{ex}

\subsubsection{Type 4-3}
In this case, there is some $N\geq 0$ such that
\begin{align}
\label{eq:case43N}
\frac{\nu U_{N-1}}{U_N} \leq - \frac{d_0}{c_0} < \frac{\nu U_N}{U_{N+1}}. 
\end{align}
(In \cite[Prop.~3.4]{Gekhtman19}, $N=0$ is excluded, but it should be included.)
We need to separate it into two subtypes.

\medskip\noindent
\emph{Type 4-3-1}. $N$ is odd.
According to \cite[Prop.~3.4]{Gekhtman19},
the sign-pattern $(\rmsgn(c_t), \rmsgn(d_t))$ is given as follows:
\begin{align}
\begin{tabular}{wc{30pt}|wc{15pt}wc{10pt}wc{10pt}wc{17pt}wc{17pt}wc{17pt}wc{17pt}wc{10pt}wc{10pt}wc{10pt}wc{10pt}wc{10pt}wc{15pt}}
$t$ &  & $$ & $$& $\scriptstyle -N-3$ & $\scriptstyle -N-2$ &$\scriptstyle -N-1$  & $\scriptstyle -N$ &  & $-1$ & $0$ & $1$  & $2$
\\
\hline
$\rmsgn(c_t)$ & $\cdots$ & $-$ & \bp & $+$ & \bm & $-$ & \bp & $\cdots$ & \bp & \cm & $+$ & \cm & $\cdots$
\\
$\rmsgn(d_t)$ & $\cdots$ & \bp & $-$ & \bp & $+$ & \bm & $-$ & $\cdots$ & $-$ & \bp & \cm & $+$ & $\cdots$
\end{tabular}
\end{align}
Comparing them with the tropical signs \eqref{eq:ts1}, we find
\begin{align}
\rmsgn(c_0)\neq \varepsilon_{1;0},
\quad
\rmsgn(d_1)\neq \varepsilon_{2;1},
\quad
\rmsgn(d_{-N-1})\neq \varepsilon_{2;-N-1},
\quad
\rmsgn(c_{-N-2})\neq \varepsilon_{1;-N-2}.
\end{align}
Therefore, $\tilde\bfg_m$ ($m \geq 1$) and $\tilde\bfg'_m$ ($1\leq m \leq N+1$)
are the same as Case 4-1.
So, let us determine $\tilde\bfg'_m$ ($m\geq N+2$).
We first note the result by \cite[(3.24)]{Gekhtman19},
\begin{align}
c_{-N-1}& = -c_0 U_{N-1} - d_0 \nu^{-1} U_N \leq  0,
\\
d_{-N-1}& =  c_0 \nu U_{N} + d_0 U_{N+1} < 0,
\end{align}
where the inequalities match the condition \eqref{eq:case43N}.
We also have
\begin{align}
c_{-N-2} &= c_{-N-1} + d_{-N-1}[a]_+ + [-d_{-N-1}]_+ a = c_{-N-1}.
\end{align}
Then, by \eqref{eq:gseq2} and \eqref{eq:type41g2},  we have
\begin{align*}
\tilde \bfg'_{N+2} & =
-\tilde \bfg'_{N} + b \tilde \bfg'_{N+1} - d_{-N-1} \bfe_3
\\ 
&=
\begin{pmatrix}
\bfg'_{N+2}
\\
- c_0 \nu U_{N} - d_0 U_{N+1} 
\end{pmatrix}
=
\begin{pmatrix}
\bfg'_{N+2}
\\
c_0 \alpha'_{N+2}
+(d_0+ bc_0)\beta'_{N+2}
\end{pmatrix}
.
\end{align*}
Similarly, 
\begin{align*}
\tilde \bfg'_{N+3} &=
-\tilde \bfg'_{N+1} + a \tilde \bfg_{N+2} - c_{-N-2} \bfe_3
\\
&=
\begin{pmatrix}
\bfg'_{N+3}
\\
- c_0 a \nu U_{N} - d_0 a U_{N+1} 
 + c_0 U_{N-1} + d_0 \nu^{-1} U_N
\\
\end{pmatrix}
\end{align*}
A little calculation shows that
\begin{align*}
&
- c_0 a \nu U_{N} - d_0 a U_{N+1} 
 + c_0 U_{N-1} + d_0 \nu^{-1} U_N
 \\
=&\
c_0 U_{N+3} + (d_0 + bc_0) (-\nu^{-1} U_{N+2})
=
c_0 \alpha'_{N+3} +  (d_0 + bc_0) \beta'_{N+3}.
\end{align*}
Thus, we have
\begin{align}
\label{eq:type431g2}
\tilde \bfg'_{m}
=
\begin{pmatrix}
\bfg'_{m}
\\
c_0 \alpha'_{m} +  (d_0 + bc_0) \beta'_{m}
\end{pmatrix}
\quad
(m \geq N+2).
\end{align}
So, we obtain
\begin{align}
\tilde \bfv = 
\begin{pmatrix}
\bfv
\\
c_0 + (d_0 + bc_0) v_2
\end{pmatrix}
,
\quad
\tilde \bfv' = 
\begin{pmatrix}
\bfv'
\\
c_0 + (d_0 + bc_0) v'_2
\end{pmatrix}.
\end{align}
Therefore, the claim in Proposition \ref{prop:v1} holds.

\begin{ex}
\label{ex:type431}
Let $N=3$ and $(c_0,d_0)=(-50,21)$.
We have
\begin{align}
\frac{\nu U_2}{U_3}= \frac{5}{12}= 0.416\cdots,
\quad
\frac{\nu U_3}{U_4}= \frac{8}{19}= 0.421\cdots.
\end{align}
Therefore, $(c_0,d_0)=(-50,21)$ satisfies the condition.
We have $d_0+bc_0=-79$,
and
\begin{align*}
\tilde\bfg_1&=
\begin{pmatrix}
-1
\\
0
\\
50
\end{pmatrix}
,
\quad
\tilde\bfg_2=
\begin{pmatrix}
0
\\
-1
\\
79
\end{pmatrix}
,
\quad
\tilde\bfg_3=
\begin{pmatrix}
1
\\
-3
\\
187
\end{pmatrix}
,
\quad
\tilde\bfg_4=
\begin{pmatrix}
2
\\
-5
\\
295
\end{pmatrix}
,
\quad
\tilde\bfg_5=
\begin{pmatrix}
5
\\
-12
\\
698
\end{pmatrix}
,
\\
\tilde\bfg'_1&=
\begin{pmatrix}
2
\\
-1
\\
0
\end{pmatrix}
,
\quad
\tilde\bfg'_2=
\begin{pmatrix}
5
\\
-3
\\
0
\end{pmatrix}
,
\quad
\tilde\bfg'_3=
\begin{pmatrix}
8
\\
-5
\\
0
\end{pmatrix}
,
\quad
\tilde\bfg'_4=
\begin{pmatrix}
19
\\
-12
\\
0
\end{pmatrix}
,
\quad
\tilde\bfg'_5=
\begin{pmatrix}
30
\\
-19
\\
1
\end{pmatrix}
,
\\
\tilde\bfg'_6&=
\begin{pmatrix}
71
\\
-45
\\
5
\end{pmatrix}
,
\quad
\tilde\bfg'_7=
\begin{pmatrix}
112
\\
-71
\\
9
\end{pmatrix}
.
\end{align*}
\end{ex}

\medskip\noindent
\emph{Type 4-3-2}. $N$ is even.
According to \cite[Prop.~3.4]{Gekhtman19},
the sign-pattern $(\rmsgn(c_t), \rmsgn(d_t))$ is given as follows:
\begin{align}
\begin{tabular}{wc{30pt}|wc{15pt}wc{10pt}wc{10pt}wc{17pt}wc{17pt}wc{17pt}wc{17pt}wc{10pt}wc{10pt}wc{10pt}wc{10pt}wc{10pt}wc{15pt}}
$t$ &  & $$ & $\scriptstyle -N-3$ & $\scriptstyle -N-2$ &$\scriptstyle -N-1$  & $\scriptstyle -N$ & & & $-1$ & $0$ & $1$  & $2$
\\
\hline
$\rmsgn(c_t)$ & $\cdots$ & $-$ & \bp & $+$ & \bm & $-$ & \bp & $\cdots$ & \bp & \cm & $+$ & \cm & $\cdots$
\\
$\rmsgn(d_t)$ & $\cdots$ & \bp & $+$ & \bm & $-$ & \bp & $-$ & $\cdots$ & $-$ & \bp & \cm & $+$ & $\cdots$
\end{tabular}
\end{align}
Comparing them with the tropical signs \eqref{eq:ts1}, we find
\begin{align}
\rmsgn(c_0)\neq \varepsilon_{1;0},
\quad
\rmsgn(d_1)\neq \varepsilon_{2;1},
\quad
\rmsgn(c_{-N-1})\neq \varepsilon_{1;-N-1},
\quad
\rmsgn(d_{-N-2})\neq \varepsilon_{2;-N-2}.
\end{align}
Therefore, $\tilde\bfg_m$ ($m \geq 1$) and $\tilde\bfg'_m$ ($1\leq m \leq N+1$)
are the same as Case 4-1.
So, let us determine $\tilde\bfg'_m$ ($m\geq N+2$).
We first note the result by \cite[(3.25)]{Gekhtman19},
\begin{align}
c_{-N-1}& =  c_0 U_{N} + d_0 \nu^{-1} U_{N+1} < 0,
\\
d_{-N-1}& =  - c_0 \nu U_{N-1} - d_0 U_{N} \leq 0,
\end{align}
where the inequalities match the condition \eqref{eq:case43N}.
We also have
\begin{align}
d_{-N-2} &= d_{-N-1} + c_{-N-1}[b]_+ + [-c_{-N-1}]_+ b = d_{-N-1}.
\end{align}
Then, by \eqref{eq:gseq2} and \eqref{eq:type41g2},  we have
\begin{align*}
\tilde \bfg'_{N+2} & =
-\tilde \bfg'_{N} + a \tilde \bfg'_{N+1} - c_{-N-1} \bfe_3
\\ 
&=
\begin{pmatrix}
\bfg'_{N+2}
\\
- c_0  U_{N} - d_0 \nu^{-1} U_{N+1} 
\end{pmatrix}
=
\begin{pmatrix}
\bfg'_{N+2}
\\
c_0 \alpha'_{N+2}
+(d_0+ bc_0)\beta'_{N+2}
\end{pmatrix}
.
\end{align*}
Similarly, 
\begin{align*}
\tilde \bfg'_{N+3} &=
-\tilde \bfg'_{N+1} + b \tilde \bfg_{N+2} - d_{-N-2} \bfe_3
\\
&=
\begin{pmatrix}
\bfg'_{N+3}
\\
- c_0 b  U_{N} - d_0 b \nu^{-1} U_{N+1} 
 + c_0 \nu U_{N-1} + d_0  U_N
\\
\end{pmatrix}
\end{align*}
A little calculation shows that
\begin{align*}
&
- c_0 b  U_{N} - d_0 b \nu^{-1} U_{N+1} 
 + c_0 \nu U_{N-1} + d_0  U_N
 \\
=&\
c_0 \nu U_{N+3} + (d_0 + bc_0) (- U_{N+2})
=
c_0 \alpha'_{N+3} +  (d_0 + bc_0) \beta'_{N+3}.
\end{align*}
Thus, we have
\begin{align}
\tilde \bfg'_{m}
=
\begin{pmatrix}
\bfg'_{m}
\\
c_0 \alpha'_{m} +  (d_0 + bc_0) \beta'_{m}
\end{pmatrix}
\quad
(m \geq N+2).
\end{align}
So, we obtain
\begin{align}
\tilde \bfv = 
\begin{pmatrix}
\bfv
\\
c_0 + (d_0 + bc_0) v_2
\end{pmatrix}
,
\quad
\tilde \bfv' = 
\begin{pmatrix}
\bfv'
\\
c_0 + (d_0 + bc_0) v'_2
\end{pmatrix}.
\end{align}
Therefore, the claim in Proposition \ref{prop:v1} holds.

\begin{ex}
\label{ex:type432}
Let $N=4$ and $(c_0,d_0)=(-500,211)$.
We have
\begin{align}
\frac{\nu U_3}{U_4}= \frac{8}{19}= 0.421\cdots,
\quad
\frac{\nu U_4}{U_5}= \frac{19}{45}= 0.422\cdots.
\end{align}
Therefore, $(c_0,d_0)=(-500,211)$ satisfies the condition.
We have $d_0+bc_0=-789$,
and
\begin{align*}
\tilde\bfg_1&=
\begin{pmatrix}
-1
\\
0
\\
500
\end{pmatrix}
,
\quad
\tilde\bfg_2=
\begin{pmatrix}
0
\\
-1
\\
789
\end{pmatrix}
,
\quad
\tilde\bfg_3=
\begin{pmatrix}
1
\\
-3
\\
1867
\end{pmatrix}
,
\quad
\tilde\bfg_4=
\begin{pmatrix}
2
\\
-5
\\
2945
\end{pmatrix}
,
\quad
\tilde\bfg_5=
\begin{pmatrix}
5
\\
-12
\\
6968
\end{pmatrix}
,
\\
\tilde\bfg'_1&=
\begin{pmatrix}
2
\\
-1
\\
0
\end{pmatrix}
,
\quad
\tilde\bfg'_2=
\begin{pmatrix}
5
\\
-3
\\
0
\end{pmatrix}
,
\quad
\tilde\bfg'_3=
\begin{pmatrix}
8
\\
-5
\\
0
\end{pmatrix}
,
\quad
\tilde\bfg'_4=
\begin{pmatrix}
19
\\
-12
\\
0
\end{pmatrix}
,
\quad
\tilde\bfg'_5=
\begin{pmatrix}
30
\\
-19
\\
0
\end{pmatrix}
,
\\
\tilde\bfg'_6&=
\begin{pmatrix}
71
\\
-45
\\
5
\end{pmatrix}
,
\quad
\tilde\bfg'_7=
\begin{pmatrix}
112
\\
-71
\\
19
\end{pmatrix}
.
\end{align*}
\end{ex}

\begin{rem}
We see from the sign-pattern $(\rmsgn(c_t), \rmsgn(d_t))$ that
for Type 4-1 the matrix $B_t$ is cyclic for any $t\in \bbT_2$, while 
for all other types $B_t$ is acyclic for some $t \in \bbT_2$.
\end{rem}

\section{Incompleteness of $G$-fans of infinite type}

As an immediate application of Proposition \ref{prop:v1}, we prove the incompleteness of the $G$-fans of infinite type.

The following fact is well-known.

\begin{thm}[{\cite[Thm.~10.6]{Reading12}, \cite[Thm.~4.1]{Reading18}}]
\label{thm:complete1}
If a cluster pattern $\bfSigma(B)$ is of finite type,
the $G$-fan $\Delta(B)$ is complete.
\end{thm}

We show the opposite statement.
\begin{thm}
\label{thm:incomplete1}
If a cluster pattern $\bfSigma(B)$ is of infinite type,
the $G$-fan $\Delta(B)$ is incomplete.
\end{thm}

\begin{figure}
\begin{tikzpicture}[scale=2]
\filldraw [fill=black!10, draw=white] (0,0) -- (-0.4,-0.7) -- (0.3,-0.7) -- (0,0);
\filldraw [fill=black!10, draw=white] (0,0) -- (0,1.4) -- (-0.55,0.55) -- (0,0);
\filldraw [fill=black!20, draw=white] (0,0) -- (0,1.4) -- (0.25,0.75) -- (0,0);
\draw(-1,0)--(1,0);
\draw(0,0)--(0,1.4);
\draw(0.7,0.7)--(-0.7,-0.7);
\draw(1.7,0.7)--(0.3,-0.7);
\draw(-0.3,0.7)--(-1.7,-0.7);
\draw(-1.7,-0.7)--(0.3,-0.7);
\draw(-0.3,0.7)--(1.7,0.7);
\draw(0,0)--(-0.4,-0.7);
\draw(0,0)--(0.3,-0.7);
\draw [dashed] (0.1,0.3)--(0.1,-0.23);
\draw [dashed] (-0.4,0.4)--(-0.4,-0.7);
\draw(0,0)--(-0.5,-0.7);
\draw(0,0)--(-0.44,-0.7);
\draw(0,0)--(0.36,-0.64);
\draw(0,0)--(0.45,-0.55);
\draw [->, very thick] (0,0)-- (-0.4,-0.7);
\draw [->, very thick] (0,0)-- (-0.4,0.4);
\draw [->, very thick] (0,0)-- (0.1,-0.23);
\draw [->, very thick] (0,0)-- (0.1,0.3);
\draw  (0,0)-- (-0.55,0.55);
\draw  (0,0)-- (0.25,0.75);
 \node at (-0.4,-0.85){\small $\bfv$};
 \node at (-0.5,0.3){\small $\tilde \bfv$};
 \node at (0,-0.35){\small $\bfv'$};
 \node at (0.25,0.35){\small $\tilde \bfv'$};
 \end{tikzpicture}
 \vskip-10pt
\caption{Boundaries of the subfan $\Delta_{12}(\tilde B)$.}
\label{fig:Gfan2}
\end{figure}
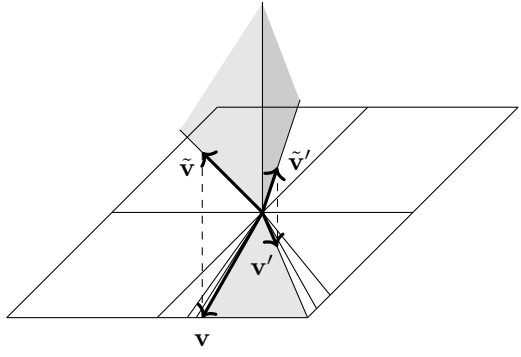

\begin{proof}
We consider the sitation in \eqref{eq:tB0}--\eqref{eq:tv1}.
Then,
thanks to Proposition \ref{prop:v1}, 
the codimen\-sion-one cones $\sigma(\tilde\bfv,  \bfe_3,\dots, \bfe_n)$ and $\sigma(\tilde\bfv',  \bfe_3,\dots, \bfe_n)$
are on the boundary of $\Delta_{12}(\tilde B)$. See Figure \ref{fig:Gfan2}.
(This is untrue if Proposition \ref{prop:v1} fails. For example, if $\tilde v_3=\pm \infty$ and other components have finite limits,
then we have the codimension-two cone $\sigma(\pm\bfe_3,\dots, \bfe_n)$ on the boundary instead of $\sigma(\tilde\bfv,  \bfe_3,\dots, \bfe_n)$.)
If $ab=4$ (affine type), then the common boundary $\sigma(\tilde\bfv,  \bfe_3,\dots, \bfe_n)$ (the crack) is codimension one.
So, it is not filled with other $G$-cones, which are full-dimensional.
Therefore, $\Delta(\tilde B)$ is incomplete.
If $ab\geq 5$ (non-affine type),
then the boundaries $\sigma(\tilde\bfv,  \bfe_3,\dots, \bfe_n)$ and $\sigma(\tilde\bfv',  \bfe_3,\dots, \bfe_n)$
have irrational normal vectors.
So, it is not covered with other $G$-cones, which have rational normal vectors.
Therefore, $\Delta(\tilde B)$ is incomplete.
This finishes the proof of Theorem \ref{thm:incomplete1}.
\end{proof}

\begin{rem}
\label{rem:incomplete1}
1. Alternatively,
one can deduce Theorem \ref{thm:incomplete1} from the result of
\cite[Thm.~5.1.1]{Muller15}
by considering the pullbacks of rank 2 scattering subdiagrams for the indices 1 and 2.
However,
the proof of \cite[Thm.~5.1.1]{Muller15} depends on the uniqueness of certain consistent scattering diagrams \cite[Thm.~1.2.1]{Gross14}.
Our approach is simpler and we only use the mutation formulas \eqref{eq:bmut1}--\eqref{eq:gmut1},
though we rely on the sign-coherence property of $c$-vectors proved by the scattering diagram method \cite{Gross14}.

2. In \cite[II.Prop.~2.18]{Nakanishi22a}, the author gave a proof of Theorem \ref{thm:incomplete1},
which turned out insufficient in detail.
Here, we replaced it with a proof with more quantitative details.
\end{rem}

\section{Pictures of rank 3 $G$-fans}
\label{sec:pictures1}

In this section, we present the pictures of the rank 3 $G$-fans for  Examples \ref{ex:type1}, \ref{ex:type2}, 
\ref{ex:type41}, \ref{ex:type421}, \ref{ex:type431}
and their variants
in Figures \ref{fig:type1}--\ref{fig:type4-3}.
These pictures match and visualize the behavior of $g$-vectors 
around the ray $\bbR_{\geq 0} \bfe_3$
classified in Section \ref{sec:lemv1}.

First, let us explain the pictures. More specific remarks on each picture will also be provided in the caption.
Addition information on the scattering diagram method can be found in \cite[\S III.6]{Nakanishi22a}.

(1)
We use the stereographic projection, which is commonly used to present rank 3 $G$-fans and scattering diagrams \cite{Muller15,Reading17}.
Namely, we first consider the projection of $G$-cones to the unit sphere $S^2$ in $\bbR^3$.
Then, we use the stereographic projection to the tangent plane at $(1,1,1)/\sqrt{3}$ from the antipode.
All pictures are topological; namely, they are not precise in metrics.

(2)
All examples in this section have the common skew-symmetrizer $D=\mathrm{diag}(3,2,6)$.
Accordingly, we consider the following inner product on $\bbR^3$:
\begin{align}
\label{eq:inner1}
(\bfa,\bfb)_D= \bfa^T D\bfb.
\end{align}
Then, the $c$- and $g$-vectors satisfy the duality \cite{Nakanishi22a}
\begin{align}
\label{eq:dualcg1}
(\bfc_{i;t}, \bfg_{j;t})_D =  d_i \delta_{ij}.
\end{align}
Each hyperplane $\bfe_i^{\perp}$ is presented as a circle.
See Figure \ref{fig:type1}.
Let
\begin{align}
H_i^{+}=\{ \bfa\in \bbR^3 \mid a_i \geq 0\},
\quad
H_i^{-}=\{ \bfa\in \bbR^3 \mid a_i \leq 0\}.
\end{align}
Under the projection, the half-space $H_i^{+}$ (resp. $H_i^{-}$) is the inside (resp. outside) of the circle $\bfe_i^{\perp}$.
In particular, the positive orthant $O_{+++}=H_1^{+}\cap H_2^{+} \cap H_3^{+}$ is 
 the triangle at the center.
The ray $\bbR_{\geq 0}\bfe_3$ is represented as the encircled vertex in the picture.
Let $v_3$ denote the vertex.
 The $g$-vectors $\tilde \bfg_m$ ($m\geq 2$) and $\tilde \bfg'_m$ ($m\geq 1$) studied in Section \ref{sec:lemv1}
 are in the orthant $O_{+-+}=H_1^{+}\cap H_2^{-} \cap H_3^{+}$,
 which is surrounded by bold edges.
 We are interested in the configurations of these $g$-vectors around the vertex $v_{3}$.
 
(3)
In general, a $G$-fan of rank 3 has a fractal nature due to the fractal structure of the $3$-regular tree $\bbT_3$.
Therefore, it is impossible to draw a complete picture.
We should be satisfied with a low-order approximation.

(4)
Let us explain briefly how we determine these pictures.
Each triagle corresponds to a $G$-cone.
Due to the duality \eqref{eq:dualcg1},
the normal vector of each edge of a triangle is a $c$-vector.
So, each $C$-matrix is identified with a $G$-cone.
For example, 
the initial $C$-matrix $I$ corresponds to the positive orthant $O_{+++}$.
By the mutation in directions 1, $2$ and $3$, we have the $C$-matrices
\begin{align}
\begin{pmatrix}
-1 & 0 & 0
\\
0 & 1 & 0 
\\
0 & 0 & 1
\end{pmatrix},
\quad
\begin{pmatrix}
1 & 0 & 0
\\
3 & -1 & 0 
\\
0 & 0 & 1
\end{pmatrix},
\quad
\begin{pmatrix}
1 & 0 & 0
\\
0 & 1 & 0
\\
2 & 2 & -1
\end{pmatrix}.
\end{align}
They correspond to the triangles that are adjacent to $O_{+++}$ in the picture.
We used Quiver Mutation \cite{Keller08c}
to calculate $C$-matrices.
They are the unshaded triangles in Figures \ref{fig:type1}
and \ref{fig:type1-closeup}.
Meanwhile, the shaded region is not explored.
In Figure \ref{fig:type1-closeup},
the fractal nature of the boundary of the shaded region already emerges.
By continuing the procedure, the boundary becomes more jagged,
but the region does not shrink drastically.
The dark-shaded region is
the union of the pullbacks of the Badlands of rank 2 scattering diagrams of infinite type
\cite{Muller15} around the vertex $v_i$ $(i=1$, 2, 3) representing the ray $\bbR_{\geq 0}\bfe_{i}$.
From the scattering diagram point of view,
the region is outside the $G$-fan
because
each boundary wall has an irrational normal vector.
(This is the proof of the incompleteness of the above $G$-fan based on the scattering diagram method
mentioned in Remark \ref{rem:incomplete1}.
But we did not use this fact to draw the pictures.)

(5) The red edges in Figure \ref{fig:type1-closeup} are drawn to 
clarify the structure of the boundary of the $G$-fan.
They are ``walls'' in the cluster scattering diagram
 in \cite{Gross14, Nakanishi22a}.
Continuing the fractal process is possible, in principle. However,
already in the next step, it becomes very hard to depict it because there will be so many edges to draw.

\clearpage

\begin{figure}
\includegraphics[width=320pt]{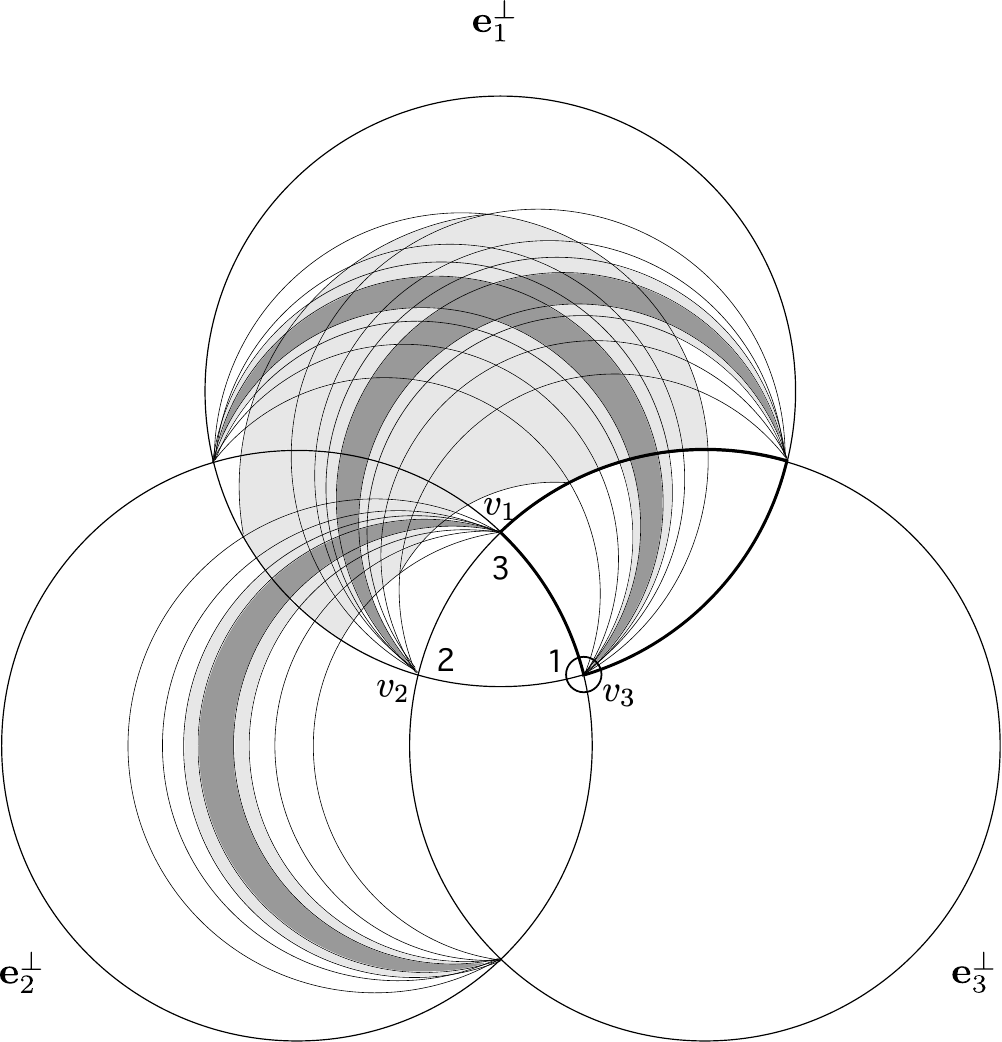}
\\
{\small
$
\tilde B=
\begin{pmatrix}
0 & - 2 & -4
\\
3 & 0 & -6
\\
2 & 2 & 0
\end{pmatrix}
$
}
\hskip30pt
\raisebox{-25pt}[45pt]
{
\begin{tikzpicture}[scale=0.8]
\draw (0,0) circle[radius=0.1];
\draw (2,0) circle[radius=0.1];
\draw (1,1.73) circle[radius=0.1];
\draw[{Latex[length=8pt,width=4pt]}-] (0.2,0)--(1.8,0);
\draw[{Latex[length=8pt,width=4pt]}-] (0.1,0.2)--(0.9,1.53);
\draw[{Latex[length=8pt,width=4pt]}-] (1.9,0.2)--(1.1,1.53);
\draw (1,0) node [below] {\small $(3,2)$};
\draw (0.4,0.85) node [left] {\small $(2,4)$};
\draw (1.6,0.85) node [right] {\small $(2,6)$};
\draw (0,0) node [below left] {\small $1$};
\draw (2,0) node [below right] {\small $2$};
\draw (1.1,1.73) node [right] {\small $3$};
\end{tikzpicture}
}
\caption{``The wing''.
The $G$-fan of Example \ref{ex:type1} for Type 1.
In the region $O_{+-+}$,
 every edge stemming from
the encircled vertex $v_3=\bbR_{\geq 0}\bfe_{3}$
reaches the boundary $\bfe_3^{\perp}$.
This visualizes the formula \eqref{eq:type1g1}.
The close-up of the region $H_1^{+}$
with a more detailed boundary is given in Figure \ref{fig:type1-closeup}.
The inscribed numbers $3$, $2$, and $1$ represent the types of the
vertices $v_1$, $v_2$, and $v_1$, respectively. They will be used in Section \ref{sec:global1}.
}
\label{fig:type1}
\end{figure}
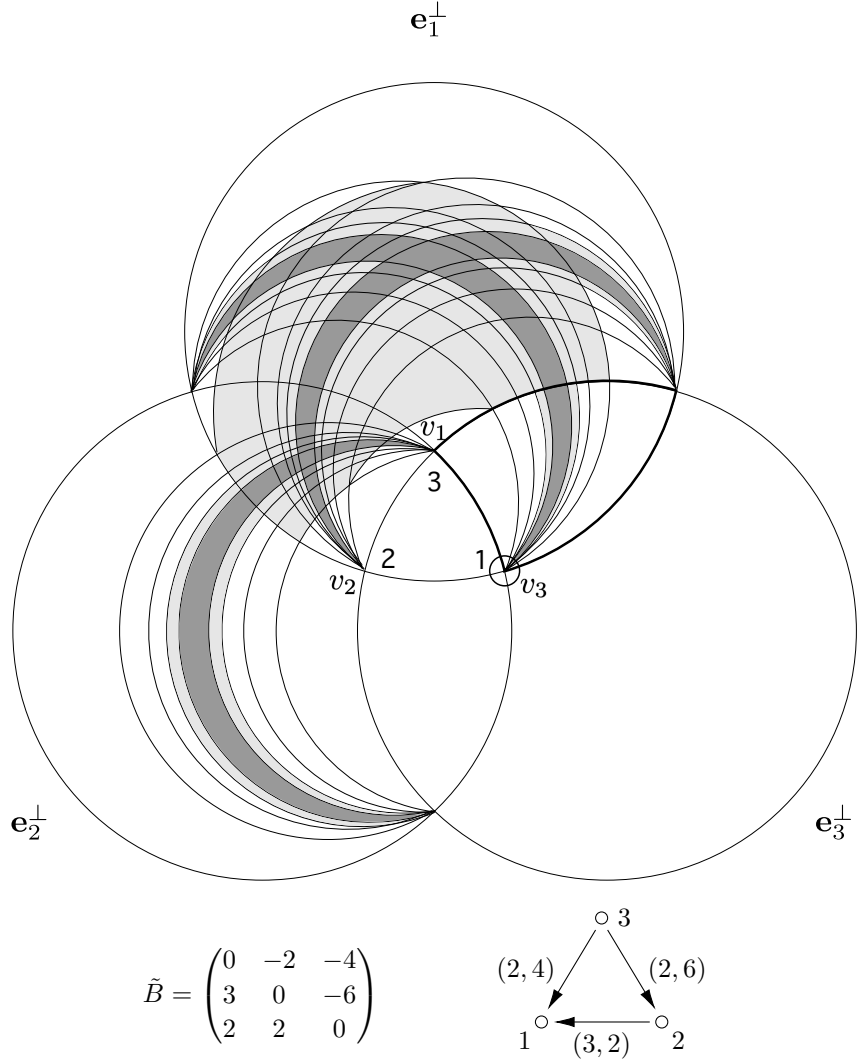

\begin{figure}
\includegraphics[width=350pt]{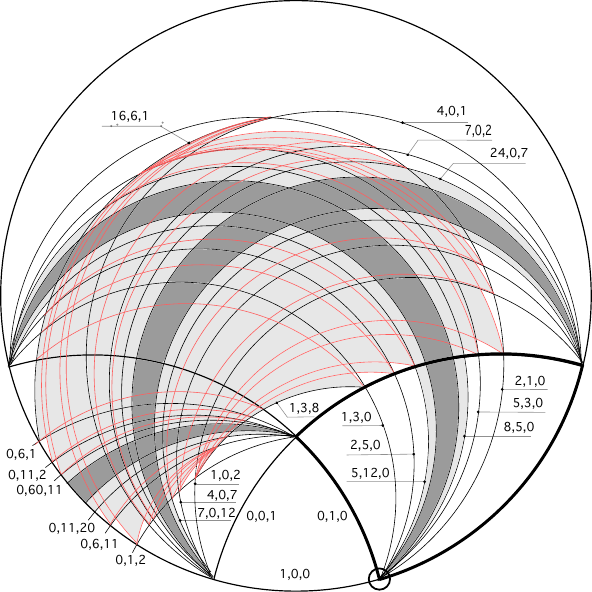}
\caption{
Close-up of the region $H_1^{+}$ in Figure \ref{fig:type1}
with a more detailed boundary.
The integer vector attached to each black edge represents its normal vector (the $c$-vector).
The added red edges (``walls'' in the cluster scattering diagram) clarify the structure of the boundary.
}
\label{fig:type1-closeup}
\end{figure}

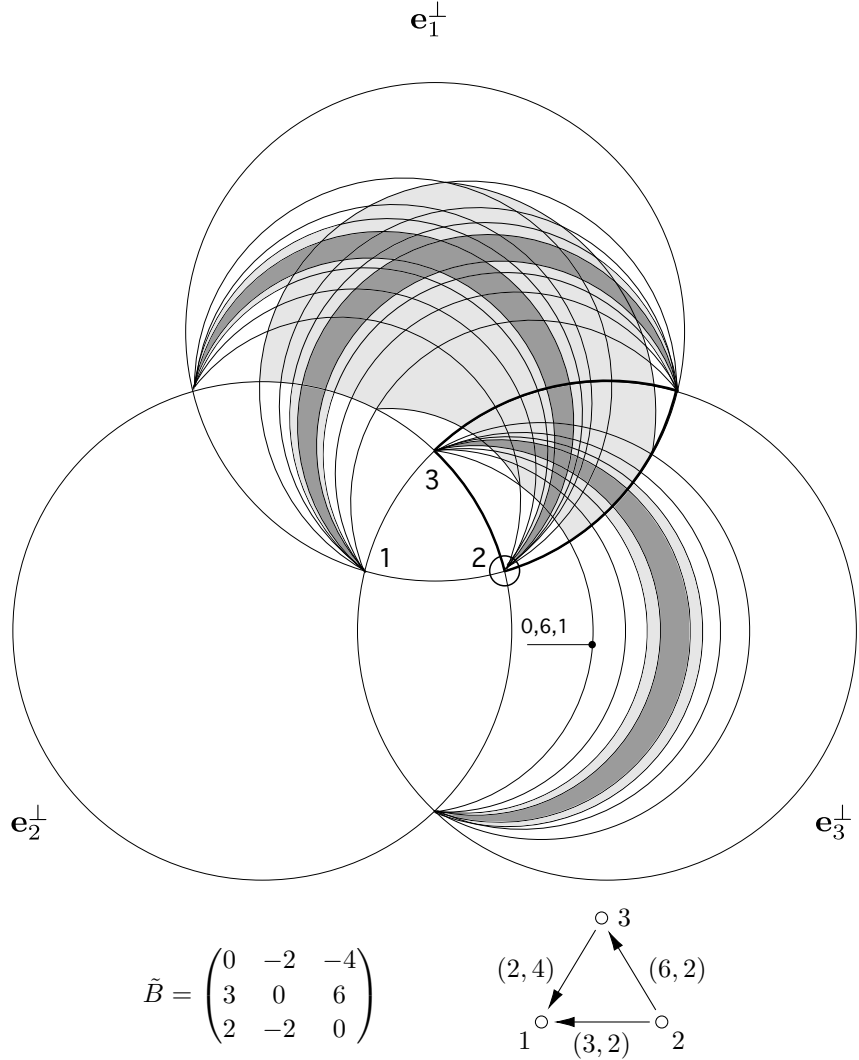
\begin{figure}
\includegraphics[width=320pt]{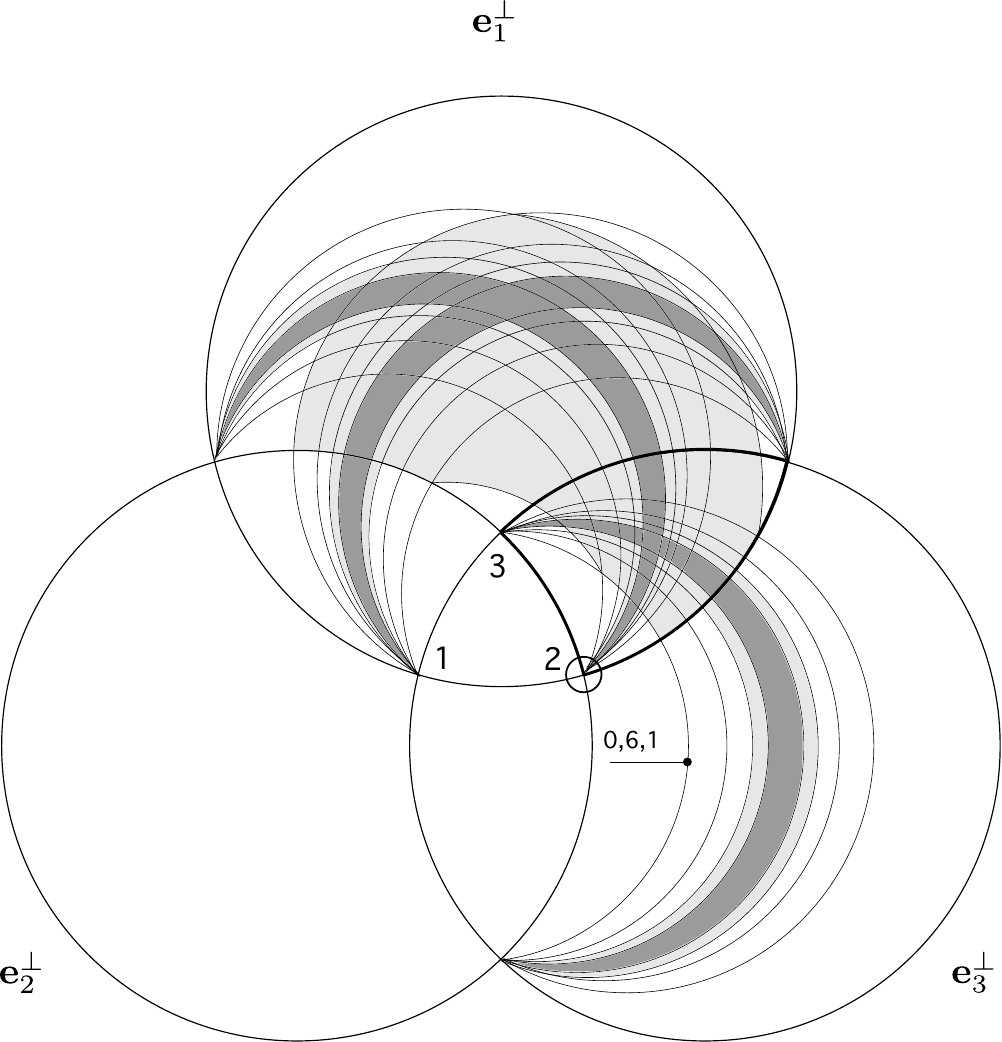}
\\
{
\small
$
\tilde B=
\begin{pmatrix}
0 & - 2 & -4
\\
3 & 0 & 6
\\
2 & -2 & 0
\end{pmatrix}
$
}
\hskip30pt
\raisebox{-25pt}[45pt]
{
\begin{tikzpicture}[scale=0.8]
\draw (0,0) circle[radius=0.1];
\draw (2,0) circle[radius=0.1];
\draw (1,1.73) circle[radius=0.1];
\draw[{Latex[length=8pt,width=4pt]}-] (0.2,0)--(1.8,0);
\draw[{Latex[length=8pt,width=4pt]}-] (0.1,0.2)--(0.9,1.53);
\draw[-{Latex[length=8pt,width=4pt]}] (1.9,0.2)--(1.1,1.53);
\draw (1,0) node [below] {\small $(3,2)$};
\draw (0.4,0.85) node [left] {\small $(2,4)$};
\draw (1.6,0.85) node [right] {\small $(6,2)$};
\draw (0,0) node [below left] {\small $1$};
\draw (2,0) node [below right] {\small $2$};
\draw (1.1,1.73) node [right] {\small $3$};
\end{tikzpicture}
}
\caption{
The $G$-fan of Example \ref{ex:type2} for Type 2.
In the region $O_{+-+}$,
 every edge stemming from
the vertex $v_3=\bbR_{\geq 0}\bfe_{3}$
crosses the edge whose normal vector is $(0,6,1)$.
The vectors $(\alpha,\beta,-2\beta)$ and $(0,6,1)$
are orthogonal for the inner product \eqref{eq:inner1}.
Therefore,
this visualizes the formulas \eqref{eq:type2g1}
and \eqref{eq:type2g2}.
Meanwhile, the global patterns of the $G$-fan is essentially
the same as Figure \ref{fig:type1} by ignoring the difference of the normal vectors,
and it is obtained from it by the horizontal reflection.
}
\label{fig:type2}
\end{figure}

\begin{figure}
\centering
\includegraphics[width=320pt]{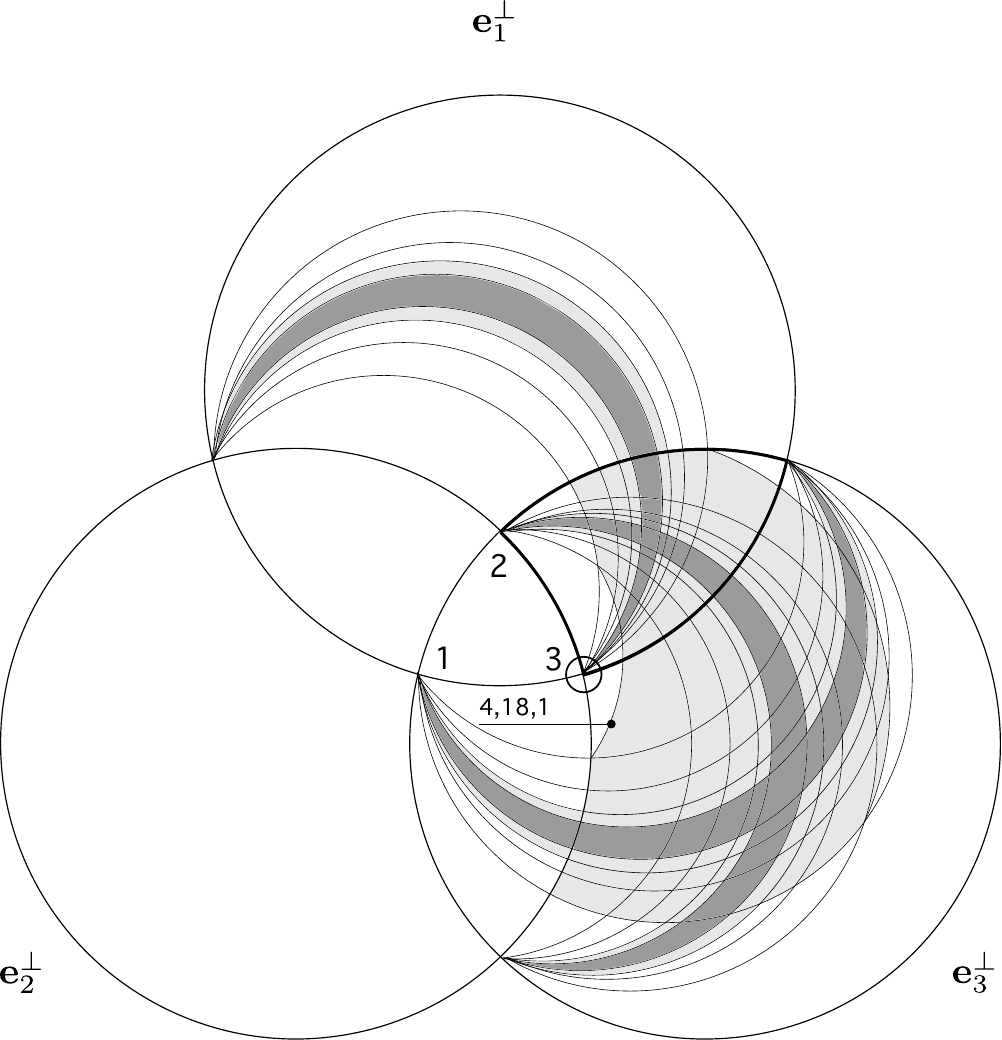}
\\
{
\small
$
\tilde B=
\begin{pmatrix}
0 & - 2 & 4
\\
3 & 0 & 6
\\
-2 & -2 & 0
\end{pmatrix}
$
}
\hskip30pt
\raisebox{-25pt}[45pt]
{
\begin{tikzpicture}[scale=0.8]
\draw (0,0) circle[radius=0.1];
\draw (2,0) circle[radius=0.1];
\draw (1,1.73) circle[radius=0.1];
\draw[{Latex[length=8pt,width=4pt]}-] (0.2,0)--(1.8,0);
\draw[-{Latex[length=8pt,width=4pt]}] (0.1,0.2)--(0.9,1.53);
\draw[-{Latex[length=8pt,width=4pt]}] (1.9,0.2)--(1.1,1.53);
\draw (1,0) node [below] {\small $(3,2)$};
\draw (0.4,0.85) node [left] {\small $(4,2)$};
\draw (1.6,0.85) node [right] {\small $(6,2)$};
\draw (0,0) node [below left] {\small $1$};
\draw (2,0) node [below right] {\small $2$};
\draw (1.1,1.73) node [right] {\small $3$};
\end{tikzpicture}
}
\caption{
The $G$-fan of Example \ref{ex:type3} for Type 3.
The vectors $(\alpha,\beta,-2 \alpha -6\beta)$ and $(4,18,1)$
are orthogonal.
Therefore,
this visualizes the formulas \eqref{eq:type3g1}
and \eqref{eq:type3g2}.
Meanwhile, the global patterns of the $G$-fan is essentially
the same as Figure \ref{fig:type1}
by ignoring the difference of the normal vectors,
and it is obtained from the rotation of Figure \ref{fig:type1}.
}
\label{fig:type3}
\end{figure}
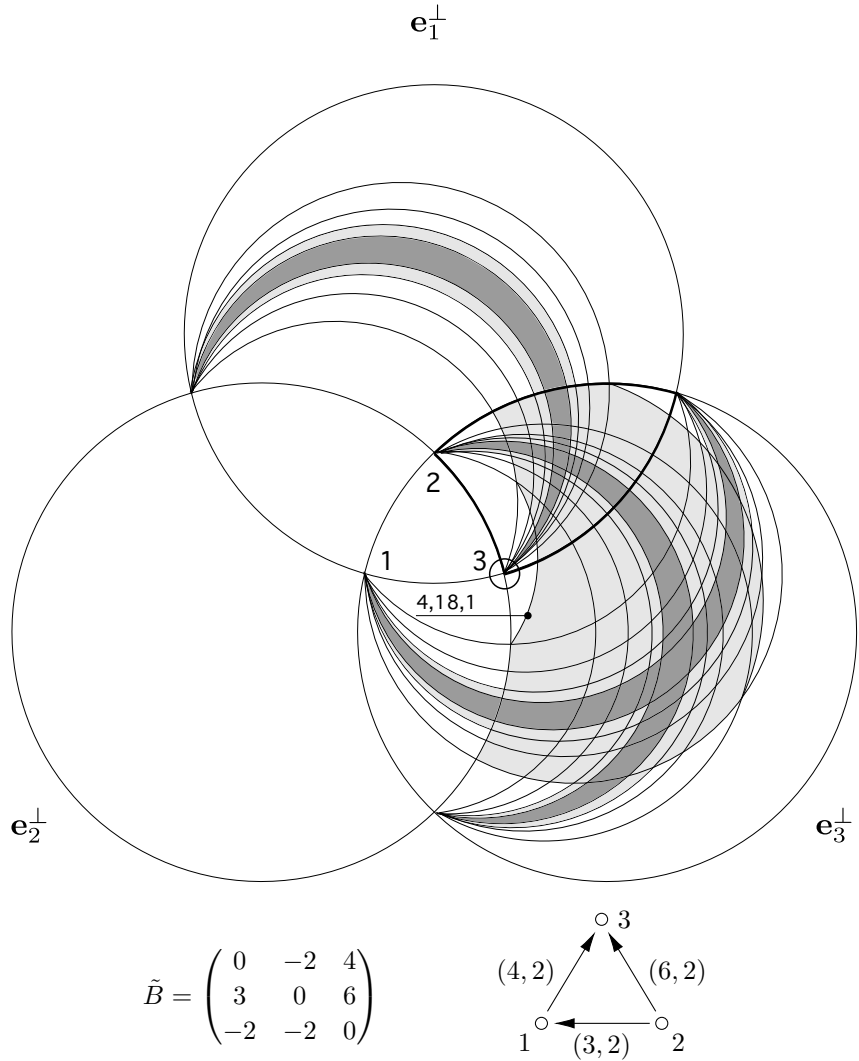

\begin{figure}
\centering
\includegraphics[width=320pt]{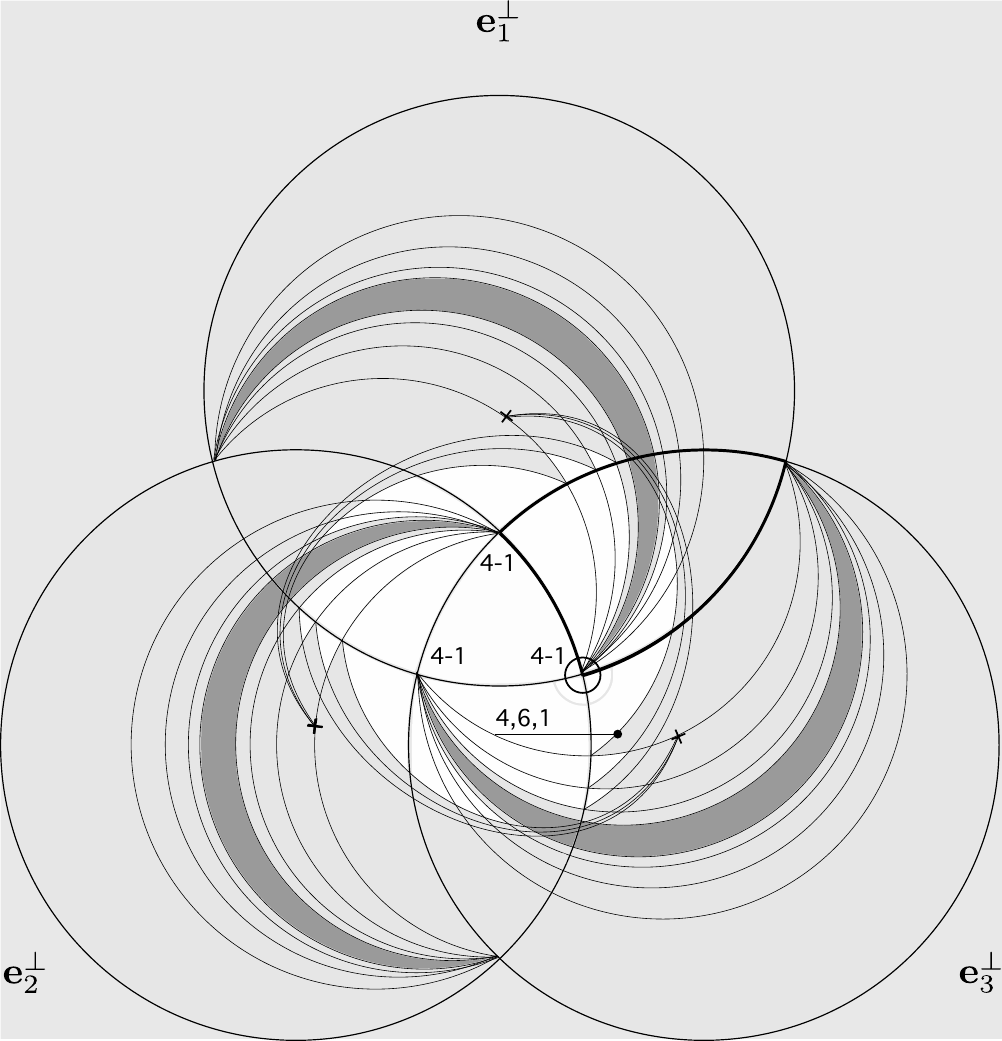}
\\
{
\small
$
\tilde B=
\begin{pmatrix}
0 & - 2 & 4
\\
3 & 0 & -6
\\
-2 & 2 & 0
\end{pmatrix}
$
}
\hskip30pt
\raisebox{-25pt}[45pt]
{
\begin{tikzpicture}[scale=0.8]
\draw (0,0) circle[radius=0.1];
\draw (2,0) circle[radius=0.1];
\draw (1,1.73) circle[radius=0.1];
\draw[{Latex[length=8pt,width=4pt]}-] (0.2,0)--(1.8,0);
\draw[-{Latex[length=8pt,width=4pt]}] (0.1,0.2)--(0.9,1.53);
\draw[{Latex[length=8pt,width=4pt]}-] (1.9,0.2)--(1.1,1.53);
\draw (1,0) node [below] {\small $(3,2)$};
\draw (0.4,0.85) node [left] {\small $(4,2)$};
\draw (1.6,0.85) node [right] {\small $(2,6)$};
\draw (0,0) node [below left] {\small $1$};
\draw (2,0) node [below right] {\small $2$};
\draw (1.1,1.73) node [right] {\small $3$};
\end{tikzpicture}
}
\caption{
``The pinwheel''.
The $G$-fan of Example \ref{ex:type41} for Type 4-1.
The vectors $(\alpha,\beta,-2 \alpha -2\beta)$ and $(4,6,1)$
are orthogonal.
Therefore,
this visualizes the formulas \eqref{eq:type41g1}
and \eqref{eq:type41g2}.
The global patterns of the $G$-fan is the same as the Markov quiver 
in \cite[Fig.~1]{Fock11}, \cite[Fig.~13]{Reading19}.
We have not yet determined the exact location of the ending points of
 arcs with the symbol $\times$.
 (The same applies to other figures.)
}
\label{fig:type4-1}
\end{figure}
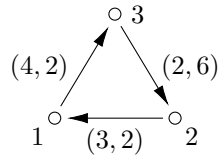

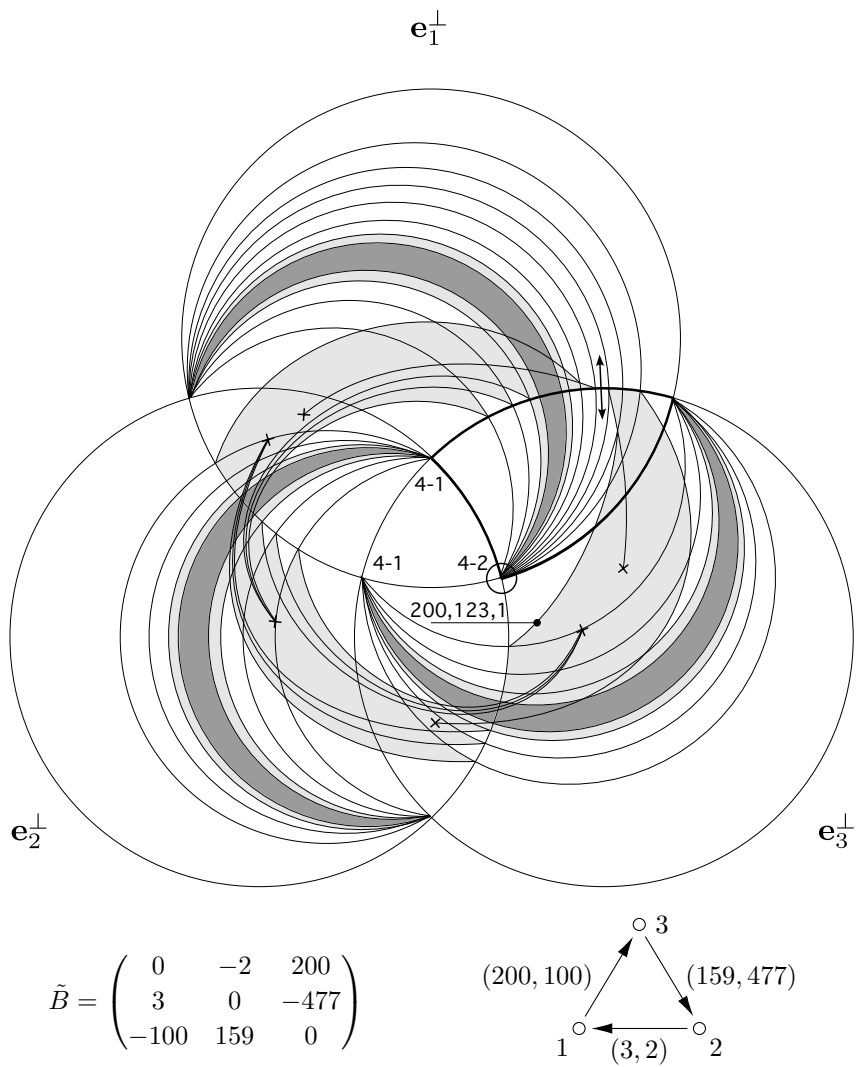
\begin{figure}
\centering
\includegraphics[width=320pt]{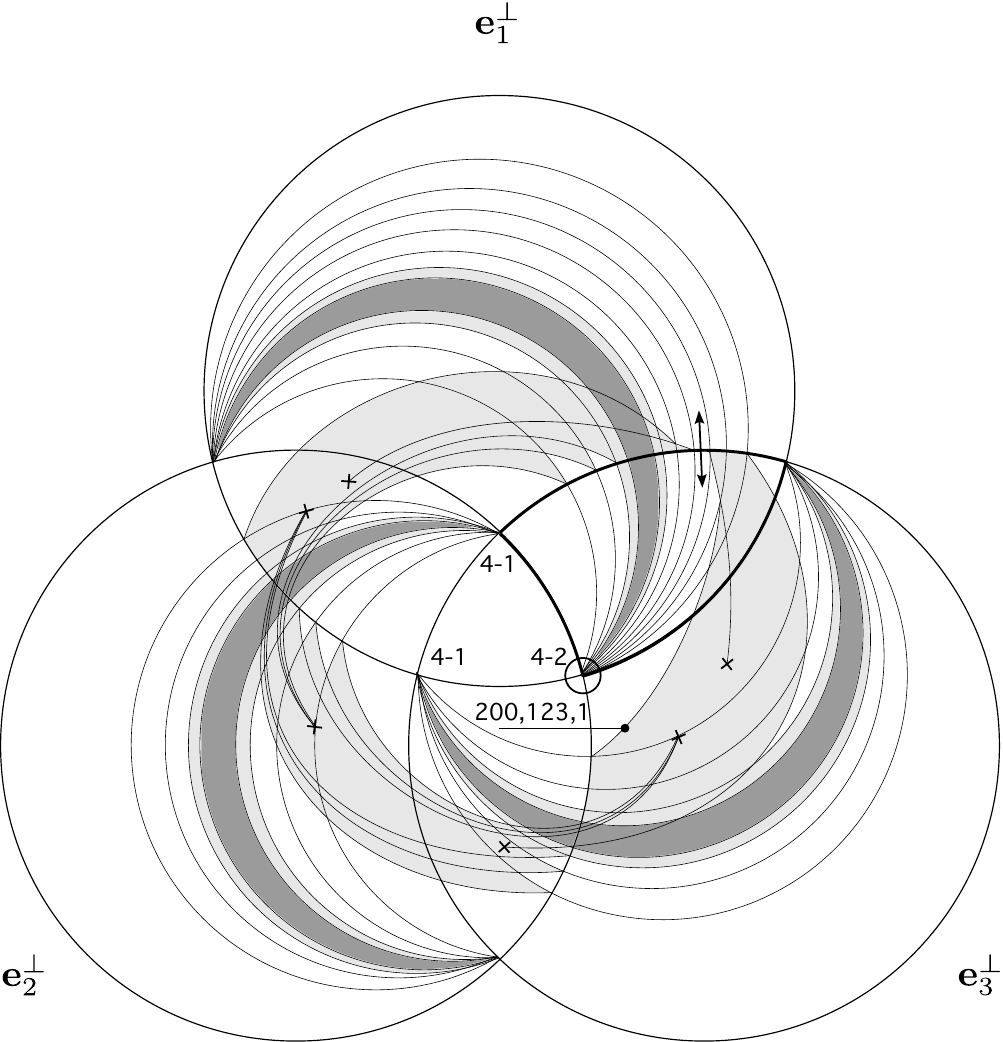}
\\
{
\small
$\tilde B=
\begin{pmatrix}
0 & - 2 & 200
\\
3 & 0 & -477
\\
-100 & 159 & 0
\end{pmatrix}
$
}
\hskip30pt
\raisebox{-25pt}[45pt]
{
\begin{tikzpicture}[scale=0.8]
\draw (0,0) circle[radius=0.1];
\draw (2,0) circle[radius=0.1];
\draw (1,1.73) circle[radius=0.1];
\draw[{Latex[length=8pt,width=4pt]}-] (0.2,0)--(1.8,0);
\draw[-{Latex[length=8pt,width=4pt]}] (0.1,0.2)--(0.9,1.53);
\draw[{Latex[length=8pt,width=4pt]}-] (1.9,0.2)--(1.1,1.53);
\draw (1,0) node [below] {\small $(3,2)$};
\draw (0.4,0.85) node [left] {\small $(200,100)$};
\draw (1.6,0.85) node [right] {\small $(159,477)$};
\draw (0,0) node [below left] {\small $1$};
\draw (2,0) node [below right] {\small $2$};
\draw (1.1,1.73) node [right] {\small $3$};
\end{tikzpicture}
}
\caption{
``The outside gate''.
The $G$-fan of Example \ref{ex:type421} for Type 4-2 with $N=3$.
The vectors $(\alpha,\beta,-100 \alpha -41\beta)$ and $(200,123,1)$
are orthogonal.
Therefore,
this visualizes the formulas \eqref{eq:type41g1},
\eqref{eq:type41g2}, and \eqref{eq:type421g1}.
The global pattern of the $G$-fan is very different from Figure \ref{fig:type4-1}.
There is a ``secret gate'' from $H_3^+$ to $H_3^-$
as depicted by an arrow.
As a result, the $G$-fan
reaches the negative orthant
as expected from Proposition \ref{prop:negative1}.
}
\label{fig:type4-2}
\end{figure}

\begin{figure}
\centering
\includegraphics[width=320pt]{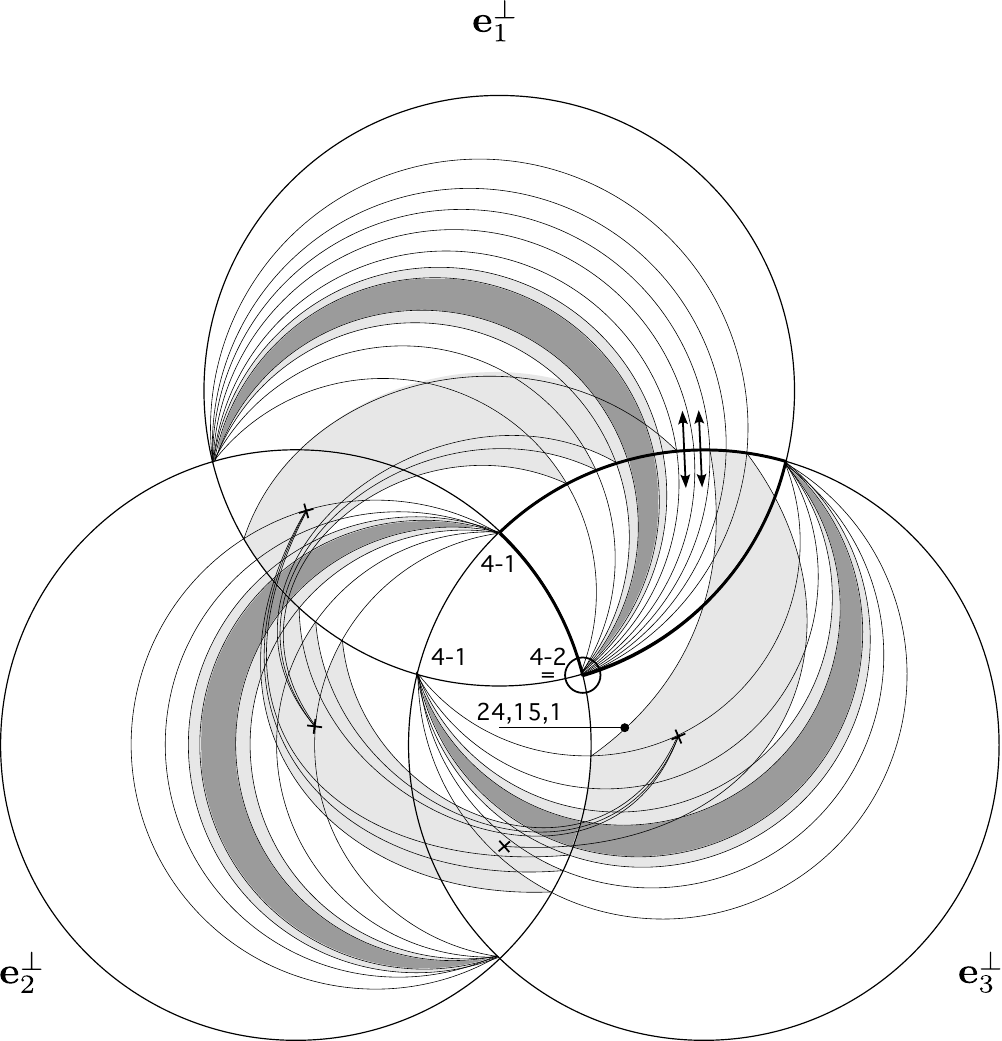}
\\
{
\small
$\tilde B=
\begin{pmatrix}
0 & - 2 & 24
\\
3 & 0 & -57
\\
-12 & 19 & 0
\end{pmatrix}
$
}
\hskip30pt
\raisebox{-25pt}[45pt]
{
\begin{tikzpicture}[scale=0.8]
\draw (0,0) circle[radius=0.1];
\draw (2,0) circle[radius=0.1];
\draw (1,1.73) circle[radius=0.1];
\draw[{Latex[length=8pt,width=4pt]}-] (0.2,0)--(1.8,0);
\draw[-{Latex[length=8pt,width=4pt]}] (0.1,0.2)--(0.9,1.53);
\draw[{Latex[length=8pt,width=4pt]}-] (1.9,0.2)--(1.1,1.53);
\draw (1,0) node [below] {\small $(3,2)$};
\draw (0.4,0.85) node [left] {\small $(24,12)$};
\draw (1.6,0.85) node [right] {\small $(19,57)$};
\draw (0,0) node [below left] {\small $1$};
\draw (2,0) node [below right] {\small $2$};
\draw (1.1,1.73) node [right] {\small $3$};
\end{tikzpicture}
}
\caption{
``The wide outside gate''.
A variant of Example \ref{ex:type421} for Type 4-2 with $N=3$.
In the first inequality of  \eqref{eq:case42N},
the equality is attained with $N=3$.
The gate is widened.
Also, there are some differences around the gate
from Figure \ref{fig:type4-2}.
}
\label{fig:type4-2-b}
\end{figure}
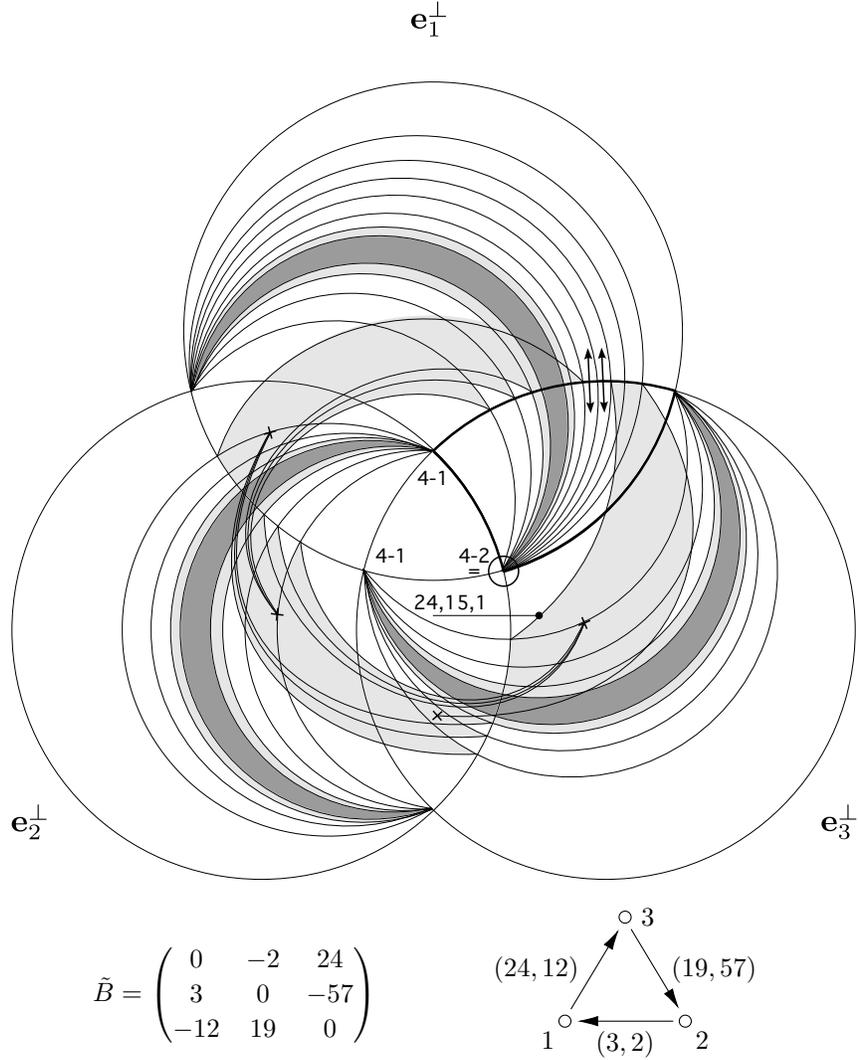

\begin{figure}
\centering
\includegraphics[width=320pt]{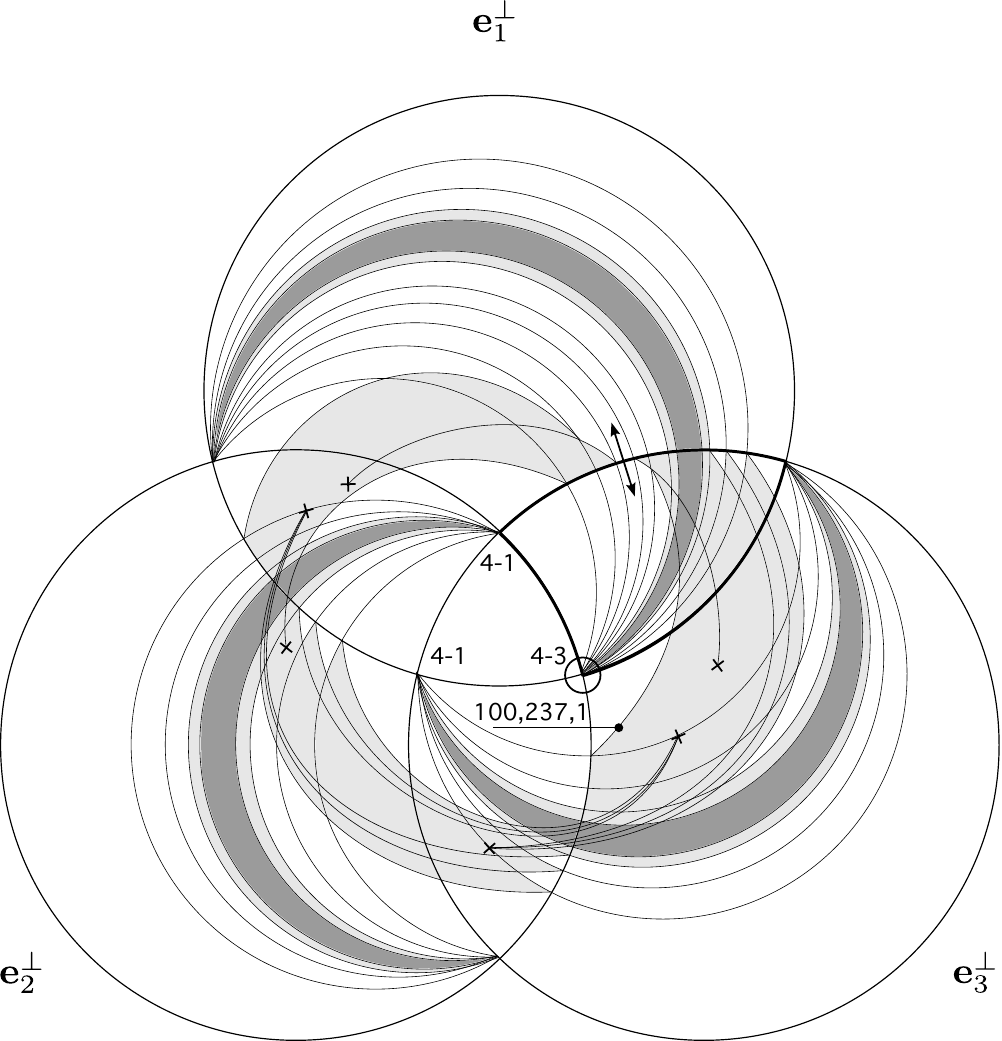}
\\
{
\small
$
\tilde B=
\begin{pmatrix}
0 & - 2 & 100
\\
3 & 0 & -63
\\
-50 & 21 & 0
\end{pmatrix}
$
}
\hskip30pt
\raisebox{-25pt}[45pt]
{
\begin{tikzpicture}[scale=0.8]
\draw (0,0) circle[radius=0.1];
\draw (2,0) circle[radius=0.1];
\draw (1,1.73) circle[radius=0.1];
\draw[{Latex[length=8pt,width=4pt]}-] (0.2,0)--(1.8,0);
\draw[-{Latex[length=8pt,width=4pt]}] (0.1,0.2)--(0.9,1.53);
\draw[{Latex[length=8pt,width=4pt]}-] (1.9,0.2)--(1.1,1.53);
\draw (1,0) node [below] {\small $(3,2)$};
\draw (0.4,0.85) node [left] {\small $(100,50)$};
\draw (1.6,0.85) node [right] {\small $(21,63)$};
\draw (0,0) node [below left] {\small $1$};
\draw (2,0) node [below right] {\small $2$};
\draw (1.1,1.73) node [right] {\small $3$};
\end{tikzpicture}
}
\caption{
``The  inside gate''.
The $G$-fan of Example \ref{ex:type431} for Type 4-3 with $N=3$.
The vectors $(\alpha,\beta,-50 \alpha -79\beta)$ and $(200,237,1)$
are orthogonal.
Therefore,
this visualizes the formulas \eqref{eq:type41g1},
\eqref{eq:type41g2}, and \eqref{eq:type431g2}.
The global pattern of the $G$-fan is
similar (dual) to Figure \ref{fig:type4-2}.
}
\label{fig:type4-3}
\end{figure}
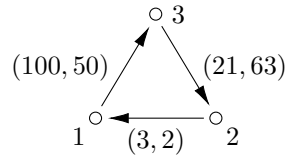

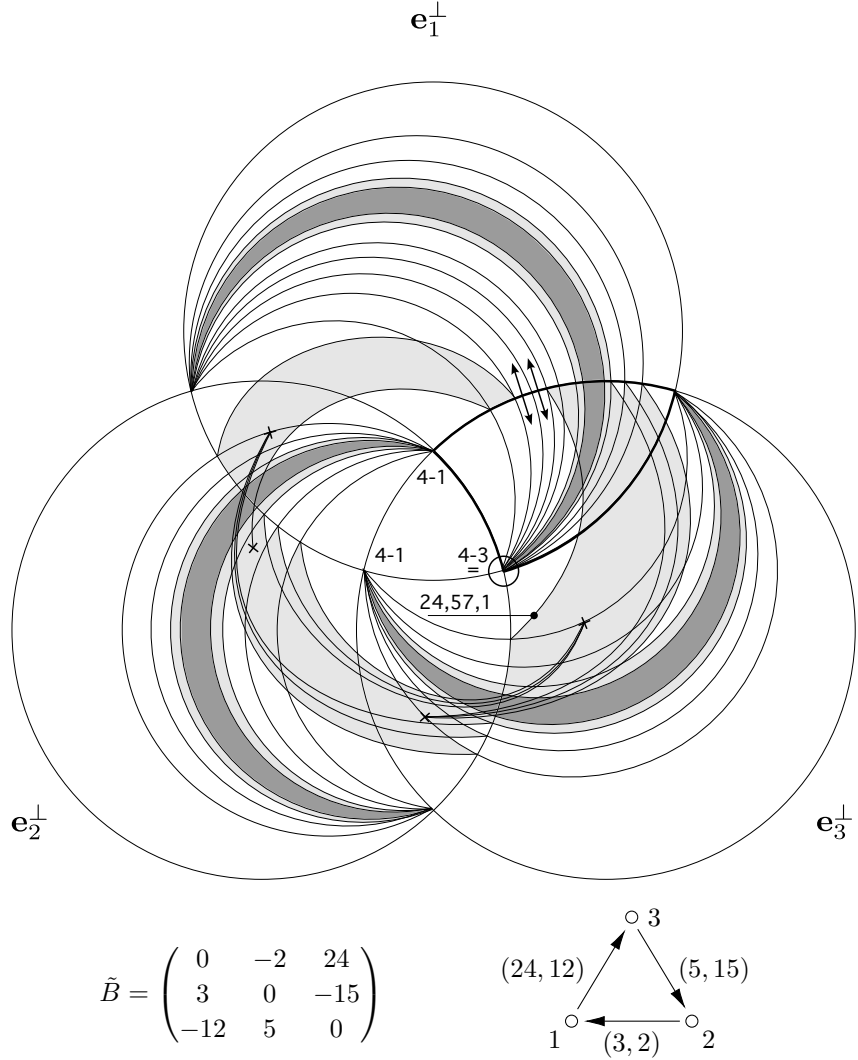
\begin{figure}
\centering
\includegraphics[width=320pt]{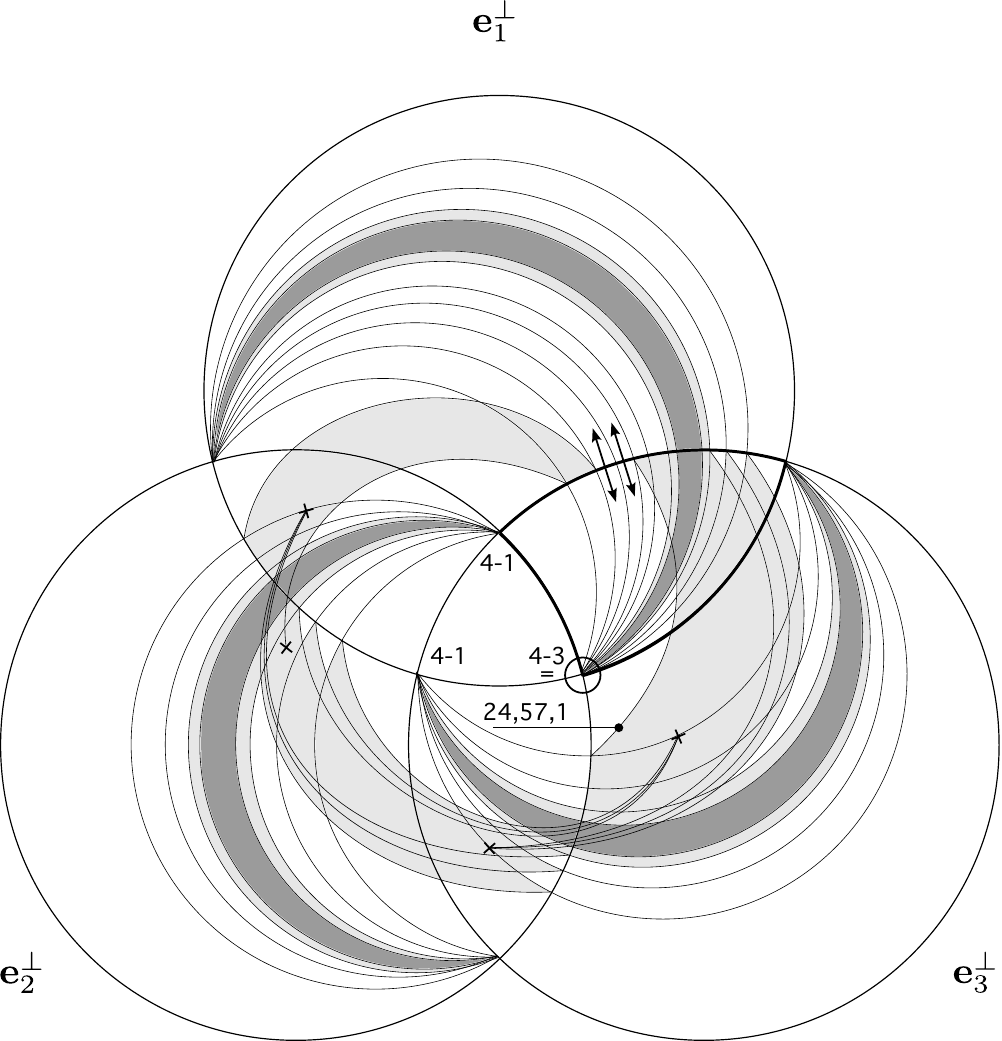}
\\
{
\small
$
\tilde B=
\begin{pmatrix}
0 & - 2 & 24
\\
3 & 0 & -15
\\
-12 & 5 & 0
\end{pmatrix}
$
}
\hskip30pt
\raisebox{-25pt}[45pt]
{
\begin{tikzpicture}[scale=0.8]
\draw (0,0) circle[radius=0.1];
\draw (2,0) circle[radius=0.1];
\draw (1,1.73) circle[radius=0.1];
\draw[{Latex[length=8pt,width=4pt]}-] (0.2,0)--(1.8,0);
\draw[-{Latex[length=8pt,width=4pt]}] (0.1,0.2)--(0.9,1.53);
\draw[{Latex[length=8pt,width=4pt]}-] (1.9,0.2)--(1.1,1.53);
\draw (1,0) node [below] {\small $(3,2)$};
\draw (0.4,0.85) node [left] {\small $(24,12)$};
\draw (1.6,0.85) node [right] {\small $(5,15)$};
\draw (0,0) node [below left] {\small $1$};
\draw (2,0) node [below right] {\small $2$};
\draw (1.1,1.73) node [right] {\small $3$};
\end{tikzpicture}
}
\caption{
``The wide inside gate ''.
A variant of Example \ref{ex:type431} for Type 4-3 with $N=3$.
In the first inequality of  \eqref{eq:case43N},
the equality is attained with $N=3$.
}
\label{fig:type4-3-b}
\end{figure}

\clearpage
\section{Prototypical examples of global patterns of rank 3 $G$-fan}
\label{sec:global1}

In this section, we present some experimental results.

In the previous section, we saw that the behavior of a rank 3 $G$-fan $\Delta(B)$ around the vertex $v_3$, which represents the ray
$\bbR_{\geq 0} \bfe_3$ in the stereographic projection, are classified
by the type of the initial matrix $B$ in Section \ref{sec:lemv1}.
If $a$ and $b$ in \eqref{eq:B0} are negative,
we assign the type by changing the roles of the indices 1 and 2.
(Namely, in this case, we assign the type by looking at the image of $\Delta(B)$ by reflection.)
We identify it as the \emph{type of the vertex $v_3$}.

Now, let us assume that $\Delta(B)$ is of totally-infinite type as defined in Section \ref{sec:intro1}.
Let us call the vertices $v_i$ ($i=1$, 2, 3)  in $\Delta(B)$, representing the ray $\bbR_{\geq 0} \bfe_3$ in the stereographic projection,
the \emph{elementary vertices}.
We assign the types to $v_1$ and $v_2$ in the same way as $v_3$ by changing the indices cyclically.
For a rank 3 $G$-fan $\Delta(B)$ of totally infinite type,
we call the triplet of the types of the elementary vertices
the \emph{type of $\Delta(B)$}.

\begin{ex}
In Figure \ref{fig:type1}, the vertex $v_3$ is clearly of Type 1.
The vertex $v_1$ is of Type 3, because it matches the vertex $v_3$ in Figure \ref{fig:type3} by rotation.
The vertex $v_2$ is of Type 2, because it matches the vertex $v_3$ in Figure \ref{fig:type3} by reflection.
Thus, the type of the $G$-fan is $(3,2,1)$.
\end{ex}

Clearly, the type of $\Delta(B)$ controls the local behaviors (the \emph{local pattern})
of $\Delta(B)$ around the elementary vertices.
On the other hand,
looking at the pictures in the previous section, we observed that 
the local pattern of a rank 3 $G$-fan also correlates with its global behavior (the \emph{global pattern}).
Based on this observation, we present prototypical examples of global patterns of $\Delta(B)$
following the classification by the type of $\Delta(B)$ with some additional data.

\subsection{Cyclic presentation of skew-symmetrizable matrix}
It is convenient to write a $3\times 3$ skew-symmetrizable matrix $B$ in a cyclic way
\begin{align}
\label{eq:Bcyclic1}
B =
\begin{pmatrix}
0 & - p_3' & p_2
\\
p_3 & 0 &- p_1'
\\
-p_2' & p_1& 0
\end{pmatrix}
\quad
(p_i,\, p'_i \in \bbZ).
\end{align}
Let $D=\mathrm{diag}(d_1,d_2,d_3)$ be a skew-symmetrizer of $B$.
We have
\begin{align}
d_3 p_1=d_2 p'_1,
\quad
d_1 p_2=d_3 p'_2,
\quad
d_2 p_3=d_1p'_3.
\end{align}
Then, the skew-symmetrizable condition is rephrased as
\cite[Lemma 7.4]{Fomin03a}
\begin{align}
p_1p_2p_3=p'_1 p'_2 p'_3,
\quad
p_i p'_i \geq  0,
\quad
\text{and}\
p_i=0\ \Longleftrightarrow 'p'_i=0.
\end{align}
Recall that $B$ is {of totally-infinite type} if $p_ip'_i\geq 4$ for any $i$.
We say that
\begin{itemize}
\item
$B$ is \emph{cyclic} (resp.~\emph{acyclic}) if $p_i>0$ for any $i$ or $p_i<0$ for any $i$ (resp.~ otherwise).
\item
$B$ is \emph{cluster-cyclic}  if any matrix which is mutation-equivalent to $B$ is cyclic (including $B$ itself).
\item
$B$ is \emph{cluster-acyclic} if it is not cluster-cyclic, namely, if $B$ is mutation-equivalent to an acyclic matrix,
\end{itemize}
Following \cite{Beineke06},
for a cyclic matrix $B$, we define  the \emph{Markov constant}
\begin{align}
C(B)=p_1p'_1 + 
p_2p'_2 + 
p_3p'_3 - |p_1p_2 p_3|.
\end{align}
The following fact is proved by  \cite{Beineke06} for the skew-symmetric case
and by \cite{Akagi24} for the skew-symmetrizable case.

\begin{thm}[{\cite{Beineke06,Akagi24}}]
\label{thm:cc1}
For a cyclic matrix $B$,
$B$ is cluster-cyclic if and only if $B$ is of totally-infinite type and
$C(B)\leq 4$.
\end{thm}

Below, for a rank 3 $G$-fan $\Delta(B)$, 
we assume that
\begin{align}
\label{cond:B1}
\text{the matrix $B$ is of totally-infinite type, and $p_3>0$.}
\end{align}
If $p_3<0$, then we change the indices 1 and 2 to have $p_3>0$.
So, we do not lose generality.
All pictures of the $G$-fans presented in Section \ref{sec:pictures1} also obey this condition.

\subsection{Acyclic case}
Suppose that $B$ is acyclic.
Then, the vertex $v_1$ is of Types 1, 2, or 3.
Moreover, the types of other vertices $v_2$ and $v_3$
are uniquely determined by the type of $v_1$ as in the examples in Figures \ref{fig:type1}, \ref{fig:type2}, and \ref{fig:type3}.
So, the roles of Types 1, 2, and 3 are mutual.
We observe that both local and global patterns of the $G$-fans in Figures \ref{fig:type1}, \ref{fig:type2}, and \ref{fig:type3} are the same
up to rotations and reflections.
Here, ``the same'' means the pictures are indistinguishable if we ignore the data of normal vectors.
It is natural to expect that the global pattern is the same (up to rotations and reflections) for any $G$-fan $\Delta(B)$
with the following exception:
If there is an exchange matrices $B_t$ with $|b_{ij;t} b_{ji;t}|<4$ for some pair $i\neq j$ (finite pair),
then the region of $\Delta(B)$ is slightly expanded so that the corresponding vertex in $\Delta(B)$ becomes an inner point of $\Delta(B)$
due to the periodicity of the alternating mutations for the pair $(i,j)$.
We call it a \emph{finite degeneration} and regard it as a minor variant of the above generic global pattern.

We see in Figures \ref{fig:type1}, \ref{fig:type2}, and \ref{fig:type3} that
the negative orthant $O_{---}$ is a $G$-cone.
This is true in general when $B$ is acyclic.
For example, when the vertex 3 is Type 1 (e.g., Figure \ref{fig:type1}),
we have a mutation sequence of $C$-matrices starting from the initial $C$-matrix
\begin{align}
\label{eq:gr1}
\begin{pmatrix}
1 & 0 & 0
\\
0 & 1 & 0
\\
0 & 0 & 1
\end{pmatrix}
\overset{1}{\mapsto}
\begin{pmatrix}
-1 & 0 & 0
\\
0 & 1 & 0
\\
0 & 0 & 1
\end{pmatrix}
\overset{2}{\mapsto}
\begin{pmatrix}
-1 & 0 & 0
\\
0 & -1 & 0
\\
0 & 0 & 1
\end{pmatrix}
\overset{3}{\mapsto}
\begin{pmatrix}
-1 & 0 & 0
\\
0 & -1 & 0
\\
0 & 0 & -1
\end{pmatrix}
\end{align}
because the factor $[\varepsilon_{k;t}b_{ki;t}]_+$ in \eqref{eq:cmut1} vanishes during the mutations.
Then, by the duality \eqref{eq:dual1}, 
we have the corresponding sequence of $G$-matrices
\begin{align}
\begin{pmatrix}
1 & 0 & 0
\\
0 & 1 & 0
\\
0 & 0 & 1
\end{pmatrix}
\overset{1}{\mapsto}
\begin{pmatrix}
-1 & 0 & 0
\\
0 & 1 & 0
\\
0 & 0 & 1
\end{pmatrix}
\overset{2}{\mapsto}
\begin{pmatrix}
-1 & 0 & 0
\\
0 & -1 & 0
\\
0 & 0 & 1
\end{pmatrix}
\overset{3}{\mapsto}
\begin{pmatrix}
-1 & 0 & 0
\\
0 & -1 & 0
\\
0 & 0 & -1
\end{pmatrix}
.
\end{align}
The $G$-cone for the last $G$-matrix is the negative orthant $O_{\tiny---}$.
(This fact is a special case of more general result \cite[Lemma 2.20]{Brustle12} for any skew-symmetric matrix of any rank.
Its proof can be extended to the skew-symmetrizable case by considering
 valued quivers.)

\subsection{Cyclic case}
Suppose that $B$ is cyclic.
Then, all elementary vertices are of Types 4-1, 4-2, or 4-3.
In contrast to the acyclic case, there are some varieties of global patterns of $\Delta(B)$.
The following factors are relevant:
\begin{enumerate}
\item
The type of  $\Delta(B)$.
\item
Whether $B$ is cluster-cyclic or not.
\end{enumerate}

Let us explain why the factor (2) enters.
The following fact is known.

\begin{prop}[{\cite[Thm.~3.2.1]{Muller15}}]
\label{prop:negative1}
For any cluster-acyclic skew-symmetric matrix $B$ of any rank, 
the negative orthant is a $G$-cone.
\end{prop}

The proof of \cite{Muller15} can be extended to the skew-symmetrizable case.
Also, we may avoid the scattering diagram formulation therein by using the
dual mutation of $G$-fans (e.g., \cite[Prop.~II.2.22]{Nakanishi22a}.)
On the other hand, in Figure \ref{fig:type4-1},
it is unlikely that the negative orthant is in $\Delta(B)$.
For the matrix $B$ therein, the Markov constant is $C(B)=2$. So, by Proposition \ref{thm:cc1},
$B$ is cluster-cyclic.
Thus, we expect that there is a qualitative difference between
the cluster-cyclic and cluster-acyclic cases.

Let us first consider the factor (1).
The conditions \eqref{eq:type41cond1}--\eqref{eq:type43cond1} for $v_{3}$ is rewritten as
\begin{align}
\label{eq:41cond12}
&\text{Type 4-1:}\quad
 p_1p'_1 + p_2p'_2 \leq p_1p_2p_3 ,
 \\
\label{eq:42cond12}
&\text{Type 4-2:}\quad
 p_1p'_1 + p_2p'_2 > p_1p_2p_3 ,
 \quad
 2p_1 p'_1 >p_1p_2p_3,
 \\
\label{eq:43cond12}
&\text{Type 4-3:}\quad
 p_1p'_1 + p_2p'_2 > p_1p_2p_3 ,
 \quad
 2p_1 p'_1 <p_1p_2p_3.
\end{align}
The conditions for $v_{1}$ and $v_{2}$ are given cyclically.

\begin{lem}
\label{lem:v2}
Let $B$ be a cyclic matrix satisfying the condition \eqref{cond:B1}.
\par
(1) For any cyclic matrix at least one of $v_1$, $v_2$, or $v_3$ is of Type 4-1.
\par
(2) Suppose that $v_2$ is the only elementary vertex of Type 4-1.
Then, $v_1$ is of Type 4-2 and $v_3$ is of Type 4-3.
\end{lem}
\begin{proof}
(1) Let $A=p_1p_2p_3=p'_1p'_2p'_3$.
Suppose that $v_3$ is not Type 4-1.
Then, by \eqref{eq:42cond12} and \eqref{eq:43cond12}, $2p_1p'_1> A$ or $2p_1p'_1<A$ holds.
Let us consider the former case $2p_1p'_1> A$.
Then, we have $2p_1> p'_2p'_3$ and $2p'_1> p_2p_3$.
It implies $4 p_1 p'_1 > (p_2p'_2)(p_3p'_3)$.
Multiplying $(p_2p'_2)(p_3p'_3)$, we have $4 A^2 > (p_2p'_2)^2(p_3p'_3)^2$.
Thus, we have $2A > (p_2p'_2)(p_3p'_3)$.
Noting that $p_2p'_2$, $p_3p'_3\geq 4$ and
$(4+\alpha)(4+\beta) \geq 2(4+\alpha) + 2(4+\beta)$ for $\alpha$, $\beta\geq0$,
we obtain
\begin{align}
A > \frac{(p_2p'_2)(p_3p'_3)}{2} \geq p_2p'_2+ p_3p'_3.
\end{align}
Thus, $v_1$ is of Type 4-1.
In the latter case, $2p_2 p'_2 > A$ holds.
Then, by the same argument, $v_2$ is of Type 4-1.
\par
(2) By the assumption, we have
\begin{align}
v_1: p_2 p'_2 + p_3 p'_3 > A,
\\
v_2: p_3p'_3 + p_1 p'_1 \leq A,
\\
v_3: p_1p'_2 + p_2 p'_2 > A.
\end{align}
Then, we obtain
\begin{align}
2A< p_1p'_2 + 2 p_2 p'_2 + p_3p'_3 \leq 2 p_2 p'_2 +A. 
\end{align}
Thus, $ 2 p_2 p'_2> A$, and $v_1$ is of Type 4-2.
Next,
suppose that $ 2 p_1 p'_1\geq  A$. Then, repeating the same argument as (1),
we obtain $p_2p'_2+ p_3p'_3 \leq A$. This is a contradiction.
Thus, $ 2 p_1 p'_1< A$, and  $v_3$ is of Type 4-3.
\end{proof}

Next, consider the factor (2).

\begin{lem}
\label{lem:v3}
For any cyclic matrix $B$ satisfying the condition \eqref{cond:B1},
if $C(B)\leq 4$, then all $v_i$ are of Type 4-1.
\end{lem}
\begin{proof}
Since $C(B)\leq 4$ and $p_1p'_1 \geq 4$,
we have $p_2 p'_2 + p_3 p'_3-p_1p_2p_3 \leq 0$.
Thus, $v_1$ is of Type 4-1.
The same is true for $v_2$ and $v_3$.
\end{proof}

Let $B$ a cyclic matrix $B$ satisfying \eqref{cond:B1}.
By Theorem \ref{thm:cc1} and Lemmas \ref{lem:v2} and \ref{lem:v3}, 
it is enough to consider the following cases.
(The letter C stands for ``cyclic''.)

\begin{tabular}{rl}
(C-1) & (4-1,4-1,4-1) and $C(B)\leq 4$ (cluster-cyclic).
\\
(C-2) & (4-1,4-1,4-1) and $C(B)> 4$ (cluster-acyclic).
\\
(C-3) & (4-1,4-1,4-2).
\\
(C-4) & (4-1,4-1,4-3).
\\
(C-5) & (4-2,4-1,4-3).
\end{tabular}

\par
{\bf (C-1)}: (4-1,4-1,4-1) and $C(B)\leq 4$ (the pinwheel).
This is the cluster-cyclic case.
As already discussed above,
Figure \ref{fig:type4-1} is an example
with $C(B)=2$.
Another example is the famous picture of
 the Markov quiver 
in \cite[Fig.~1]{Fock11} and \cite[Fig.~13]{Reading19}
with the exchange matrix
\begin{align}
B=
\begin{pmatrix}
0 & - 2 & 2
\\
2 & 0 & -2
\\
-2 & 2 & 0
\end{pmatrix}
,
\quad
C(B)=4.
\end{align}
Its global pattern is the same as Figure \ref{fig:type4-1}.

\par
{\bf (C-2)}: (4-1,4-1,4-1) and $C(B)> 4$ (the tunnel).
This is the cluster-acyclic case.
We give two examples in Figure \ref{fig:type6-a} and  \ref{fig:type6-b}.
Their global patterns are significantly different from Case 1 as expected.
We note that
the initial exchange matrices for these examples mutate to the acyclic matrices
\begin{align}
\label{eq:acyclic1}
\begin{pmatrix}
0 & - 2 & -2
\\
3 & 0 & 2
\\
3 & -2 & 0
\end{pmatrix}
,
\quad
\begin{pmatrix}
0 & 0 & -3
\\
0 & 0 & 1
\\
2 & -5 & 0
\end{pmatrix}
,
\end{align}
respectively, by the mutation sequence (2, 1, 2, 1, 3).
The first example is a generic one,
while the second one is a finite degeneration
with the finite pair $(1,2)$ in the second matrix in \eqref{eq:acyclic1}.
The inner and the outer regions of the $G$-fan is
connected by a ``tunnel'' whose entrance and exit correspond 
to the acyclic matrices in \eqref{eq:acyclic1}.
For example, 
in Figure  \ref{fig:type6-a}, 
the entrance and the exit are depicted by framed triangles.
The tunnel consists of  the following mutations of the acyclic quiver
for the first matrix in \eqref{eq:acyclic1}:
\vskip-35pt
\begin{align}
\label{eq:tunnel1}
\raisebox{-25pt}[45pt]
{
\begin{tikzpicture}[scale=0.5]
\draw (0,0) circle[radius=0.1];
\draw (2,0) circle[radius=0.1];
\draw (1,1.73) circle[radius=0.1];
\draw[{Latex[length=8pt,width=4pt]}-] (0.2,0)--(1.8,0);
\draw[{Latex[length=8pt,width=4pt]}-] (0.1,0.2)--(0.9,1.53);
\draw[-{Latex[length=8pt,width=4pt]}] (1.9,0.2)--(1.1,1.53);
\draw (1,0) node [below] {\small $(3,2)$};
\draw (0.4,0.85) node [left] {\small $(3,2)$};
\draw (1.6,0.85) node [right] {\small $(2,2)$};
\draw (0,0) node [below left] {\small $1$};
\draw (2,0) node [below right] {\small $2$};
\draw (1.1,1.73) node [right] {\small $3$};
\end{tikzpicture}
}
\overset {1}{\rightarrow}
\raisebox{-25pt}[45pt]
{
\begin{tikzpicture}[scale=0.5]
\draw (0,0) circle[radius=0.1];
\draw (2,0) circle[radius=0.1];
\draw (1,1.73) circle[radius=0.1];
\draw[-{Latex[length=8pt,width=4pt]}] (0.2,0)--(1.8,0);
\draw[-{Latex[length=8pt,width=4pt]}] (0.1,0.2)--(0.9,1.53);
\draw[-{Latex[length=8pt,width=4pt]}] (1.9,0.2)--(1.1,1.53);
\draw (1,0) node [below] {\small $(2,3)$};
\draw (0.4,0.85) node [left] {\small $(2,3)$};
\draw (1.6,0.85) node [right] {\small $(2,2)$};
\draw (0,0) node [below left] {\small $1$};
\draw (2,0) node [below right] {\small $2$};
\draw (1.1,1.73) node [right] {\small $3$};
\end{tikzpicture}
}
\overset {3}{\rightarrow}
\raisebox{-25pt}[45pt]
{
\begin{tikzpicture}[scale=0.5]
\draw (0,0) circle[radius=0.1];
\draw (2,0) circle[radius=0.1];
\draw (1,1.73) circle[radius=0.1];
\draw[-{Latex[length=8pt,width=4pt]}] (0.2,0)--(1.8,0);
\draw[{Latex[length=8pt,width=4pt]}-] (0.1,0.2)--(0.9,1.53);
\draw[{Latex[length=8pt,width=4pt]}-] (1.9,0.2)--(1.1,1.53);
\draw (1,0) node [below] {\small $(2,3)$};
\draw (0.4,0.85) node [left] {\small $(3,2)$};
\draw (1.6,0.85) node [right] {\small $(2,2)$};
\draw (0,0) node [below left] {\small $1$};
\draw (2,0) node [below right] {\small $2$};
\draw (1.1,1.73) node [right] {\small $3$};
\end{tikzpicture}
}
\overset {2}{\rightarrow}
\raisebox{-25pt}[45pt]
{
\begin{tikzpicture}[scale=0.5]
\draw (0,0) circle[radius=0.1];
\draw (2,0) circle[radius=0.1];
\draw (1,1.73) circle[radius=0.1];
\draw[{Latex[length=8pt,width=4pt]}-] (0.2,0)--(1.8,0);
\draw[{Latex[length=8pt,width=4pt]}-] (0.1,0.2)--(0.9,1.53);
\draw[-{Latex[length=8pt,width=4pt]}] (1.9,0.2)--(1.1,1.53);
\draw (1,0) node [below] {\small $(3,2)$};
\draw (0.4,0.85) node [left] {\small $(3,2)$};
\draw (1.6,0.85) node [right] {\small $(2,2)$};
\draw (0,0) node [below left] {\small $1$};
\draw (2,0) node [below right] {\small $2$};
\draw (1.1,1.73) node [right] {\small $3$};
\end{tikzpicture}
}
\end{align}
We may think it as a parallel process to \eqref{eq:v1}.
The boundary of the tunnel becomes fractal if we explore further details.

Generally, the tunnel's entrance's location varies and could be deeper in the fractal structure.
Another example in Figure \ref{fig:type6-a-closeup}
shows that the tunnel still has the same structure
regardless of the location as follows:
\vskip-35pt
\begin{align}
\label{eq:tunnel2}
\raisebox{-25pt}[45pt]
{
\begin{tikzpicture}[scale=0.5]
\draw (0,0) circle[radius=0.1];
\draw (2,0) circle[radius=0.1];
\draw (1,1.73) circle[radius=0.1];
\draw[{Latex[length=8pt,width=4pt]}-] (0.2,0)--(1.8,0);
\draw[-{Latex[length=8pt,width=4pt]}] (0.1,0.2)--(0.9,1.53);
\draw[-{Latex[length=8pt,width=4pt]}] (1.9,0.2)--(1.1,1.53);
\draw (1,0) node [below] {\small $(3,2)$};
\draw (0.4,0.85) node [left] {\small $(2,3)$};
\draw (1.6,0.85) node [right] {\small $(3,3)$};
\draw (0,0) node [below left] {\small $1$};
\draw (2,0) node [below right] {\small $2$};
\draw (1.1,1.73) node [right] {\small $3$};
\end{tikzpicture}
}
\overset {3}{\rightarrow}
\raisebox{-25pt}[45pt]
{
\begin{tikzpicture}[scale=0.5]
\draw (0,0) circle[radius=0.1];
\draw (2,0) circle[radius=0.1];
\draw (1,1.73) circle[radius=0.1];
\draw[{Latex[length=8pt,width=4pt]}-] (0.2,0)--(1.8,0);
\draw[{Latex[length=8pt,width=4pt]}-] (0.1,0.2)--(0.9,1.53);
\draw[{Latex[length=8pt,width=4pt]}-] (1.9,0.2)--(1.1,1.53);
\draw (1,0) node [below] {\small $(3,2)$};
\draw (0.4,0.85) node [left] {\small $(3,2)$};
\draw (1.6,0.85) node [right] {\small $(3,3)$};
\draw (0,0) node [below left] {\small $1$};
\draw (2,0) node [below right] {\small $2$};
\draw (1.1,1.73) node [right] {\small $3$};
\end{tikzpicture}
}
\overset {1}{\rightarrow}
\raisebox{-25pt}[45pt]
{
\begin{tikzpicture}[scale=0.5]
\draw (0,0) circle[radius=0.1];
\draw (2,0) circle[radius=0.1];
\draw (1,1.73) circle[radius=0.1];
\draw[-{Latex[length=8pt,width=4pt]}] (0.2,0)--(1.8,0);
\draw[-{Latex[length=8pt,width=4pt]}] (0.1,0.2)--(0.9,1.53);
\draw[{Latex[length=8pt,width=4pt]}-] (1.9,0.2)--(1.1,1.53);
\draw (1,0) node [below] {\small $(2,3)$};
\draw (0.4,0.85) node [left] {\small $(2,3)$};
\draw (1.6,0.85) node [right] {\small $(3,3)$};
\draw (0,0) node [below left] {\small $1$};
\draw (2,0) node [below right] {\small $2$};
\draw (1.1,1.73) node [right] {\small $3$};
\end{tikzpicture}
}
\overset {2}{\rightarrow}
\raisebox{-25pt}[45pt]
{
\begin{tikzpicture}[scale=0.5]
\draw (0,0) circle[radius=0.1];
\draw (2,0) circle[radius=0.1];
\draw (1,1.73) circle[radius=0.1];
\draw[{Latex[length=8pt,width=4pt]}-] (0.2,0)--(1.8,0);
\draw[-{Latex[length=8pt,width=4pt]}] (0.1,0.2)--(0.9,1.53);
\draw[-{Latex[length=8pt,width=4pt]}] (1.9,0.2)--(1.1,1.53);
\draw (1,0) node [below] {\small $(3,2)$};
\draw (0.4,0.85) node [left] {\small $(2,3)$};
\draw (1.6,0.85) node [right] {\small $(3,3)$};
\draw (0,0) node [below left] {\small $1$};
\draw (2,0) node [below right] {\small $2$};
\draw (1.1,1.73) node [right] {\small $3$};
\end{tikzpicture}
}
\end{align}

\par
{\bf (C-3)}: (4-1,4-1,4-2) (the outside gate).
We give two examples in Figures \ref{fig:type4-2} and \ref{fig:type4-2-b}.
We may think that
the former is a generic one, while the latter is a finite degeneration.

\par
{\bf (C-4)}: (4-1,4-1,4-3) (the inside gate).
We give two examples in Figures \ref{fig:type4-3} and \ref{fig:type4-3-b}.
We may think that
the former is a generic one, while the latter is a finite degeneration.

\par
{\bf (C-5)}:  (4-2,4-1,4-3) (the dual gate).
 We first study the following simple situation.
Consider the skew-symmetrizable matrix
\begin{align}
\label{eq:Bcyclic2}
B=
\begin{pmatrix}
0 & - 2 & p_2
\\
3 & 0 & -3
\\
-p_2 & 2 & 0
\end{pmatrix}
\quad
(p_2 \geq 2).
\end{align}
Assume that the vertex $v_2$ is Type 4-1.
If $p_2\leq 4$, then $v_1$ and $v_3$ are also Type 4-1.
So, we dismiss the case.
The rest is separated into three cases.
\par
(a) If $p_2\geq 7$, we have
\begin{itemize}
\item
the vertex $v_{1}$ is of Type 4-2 with $N=0$,
where the equality in \eqref{eq:case42N} is not attained,
\item
the vertex $v_{3}$ is of Type 4-3 with $N=0$,
where the equality in \eqref{eq:case42N} is not attained.
\end{itemize}
\par
(b) If $p_2=6$, we have
\begin{itemize}
\item
the vertex $v_{1}$ is of Type 4-2 with $N=0$,
where the equality in \eqref{eq:case42N} is attained.
\item
the vertex $v_{3}$ is of Type 4-3 with $N=1$.
where the equality in \eqref{eq:case43N} is attained.
\end{itemize}
\par
(c) If $p_2=5$, we have
\begin{itemize}
\item
the vertex $v_{1}$ is of Type 4-2 with $N=1$,
where the equality in \eqref{eq:case42N} is attained.
\item
the vertex $v_{3}$ is of Type 4-3 with $N=2$,
where the equality in \eqref{eq:case43N} is attained.
\end{itemize}
The $G$-fans for $p_2=7$, 6, and 5 are presented in Figures \ref{fig:type5-a}--\ref{fig:type5-c}.
We may think that Figure \ref{fig:type5-a} is a generic one,
while Figures \ref{fig:type5-b} and  \ref{fig:type5-c} are its finite degenerations.

The above classification seems applicable to a general matrix $B$ in \eqref{eq:Bcyclic1}.
Assume that that only the vertex $v_2$ is Type 4-1.
Then, if $p'_2 > p_1p_3$, it belongs to Case (a).
(We may consider it as a generic one.)
 Similarly, if $p'_2 = p_1p_3$,
it belongs to Case (b). 
If $p'_2 < p_1p_3$,
we found only the following solution so far, which belongs to Case (c):
\begin{align}
\begin{pmatrix}
0 & - b & ab-1
\\
a & 0 & -a
\\
-ab+1 & b & 0
\end{pmatrix}
.
\end{align}

\begin{table}
\begin{tabular}{l|l|l}
\hline
local pattern & global pattern (Fig.) & finite degeneration (Fig.)
\\
\hline
(A)\ \ \, (1,2,3) & the wing (\ref{fig:type1}, \ref{fig:type2}, \ref{fig:type3})
& not presented
\\
(C-1) (4-1,4-1,4-1), $C(B)\leq 4$ & the pinwheel (\ref{fig:type4-1}) & none
\\
(C-2) (4-1,4-1,4-1), $C(B)> 4$ & the tunnel (\ref{fig:type6-a}, \ref{fig:type6-a-closeup})
& the wide tunnel (\ref{fig:type6-b})
\\
(C-3) (4-1,4-1,4-2) & the outside gate (\ref{fig:type4-2}) & the wide outside gate (\ref{fig:type4-2-b})
\\
(C-4) (4-1,4-1,4-3)  & the inside gate (\ref{fig:type4-3}) & the wide intside gate (\ref{fig:type4-3-b})
\\
(C-5) (4-2,4-1,4-3)  & the dual gates (\ref{fig:type5-a}) & the dual wide gates (\ref{fig:type5-b})
\\
  & & the skewed dual wide gates (\ref{fig:type5-c})
\\
\hline
\end{tabular}
\bigskip
\caption{Correspondence between local and global patterns.}
\label{tab:corres1}
\end{table}

The correspondence between the local and global
patterns is summarized in Table \ref{tab:corres1}.
We naively expect that they exhaust all global patterns of rank 3 $G$-fans up to rotations and reflections
and finite degenerations.
The more complete study will be left as a future problem.
\clearpage

\begin{figure}
\centering
\includegraphics[width=320pt]{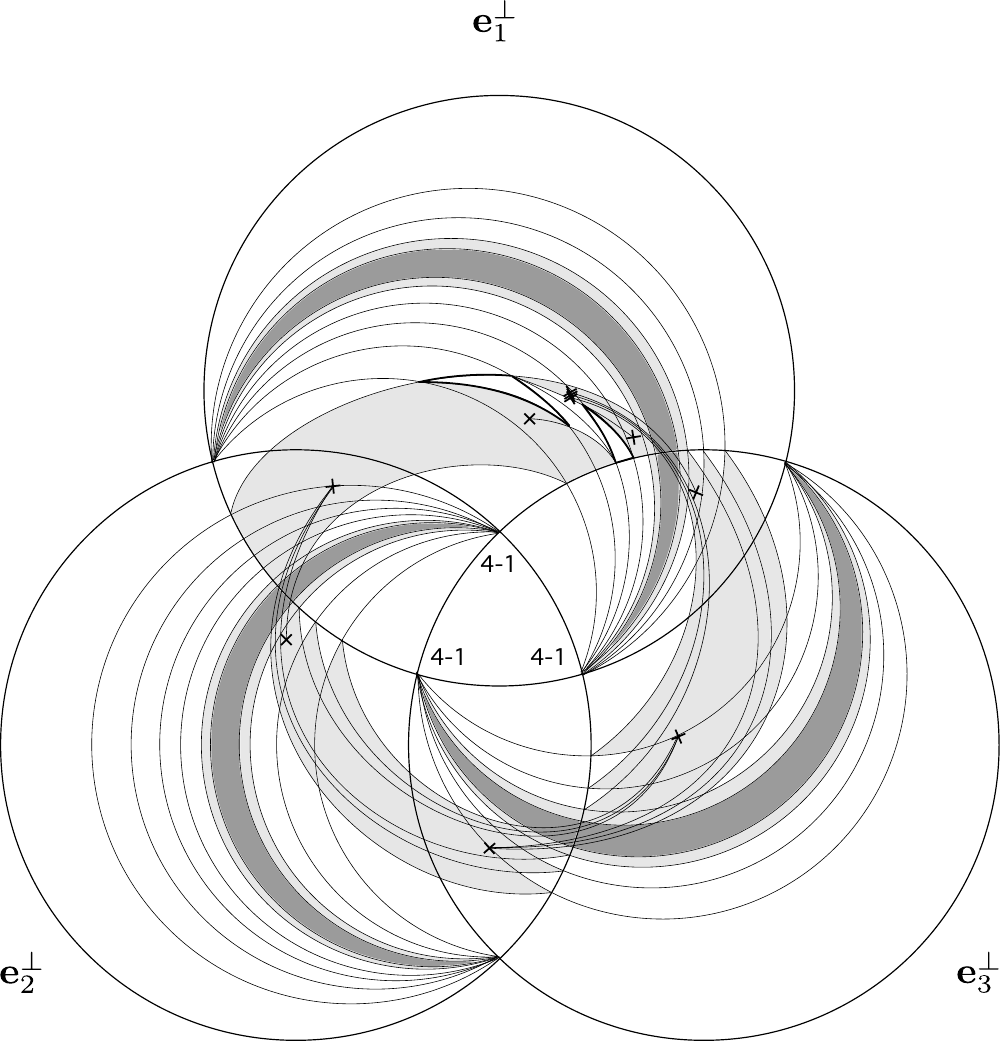}
\\
{
\small
$
B=
\begin{pmatrix}
0 & - 6 & 4886
\\
9 & 0 & -830
\\
-7329 & 830 & 0
\end{pmatrix}
$
}
\hskip30pt
\raisebox{-25pt}[45pt]
{
\begin{tikzpicture}[scale=0.8]
\draw (0,0) circle[radius=0.1];
\draw (2,0) circle[radius=0.1];
\draw (1,1.73) circle[radius=0.1];
\draw[{Latex[length=8pt,width=4pt]}-] (0.2,0)--(1.8,0);
\draw[-{Latex[length=8pt,width=4pt]}] (0.1,0.2)--(0.9,1.53);
\draw[{Latex[length=8pt,width=4pt]}-] (1.9,0.2)--(1.1,1.53);
\draw (1,0) node [below] {\small $(9,6)$};
\draw (0.4,0.85) node [left] {\small $(4886,7329)$};
\draw (1.6,0.85) node [right] {\small $(830,830)$};
\draw (0,0) node [below left] {\small $1$};
\draw (2,0) node [below right] {\small $2$};
\draw (1.1,1.73) node [right] {\small $3$};
\end{tikzpicture}
}
\caption[dummy]{
``The tunnel''.
An example with
$L=3$ and $C(B)=28$.
The framed triangle on the right side corresponds to the acyclic matrix obtained from $B$
by the mutation sequence $(2, 1, 2, 1, 3)$.
Along the tunnel,
it is mutated into the same matrix at the framed triangle on the left side by 
the mutation sequence $(1,3,2)$
in \eqref{eq:tunnel1}.
}
\label{fig:type6-a}
\end{figure}
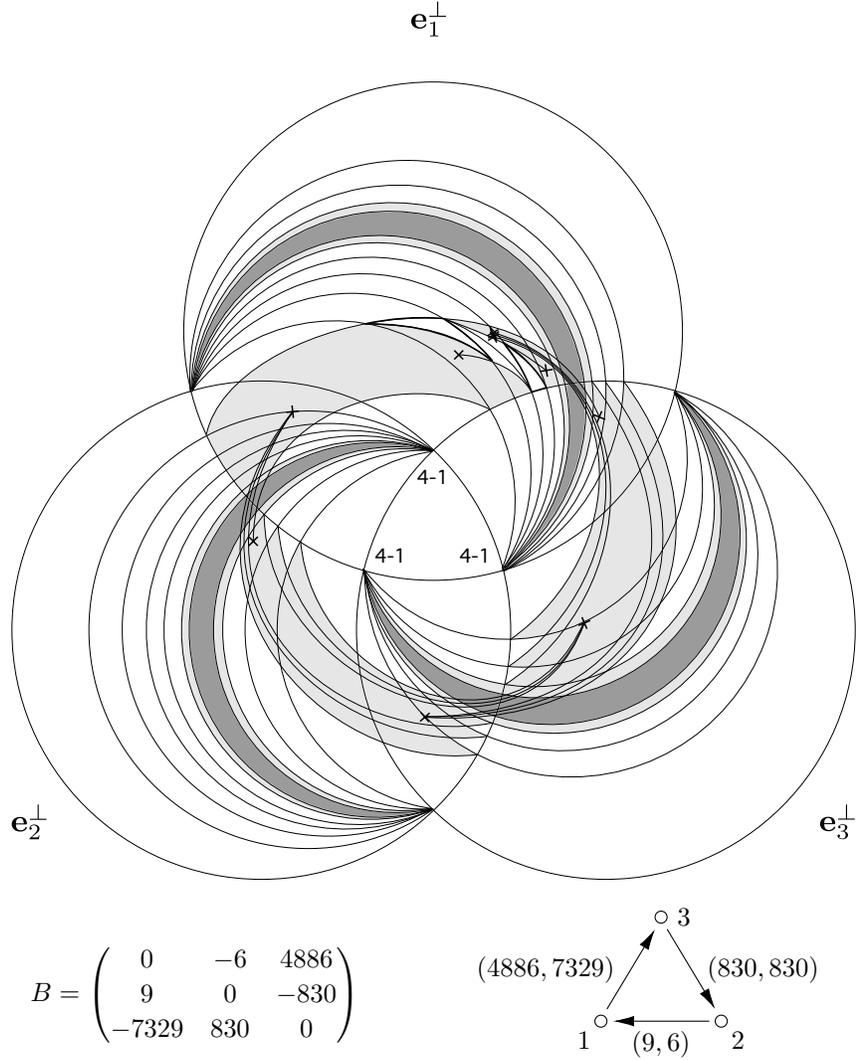

\begin{figure}
\centering
\includegraphics[width=320pt]{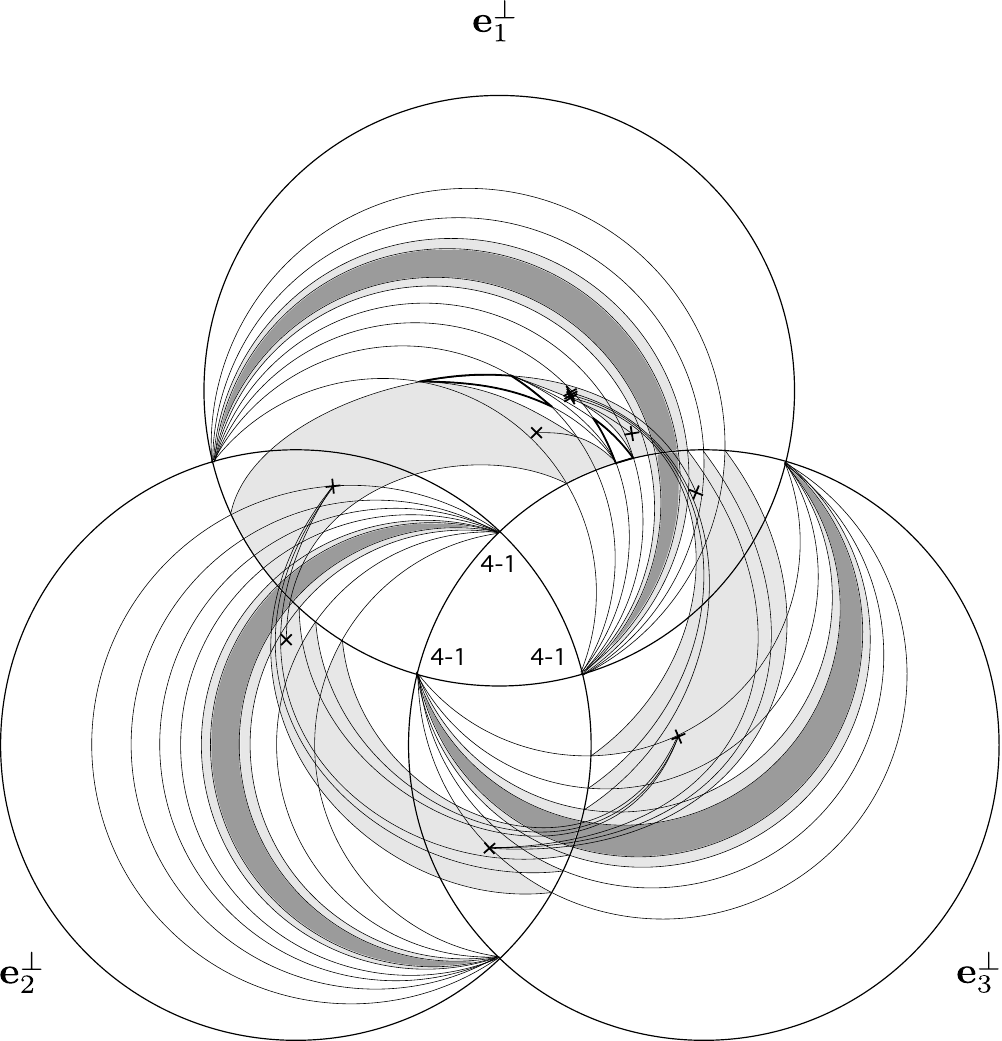}
\\
{
\small
$
B=
\begin{pmatrix}
0 & - 15 & 2013
\\
2 & 0 & -139
\\
-1342 & 695 & 0
\end{pmatrix}
$
}
\hskip30pt
\raisebox{-25pt}[45pt]
{
\begin{tikzpicture}[scale=0.8]
\draw (0,0) circle[radius=0.1];
\draw (2,0) circle[radius=0.1];
\draw (1,1.73) circle[radius=0.1];
\draw[{Latex[length=8pt,width=4pt]}-] (0.2,0)--(1.8,0);
\draw[-{Latex[length=8pt,width=4pt]}] (0.1,0.2)--(0.9,1.53);
\draw[{Latex[length=8pt,width=4pt]}-] (1.9,0.2)--(1.1,1.53);
\draw (1,0) node [below] {\small $(2,15)$};
\draw (0.4,0.85) node [left] {\small $(2013,1342)$};
\draw (1.6,0.85) node [right] {\small $(695,139)$};
\draw (0,0) node [below left] {\small $1$};
\draw (2,0) node [below right] {\small $2$};
\draw (1.1,1.73) node [right] {\small $3$};
\end{tikzpicture}
}
\caption{
``The wide tunnel''.
A variant of Figure \ref{fig:type6-a}.
An example with
$L=3$ and $C(B)=11$.
The pattern is the same as Figure \ref{fig:type6-a}
except for the details around the tunnel.
}
\label{fig:type6-b}
\end{figure}
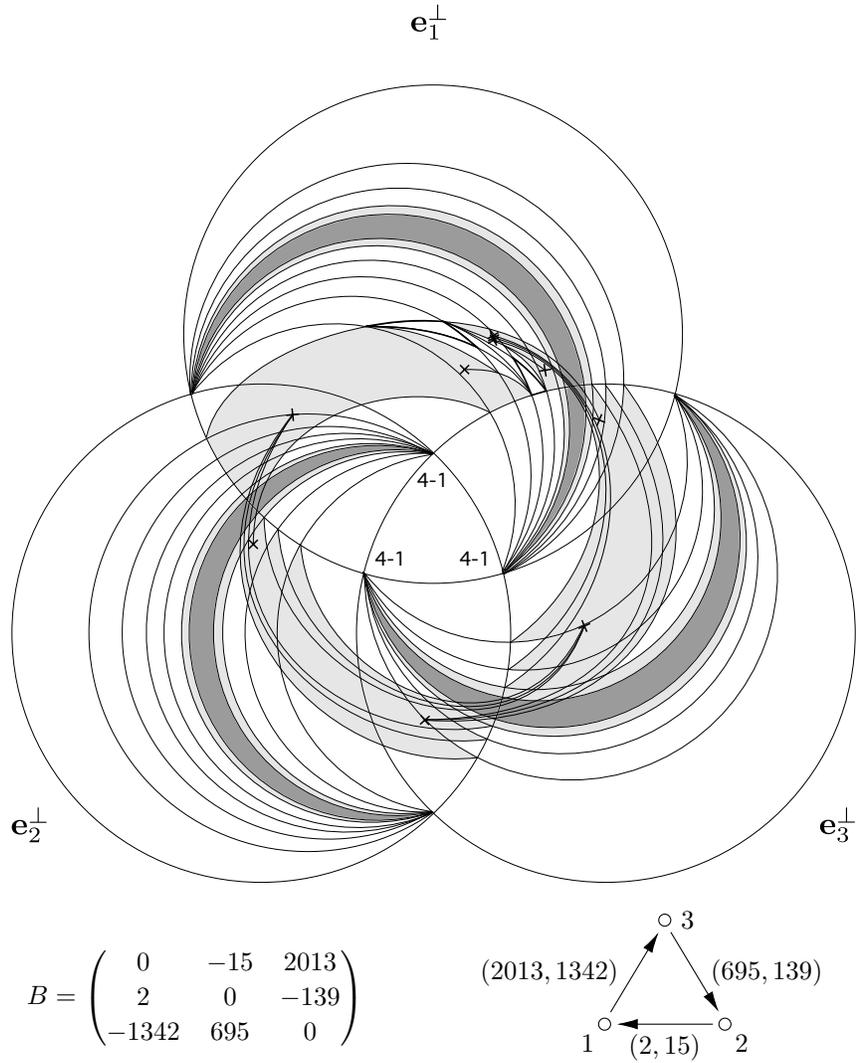

\begin{figure}
\centering
\includegraphics[width=320pt]{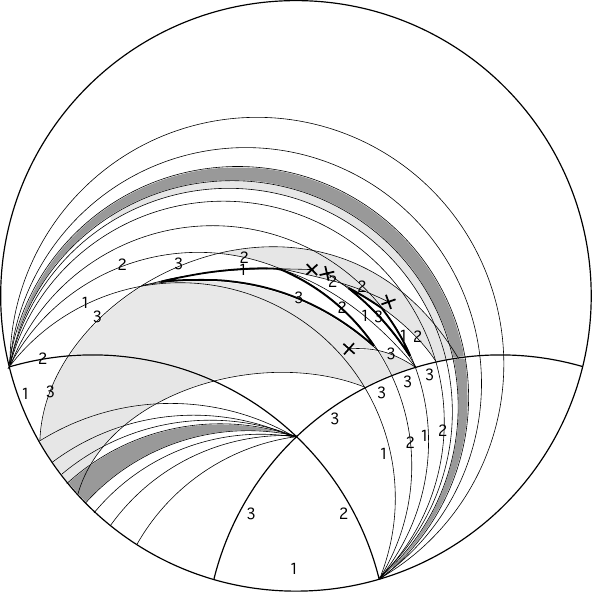}
\\
{
\small
$
B=
\begin{pmatrix}
0 & - 16 & 237602
\\
24 & 0 & -14889
\\
-356403 & 14889 & 0
\end{pmatrix}
$
}
\hskip30pt
\raisebox{-25pt}[45pt]
{
\begin{tikzpicture}[scale=0.8]
\draw (0,0) circle[radius=0.1];
\draw (2,0) circle[radius=0.1];
\draw (1,1.73) circle[radius=0.1];
\draw[{Latex[length=8pt,width=4pt]}-] (0.2,0)--(1.8,0);
\draw[-{Latex[length=8pt,width=4pt]}] (0.1,0.2)--(0.9,1.53);
\draw[{Latex[length=8pt,width=4pt]}-] (1.9,0.2)--(1.1,1.53);
\draw (1,0) node [below] {\small $(24,16)$};
\draw (0.4,0.85) node [left] {\small $(237602,356403)$};
\draw (1.6,0.85) node [right] {\small $(14889,14889)$};
\draw (0,0) node [below left] {\small $1$};
\draw (2,0) node [below right] {\small $2$};
\draw (1.1,1.73) node [right] {\small $3$};
\end{tikzpicture}
}
\caption{
Close-up of the tunnel for another variant of Figure \ref{fig:type6-a} with $C(B)=39$.
The attached numbers are indices of edges.
The framed triangle on the right side corresponds to the acyclic matrix
$
\begin{pmatrix}
0 & - 2 & 2
\\
3 & 0 & 3
\\
-3 & -3 & 0
\end{pmatrix}
$
obtained from $B$
by the mutation sequence $(2, 1, 2, 1, 3, 1)$.
Along the tunnel,
it is mutated into the same matrix at the framed triangle on the left side by 
the mutation sequence $(3,1,2)$ in in \eqref{eq:tunnel2}.
}
\label{fig:type6-a-closeup}
\end{figure}
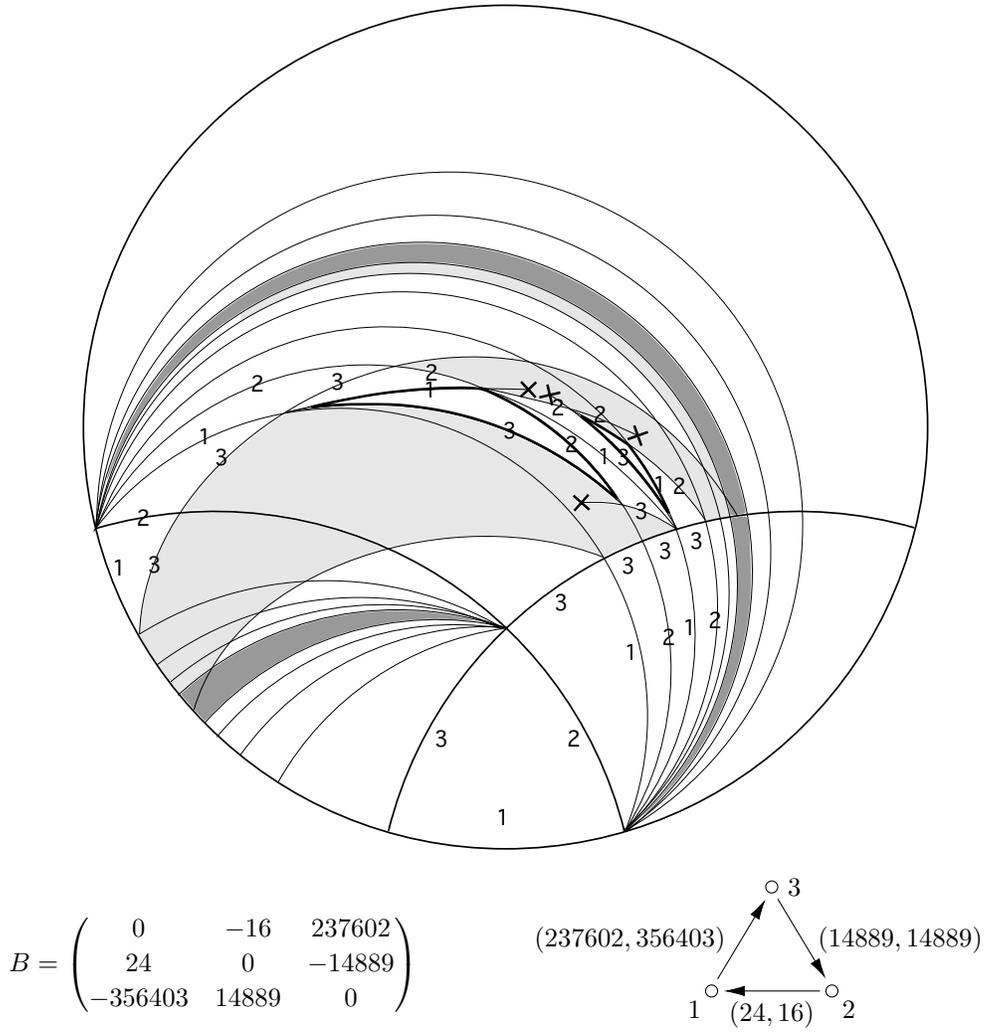

\begin{figure}
\centering
\includegraphics[width=320pt]{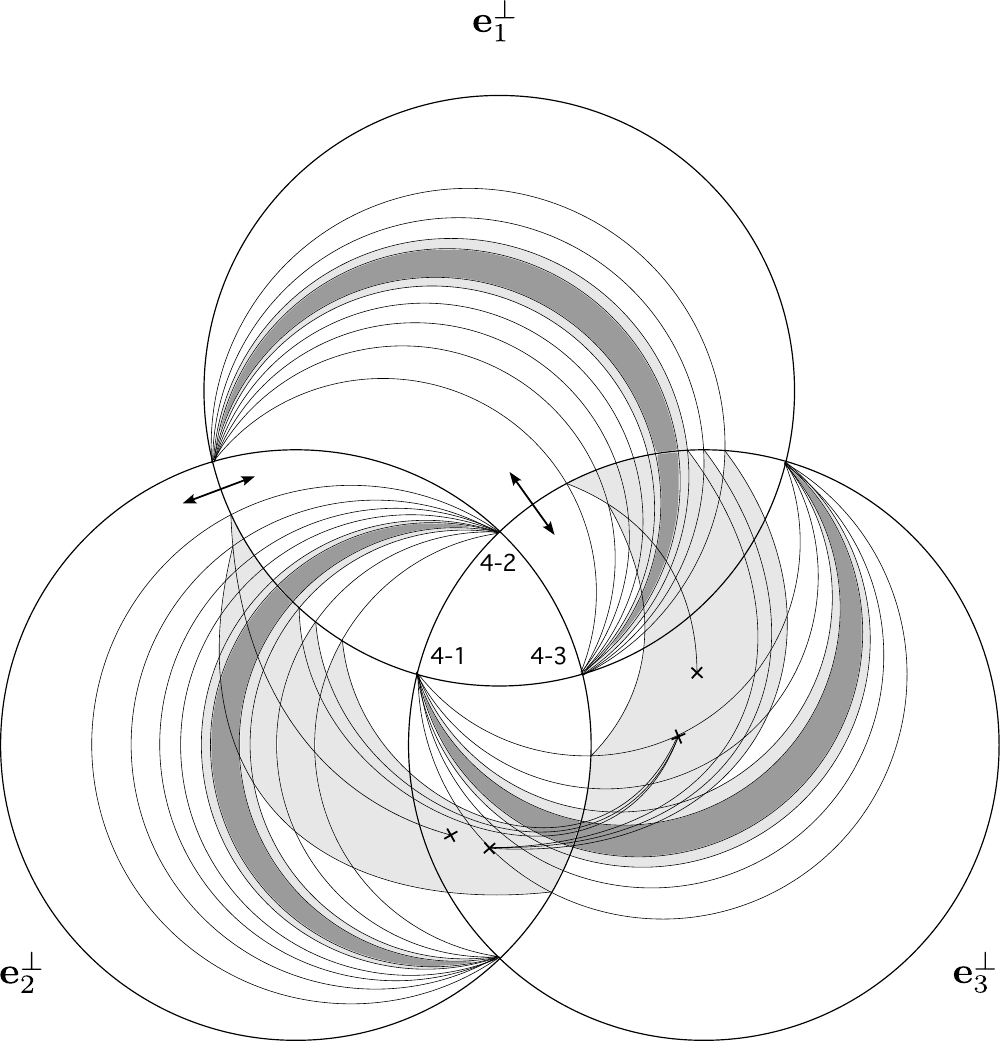}
\\
{
\small
$
B=
\begin{pmatrix}
0 & - 2 & 7
\\
3 & 0 & -3
\\
-7 & 2 & 0
\end{pmatrix}
$
}
\hskip30pt
\raisebox{-25pt}[45pt]
{
\begin{tikzpicture}[scale=0.8]
\draw (0,0) circle[radius=0.1];
\draw (2,0) circle[radius=0.1];
\draw (1,1.73) circle[radius=0.1];
\draw[{Latex[length=8pt,width=4pt]}-] (0.2,0)--(1.8,0);
\draw[-{Latex[length=8pt,width=4pt]}] (0.1,0.2)--(0.9,1.53);
\draw[{Latex[length=8pt,width=4pt]}-] (1.9,0.2)--(1.1,1.53);
\draw (1,0) node [below] {\small $(3,2)$};
\draw (0.4,0.85) node [left] {\small $7$};
\draw (1.6,0.85) node [right] {\small $(2,3)$};
\draw (0,0) node [below left] {\small $1$};
\draw (2,0) node [below right] {\small $2$};
\draw (1.1,1.73) node [right] {\small $3$};
\end{tikzpicture}
}
\caption{
``The dual gates''.
The $G$-fan for the matrix \eqref{eq:Bcyclic2} with $p_2=7$.
This is a hybrid of Figures \ref{fig:type4-2} and \ref{fig:type4-3}.
}
\label{fig:type5-a}
\end{figure}

\begin{figure}
\centering
\includegraphics[width=320pt]{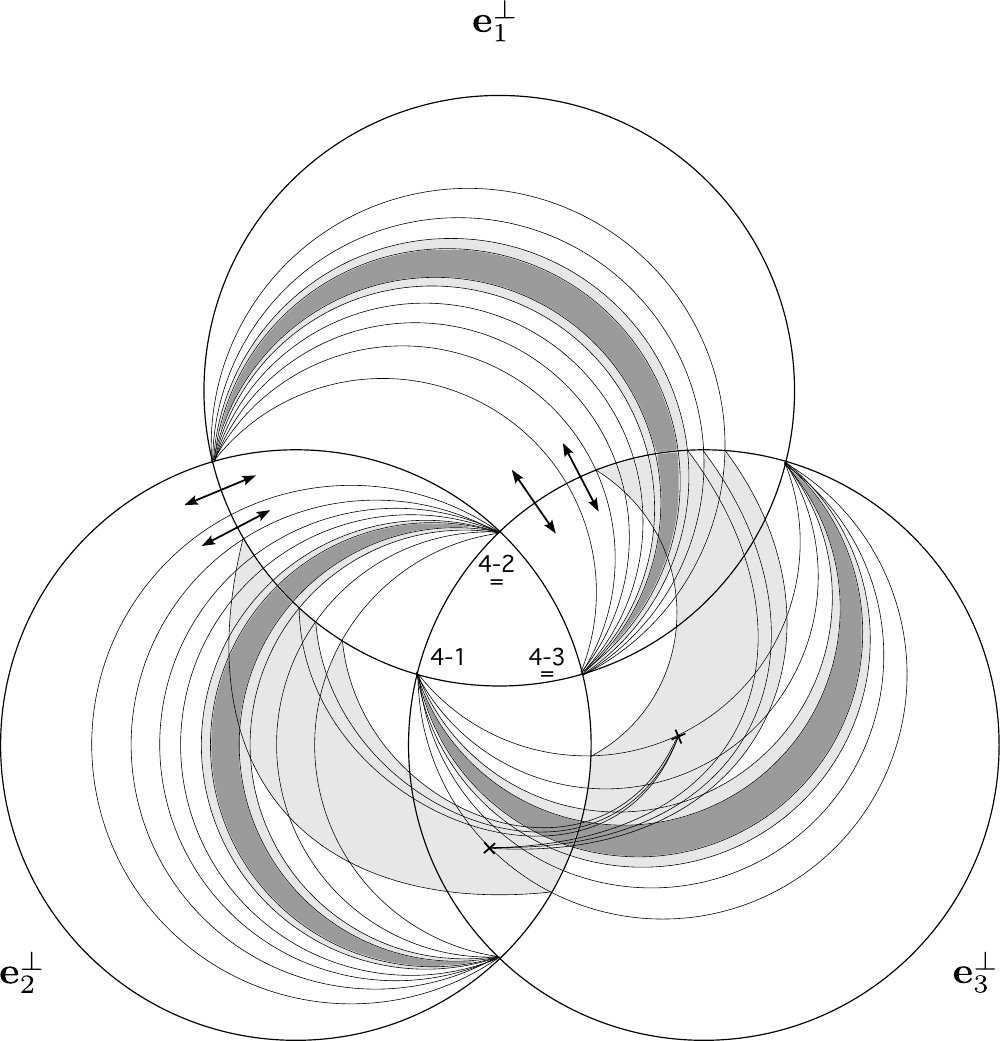}
\\
{
\small
$
B=
\begin{pmatrix}
0 & - 2 & 6
\\
3 & 0 & -3
\\
-6 & 2 & 0
\end{pmatrix}
$
}
\hskip30pt
\raisebox{-25pt}[45pt]
{
\begin{tikzpicture}[scale=0.8]
\draw (0,0) circle[radius=0.1];
\draw (2,0) circle[radius=0.1];
\draw (1,1.73) circle[radius=0.1];
\draw[{Latex[length=8pt,width=4pt]}-] (0.2,0)--(1.8,0);
\draw[-{Latex[length=8pt,width=4pt]}] (0.1,0.2)--(0.9,1.53);
\draw[{Latex[length=8pt,width=4pt]}-] (1.9,0.2)--(1.1,1.53);
\draw (1,0) node [below] {\small $(3,2)$};
\draw (0.4,0.85) node [left] {\small $6$};
\draw (1.6,0.85) node [right] {\small $(2,3)$};
\draw (0,0) node [below left] {\small $1$};
\draw (2,0) node [below right] {\small $2$};
\draw (1.1,1.73) node [right] {\small $3$};
\end{tikzpicture}
}
\caption{
``The dual wide gates''.
The $G$-fan for the matrix \eqref{eq:Bcyclic2} with $p_2=6$.
This is a hybrid of Figures \ref{fig:type4-2-b} and \ref{fig:type4-3-b}.
}
\label{fig:type5-b}
\end{figure}

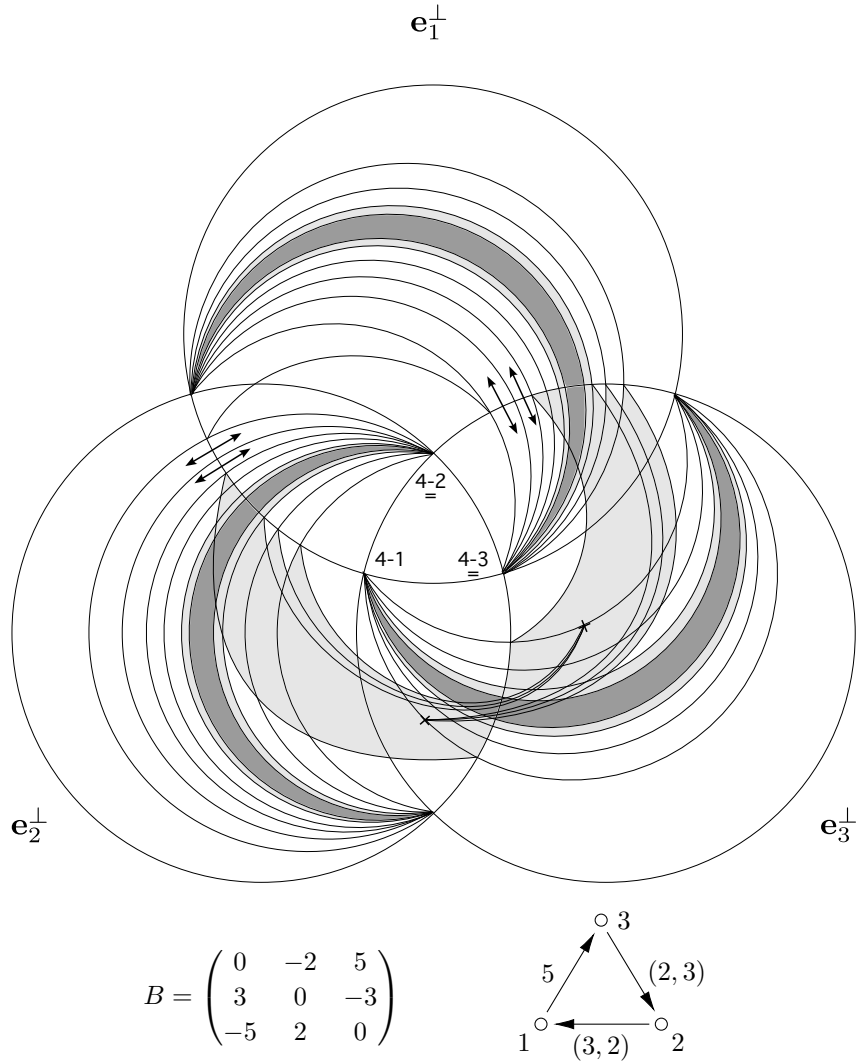
\begin{figure}
\centering
\includegraphics[width=320pt]{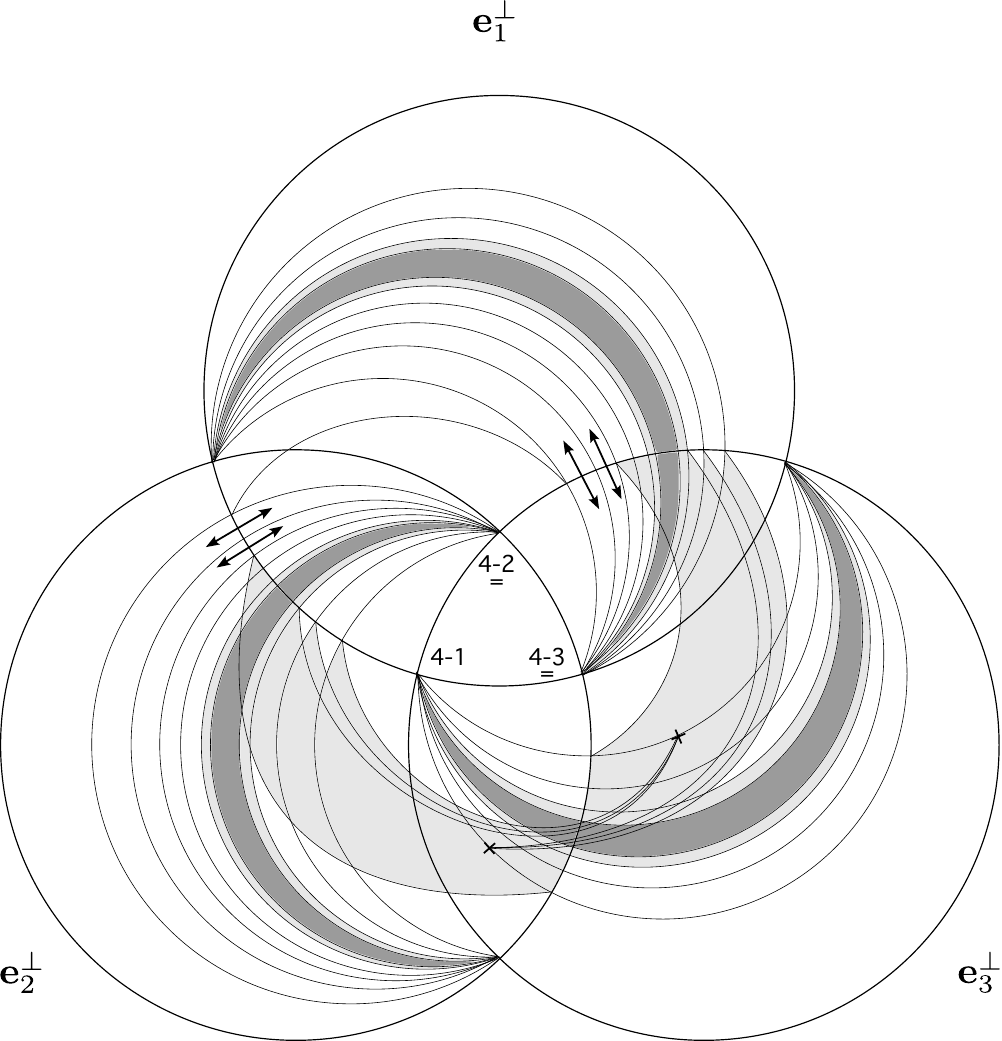}
\\
{
\small
$
B=
\begin{pmatrix}
0 & - 2 & 5
\\
3 & 0 & -3
\\
-5 & 2 & 0
\end{pmatrix}
$
}
\hskip30pt
\raisebox{-25pt}[45pt]
{
\begin{tikzpicture}[scale=0.8]
\draw (0,0) circle[radius=0.1];
\draw (2,0) circle[radius=0.1];
\draw (1,1.73) circle[radius=0.1];
\draw[{Latex[length=8pt,width=4pt]}-] (0.2,0)--(1.8,0);
\draw[-{Latex[length=8pt,width=4pt]}] (0.1,0.2)--(0.9,1.53);
\draw[{Latex[length=8pt,width=4pt]}-] (1.9,0.2)--(1.1,1.53);
\draw (1,0) node [below] {\small $(3,2)$};
\draw (0.4,0.85) node [left] {\small $5$};
\draw (1.6,0.85) node [right] {\small $(2,3)$};
\draw (0,0) node [below left] {\small $1$};
\draw (2,0) node [below right] {\small $2$};
\draw (1.1,1.73) node [right] {\small $3$};
\end{tikzpicture}
}
\caption{
``The skewed dual wide gates''.
The $G$-fan  for the matrix \eqref{eq:Bcyclic2} with $p_2=5$.
This pattern appeared in \cite[Figure 18]{Muller15}.
}
\label{fig:type5-c}
\end{figure}

\clearpage

\bibliography{../../biblist/biblist.bib}
\end{document}